%% file: arxiv.tex
\documentclass[onefignum,onetabnum]{siamart250211}

\usepackage{fixltx2e}
\usepackage{amssymb,amsmath}
\usepackage{bbm}
\usepackage{booktabs}
\usepackage{cases,enumitem}
\usepackage[dvipsnames]{xcolor}
\usepackage{enumerate}
\usepackage{float}
\usepackage{graphicx}
\usepackage{hyperref}
\usepackage[numbers,sort]{natbib}
\usepackage{subfig}
\usepackage{mathabx}
\usepackage[normalem]{ulem}
\usepackage{tikz}
\usepackage{pgfplots}
\pgfplotsset{compat=1.18}
\usepackage[algo2e,vlined,noend]{algorithm2e}
\SetKwProg{Fn}{}{}{}
\SetKw{Return}{return}
\SetKw{Break}{break}
\SetKw{Not}{not}
\SetKw{Or}{or}
\SetCommentSty{texttt}
\SetAlFnt{\small}
\SetAlCapFnt{\small}
\SetAlCapNameFnt{\small}
\DontPrintSemicolon
\newcommand{\Hdr}[1]{\textnormal{\textbf{#1}}}
\usetikzlibrary{arrows.meta}
\hypersetup{
    colorlinks = true,
    urlcolor   = blue,
    citecolor  = black,
}

\newsiamremark{remark}{Remark}

\newcommand{\mH}{\mathcal{H}}

\undef\div
\DeclareMathOperator*{\div}{div}

\newcommand{\ph}{p_{\mathrm{h}}}
\newcommand{\barph}{\bar p_{\mathrm{h}}}
\newcommand{\underph}{\underbar p_{\mathrm{h}}}
\newcommand{\overbar}[1]{\mkern1.5mu\overline{\mkern-1.5mu#1\mkern-1.5mu}\mkern1.5mu}
\renewcommand{\underbar}[1]{\mkern1.5mu\underline{\mkern-1.5mu#1\mkern-1.5mu}\mkern1.5mu}
\newcommand{\pb}{p_{\mathrm{b}}}
\newcommand{\eh}{e_{\mathrm{h}}}
\newcommand{\eb}{e_{\mathrm{b}}}
\newcommand{\ah}{a_{\mathrm{h}}}
\newcommand{\ab}{a_{\mathrm{b}}}
\newcommand{\sh}{s_{\mathrm{h}}}

\newcommand{\hatph}{\hat p_{\mathrm{h}}}
\newcommand{\hatpb}{\hat p_{\mathrm{b}}}
\newcommand{\hateh}{\hat e_{\mathrm{h}}}
\newcommand{\hateb}{\hat e_{\mathrm{b}}}
\newcommand{\hatah}{\hat a_{\mathrm{h}}}
\newcommand{\hatab}{\hat a_{\mathrm{b}}}
\newcommand{\hatAb}{\hat A_{\mathrm{b}}}
\newcommand{\hatsh}{\hat s_{\mathrm{h}}}

\renewcommand{\L}{^{\mathrm L}}
\newcommand{\R}{^{\mathrm R}}
\renewcommand{\S}{^{\mathrm S}}
\renewcommand{\star}{^{\ast}}
\newcommand{\starL}{^{\ast\,\mathrm L}}
\newcommand{\starR}{^{\ast\,\mathrm R}}
\newcommand{\starRR}{^{\ast,\,\mathrm{RR}}}

\begin{document}

\title{Guaranteed wave-speed bounds for the compressible Euler
  equations with a composite equation of state%
  \thanks{Draft version of \today}}

\author{
  Nicolas Favrie\footnotemark[2]
  \and
  Matthias Maier\footnotemark[3]}

\maketitle

\renewcommand{\thefootnote}{\fnsymbol{footnote}}
\footnotetext[2]{%
  Aix Marseille Université, CNRS, IUSTI, Marseille, France.
  \url{nicolas.favrie@univ-amu.fr}}
\footnotetext[3]{%
  Dept. of Mathematics, Texas A\&M University, College Station, TX, USA. \url{maier@tamu.edu}}
\renewcommand{\thefootnote}{\arabic{footnote}}

\tableofcontents

\begin{abstract}
  We derive computable upper bounds on the maximum wave speed in the
  Riemann problem for the compressible Euler equations with a composite
  equation of state, in which the pressure is the sum of a convex
  hydrodynamical part and a barotropic correction. The correction may
  destroy the convexity of the equation of state, so that the Riemann
  solution can contain composite shock--rarefaction waves. The bounds are
  obtained by comparison with the auxiliary Riemann problem for the
  hydrodynamical equation of state alone. The bounds hold irrespective of
  whether the outer waves are elementary or composite and require no case
  distinction on the wave type.
\end{abstract}

\begin{keywords}
  Riemann problem, compressible Euler equations, composite equation of
  state, nonconvex equation of state, maximum wave speed, guaranteed upper
  bounds, Hugoniot curves, composite waves, invariant-domain-preserving
  schemes
\end{keywords}

\begin{AMS}
  35L65, 35L67, 35Q31, 76N15, 76L05, 65M08
\end{AMS}

\pagestyle{myheadings}
\thispagestyle{plain}
\markboth{Favrie \& Maier}{Guaranteed wave-speed bounds for composite equations of state}

\section{Introduction}
This paper is concerned with guaranteed and inexpensive estimates of the
maximum wave speed in the Riemann problem for the compressible Euler
equations with a composite equation of state. The pressure is the sum of a
hydrodynamical contribution with the classical thermodynamic convexity
properties and a barotropic contribution that depends on the density alone.
{Additive splittings of this kind occur, for instance, in
Mie--Gr\"uneisen-type closures, which combine a cold pressure with a
thermal part, and in separable energies of hyperelastic solids.} The
barotropic contribution may destroy the convexity of the total equation of
state. The Riemann solution may then contain composite shock--rarefaction
waves. Our goal is to bound the extremal speeds of the Riemann fan by
quantities computed from the auxiliary Riemann problem for the
hydrodynamical equation of state alone, without constructing the wave fan
of the full problem. We first recall why such estimates are needed and what
is available in the literature, and then outline our approach and the main
results.

\subsection{Riemann problems and wave-speed estimates}

Many numerical methods for hyperbolic systems of conservation laws require
an estimate of the largest local propagation speed. Such estimates enter,
for instance, the construction of HLL-type and local Lax--Friedrichs
numerical fluxes, the definition of artificial- or graph-viscosity
coefficients, and the evaluation of stable time steps through a CFL
condition. In these applications, the estimate must be large enough to
bound the speed of every physically relevant wave. An underestimated speed
may produce insufficient numerical dissipation and compromise stability,
positivity, or invariant-domain preservation. Conversely, an excessive
overestimate increases numerical diffusion and may impose an unnecessarily
restrictive time step.

This requirement is particularly explicit in Godunov-type methods. At
each cell interface, the left and right reconstructed states define a local
initial-value problem whose solution determines how information propagates
between neighboring cells \cite{Godunov1959,Toro2009}. Exact Godunov
schemes use the complete local solution, whereas approximate Riemann
solvers retain only part of its wave structure. For example, HLL-type
solvers require estimates of the leftmost and rightmost signal velocities,
while local Lax--Friedrichs and graph-viscosity methods use a bound on the
maximum absolute propagation speed
\cite{HartenLaxVanLeer1983,Davis1988,
GuermondPopov2016InvariantDomains}. The relevant quantity is therefore not
merely the spectral radius evaluated at the initial states, but the maximum
speed attained throughout the nonlinear wave pattern generated at the
interface.

The local problem underlying these estimates is the Riemann problem.
For a one-dimensional hyperbolic system of conservation laws,
\begin{equation}
  \label{eq:intro_conservation_law}
  \partial_t \boldsymbol{U} + \partial_x \boldsymbol{F}(\boldsymbol{U})  =  0,
\end{equation}
it is defined by the piecewise constant initial data
\begin{equation}
  \label{eq:intro_riemann_data}
      \boldsymbol{U}(x,0) =
      \begin{cases}
    \boldsymbol{U}_L, & x<0,\\
    \boldsymbol{U}_R, & x>0.
  \end{cases}
\end{equation}
The problem is invariant under the scaling $(x,t)\mapsto(\alpha x,\alpha
t)$, $\alpha>0$, and its solution is therefore self-similar,
$\boldsymbol{U}(x,t)=\boldsymbol{W}(x/t)$. It consists of a finite
collection of waves connecting the two initial states. Depending on the
system and on its constitutive relations, these waves may include shocks,
rarefaction waves, contact discontinuities, or composite wave patterns. The
extremal propagation speeds are the speeds of the leftmost and rightmost
edges of this wave fan.

Riemann problems arise in a wide range of hyperbolic continuum models,
including gas dynamics, multiphase flows, magnetohydrodynamics, and
hyperelasticity. General treatments can be found in
\cite{KulikovskiiPogorelovSemenov2000,Toro2009}, while Riemann problems and
nonlinear wave structures in elastic media are discussed extensively in
\cite{kulikovskii2021nonlinear}. In the present work, attention is focused
on the compressible Euler equations with a composite equation of state. The
difficulty is then to obtain a computationally inexpensive and guaranteed
estimate of the extremal speeds without constructing the complete Riemann
solution.

\subsection{Riemann problems for general and nonconvex equations of state}

For the compressible Euler equations equipped with a thermodynamically
admissible and convex equation of state, the Riemann solution has the
classical three-wave structure. The left and right initial states are
connected to two intermediate states by acoustic waves, each of which is
either a shock or a rarefaction, while the intermediate states are
separated by a contact discontinuity. The velocity and pressure are
continuous across the contact wave. This structure is well understood for
standard closures such as the ideal-gas, stiffened-gas, and Noble--Abel
stiffened-gas equations of state
\cite{Smith1979,KulikovskiiPogorelovSemenov2000,Toro2009}.

Even in this classical setting, the exact solution generally requires the
determination of an intermediate pressure through a nonlinear scalar
equation. For more general equations of state, including Mie--Gr\"uneisen,
JWL, real-fluid, and tabulated models, the construction becomes
substantially more involved. Shock curves require the solution of the
Rankine--Hugoniot energy relation together with repeated thermodynamic
inversions, whereas rarefaction curves generally require numerical
integration along isentropes. The complete Riemann solver may therefore
combine nonlinear root finding, numerical quadrature, and repeated
evaluations of the equation of state and its derivatives
\cite{MenikoffPlohr1989}.

The difficulty is even more pronounced in nonlinear solid mechanics. In
hyperelasticity, the analytical construction of a self-similar piston
solution is already highly involved when the material is initially at rest
and the loading preserves a sufficiently simple symmetry
\cite{NdanouFavrieGavrilyuk2017}. Even under these restrictive assumptions,
the solution requires the detailed analysis of several characteristic
families and may involve nonclassical transverse waves. The problem becomes
considerably more difficult when a nonzero initial deformation or velocity
field is prescribed, since the characteristic structure and the matching
conditions between the different waves can no longer be reduced in the same
way. Complete Riemann solutions may thus become analytically intractable,
even when highly structured configurations can still be solved exactly and
serve as reference solutions.

The classical wave structure also relies on suitable convexity properties
of the constitutive law. In gas dynamics, these properties may be expressed
in terms of the curvature of the pressure with respect to the specific
volume along isentropes. When this curvature changes sign, the acoustic
characteristic fields may lose genuine nonlinearity. A single
characteristic family may then contain a composite wave formed by
successive shock and rarefaction segments. The corresponding transition
states must satisfy additional tangency and admissibility conditions, and
the number and ordering of the elementary segments may depend on the
initial data
\cite{MenikoffPlohr1989,MullerVoss2006,KulikovskiiPogorelovSemenov2000}.

For general or nonconvex equations of state, the computational difficulty
is therefore not limited to determining the intermediate pressure. One may
also need to identify changes in convexity, locate transition states, and
construct the complete sequence of shock and rarefaction segments. This
makes an exact Riemann solver difficult to use systematically at every
interface of a multidimensional computation. It also complicates the
estimation of the extremal propagation speeds, since the fastest wave may
depend on an intermediate state that is not known without first resolving a
substantial part of the complete wave fan.

\subsection{Existing wave-speed estimates and their limitations}

The difficulty of estimating extremal wave speeds becomes particularly
apparent when composite waves are present. In one-dimensional nonlinear
elastodynamics, analytical solutions involving shocks, rarefactions, and
composite-wave structures have been constructed for general constitutive
laws in
\cite{BerjaminLombardChiavassaFavrie2017}. These solutions show that the
extremal propagation speed may depend on intermediate transition states
whose determination requires the complete construction of the composite
wave. Although such analytical solutions are valuable as references,
this level of resolution of the wave structure is generally too costly
to be used systematically within a numerical scheme.
{The same applies to exact solvers for nonclassical
Riemann problems in gas dynamics: the automatic wave-curve detection of
Sirianni, Guardone, and Re~\cite{SirianniGuardoneRe2025} is intended for
reference solutions rather than for flux evaluation. Numerical fluxes that
resolve composite waves, such as the flux formula of Marquina, Serna, and
Ib\'a\~nez~\cite{MarquinaSernaIbanez2019} for nonconvex relativistic
hydrodynamics, do not come with a bound on the propagation speeds.}

To avoid the cost of exact Riemann solvers, numerous approximate solvers
have been developed, including acoustic, Roe-type, HLL-type, local
Lax--Friedrichs, and AUSM-type methods; see, for instance,
\cite{Toro2009}. Most of these approaches replace the exact wave fan by a
reduced representation and therefore require estimates of the leftmost and
rightmost propagation speeds. Classical formulas, such as those proposed
by Davis \cite{Davis1988}, are inexpensive and widely used in practical
computations. However, estimates based only on the initial states or on a
linearized system are not necessarily rigorous bounds for the extremal
speeds of the nonlinear Riemann solution. The objective is
therefore not only to obtain an inexpensive estimate, but also to ensure
that it is a guaranteed and sufficiently sharp bound for the complete
Riemann fan.

For the Euler equations with ideal-gas and covolume equations of state,
Guermond and Popov~\cite{GuermondPopov2016WaveSpeed} derived an efficiently
computable, iterative upper estimate for the maximum wave speed. Direct
theoretical estimates for several hyperbolic systems were also investigated
by Toro, M{\"u}ller, and Siviglia \cite{ToroMullerSiviglia2020}, who showed
that some commonly used formulas do not systematically bound the exact wave
speeds. {Toro and Tokareva~\cite{ToroTokareva2026}
further showed that underestimated wave speeds in Rusanov-type schemes can
destroy monotonicity and reduce the stability limit, whereas overestimates
preserve monotonicity.} Guaranteed upper bounds are particularly important in
invariant-domain-preserving methods, where the graph-viscosity coefficient
is commonly chosen as an upper bound on the maximum speed of an associated
one-dimensional Riemann problem \cite{GuermondPopov2016InvariantDomains}.

Recent works have extended these ideas to arbitrary analytical or tabulated
equations of state by constructing local thermodynamic surrogate models
with suitable stability and convexity properties
\cite{ClaytonGuermondPopov2022,ClaytonEtAl2023}. These approaches avoid the
direct construction of the exact Riemann solution for the original equation
of state. The resulting estimates are guaranteed with respect to the
surrogate model rather than with respect to the original closure: they
bound the wave speeds of the surrogate Riemann problem and, by
construction, ensure that the update stays in the admissible set of the
original equation of state. They imply, however, neither a bound on the
extremal speeds of the exact Riemann fan of the original closure, nor an
entropy inequality for its entropy. Partial remedies are available; for
instance, Clayton and Tovar~\cite{ClaytonTovar2026} recover the minimum
principle on the physical entropy of the original equation of state by
introducing suitable discrete auxiliary states. Guermond and
Tovar~\cite{GuermondTovar2026} apply the same surrogate strategy to
radiation hydrodynamics with an equation of state that contains an additive
cold curve, and the resulting bound again holds for the surrogate only.

\subsection{Present approach and main contributions}

In the present work, we consider the compressible Euler equations closed
by a composite equation of state of the form
\begin{align*}
  p\;=\;\ph(\rho,\eh)\,+\,\pb(\rho),
  \qquad\quad
  \mathcal{E}\;=\;\frac{1}{2}\rho\,u^2\,+\,\eh\,+\,\eb(\rho),
\end{align*}
where the barotropic pressure and energy are related through
\begin{align*}
  \pb(\rho)\;:=\;\rho^2\,\partial_\rho\eb(\rho).
\end{align*}
The hydrodynamical contribution $\ph$ is assumed to satisfy the standard
thermodynamic stability and convexity conditions ensuring that its
associated Riemann problem is well posed and possesses the classical
shock--rarefaction structure. The barotropic contribution is assumed to
preserve hyperbolicity, but its curvature may modify the convexity of the
total isentropic pressure and may generate composite waves.
The additive structure of the pressure makes it possible to relate the
complete Riemann problem to an auxiliary problem governed by the
hydrodynamical equation of state $\ph$ alone. The auxiliary problem has
the classical three-wave structure, and its star pressure is determined by
a single scalar equation. Our strategy is to parametrize the wave curves
of the complete system by the hydrodynamical pressure and to compare them
with the wave curves of the auxiliary system. The influence of $\pb$ then
enters only through computable corrections, and the explicit construction
of the complete wave fan is avoided altogether.

The main contributions of this work are summarized as follows:
\begin{itemize}
  \item[--]
    Under the sole assumption that the barotropic sound speed is real, the
    complete system is hyperbolic and its outer wave curves, elementary or
    composite, can be parametrized by the hydrodynamical star pressure.
    The star state is characterized by a pair of hydrodynamical star
    pressures that differ by the barotropic pressure jump across the
    contact discontinuity.
  \item[--]
    We derive comparison principles between the wave curves of the
    complete and the auxiliary system: an ordering of Hugoniot curves and
    shock functions for shocks, and a bound on the velocity at the edge of
    the fan for rarefactions based on $a^2=\ah^2+\ab^2$. Both extend to
    compressive and expansive composite waves.
  \item[--]
    Theorem~\ref{thm:main} turns these principles into two-sided
    computable bounds on the pair of star pressures and, in turn, into
    guaranteed bounds on the extremal wave speeds. They require only two
    one-sided shock relations of the auxiliary problem and the evaluation
    of $\pb$ at the initial states, and they hold irrespective of whether
    the outer waves are elementary or composite.
  \item[--]
    Theorem~\ref{thm:main_iterative} provides a computable criterion
    certifying that both outer waves are regular, that
    is, elementary with a positive fundamental derivative. In this case a
    monotone iteration sharpens the bounds, and every iterate yields a
    valid wave-speed bound.
  \item[--]
    We assess the bounds numerically for a Noble--Abel stiffened gas
    combined with a localized barotropic contribution that destroys
    convexity, against exact Riemann solutions with elementary and with
    composite outer waves.
\end{itemize}
The resulting estimates are intended for numerical schemes requiring a
guaranteed upper bound on the local maximum signal velocity. In particular,
they provide sufficiently large viscosity coefficients for
invariant-domain-preserving discretizations, such as the first-order
graph-viscosity approximation of
\cite{GuermondPopov2014ViscousRegularization,
GuermondPopov2016InvariantDomains}, while avoiding the cost and complexity
of an exact Riemann solver for the complete equation of state.

\subsection{Paper organization}
The remainder of the paper is organized as follows.
Section~\ref{sec:preliminaries} introduces the Riemann problem, recalls the
classical solution strategy, and states the main results,
Theorems~\ref{thm:main} and~\ref{thm:main_iterative}.
Section~\ref{sec:structural} discusses the structural assumptions on the
hydrodynamical and barotropic contributions and establishes the comparison
principles for elementary shocks, elementary rarefactions, and composite
waves. Section~\ref{sec:matching} combines these principles with the
matching conditions at the contact discontinuity to prove the wave-speed
bounds, a regular-wave criterion, and an iterative refinement of the bounds on the
star pressures. Section~\ref{sec:numerics} assesses the bounds numerically
for a Noble--Abel stiffened gas with a localized barotropic contribution.
We discuss implications and future work in the conclusion
(Section~\ref{sec:conclusion}).

\section{Preliminaries and Notation}
\label{sec:preliminaries}

In this section we describe the precise setup of the Riemann problem under
consideration, and outline the solution strategy that we will follow in the
remainder of the paper.

Consider the Riemann problem \eqref{eq:intro_conservation_law},
\eqref{eq:intro_riemann_data} for the one-dimensional compressible Euler
equations,
\begin{align}
  \label{eq:riemann_problem}
  \begin{cases}
    \begin{aligned}
      \partial_t\rho + \partial_x (\rho u) &= 0, \\[0.25em]
      \partial_t(\rho u) + \partial_x\big(\rho u^2 + p\big)
      &= 0, \\[0.25em]
      \partial_t (\mathcal{E})  +
      \partial_x\big(u (\mathcal{E}+p)\big) &= 0,
    \end{aligned}
  \end{cases}
\end{align}
with a given pressure and total energy
\begin{align}
  \label{eq:composite_eos}
  p\;=\;\ph(\rho,\eh)\,+\,\pb(\rho),
  \qquad\quad
  \mathcal{E}\;=\;\frac{1}{2}\rho\,u^2\,+\,\eh\,+\,\eb(\rho),
\end{align}
and with given initial states $W\L=[\rho\L,u\L,p\L]$ on the left ($x<0$),
and correspondingly $W\R=[\rho\R,u\R,p\R]$ on the right ($x>0$). Here,
$\rho>0$ denotes the density, $u$ the velocity, and $\eh$ the
hydrodynamical internal energy. The total pressure $p$ is
composed of a hydrodynamical contribution $\ph(\rho,\eh)$ given by a
complete equation of state, and a barotropic contribution $\pb(\rho)$
depending on the density alone and related to the barotropic energy
$\eb(\rho)$ by $\pb(\rho)=\rho^2\,\partial_\rho\eb(\rho)$. Precise
structural assumptions on $\ph$ and $\pb$ are postponed to
Section~\ref{sec:structural}. Throughout, a state is described by the
vector $W=[\rho,u,p]$ of primitive variables, in which $p$ stands for the
total pressure \eqref{eq:composite_eos}, and subscripts $\mathrm{h}$ and
$\mathrm{b}$ consistently distinguish hydrodynamical and barotropic
quantities. A hat denotes a quantity expressed as a function of the
specific volume $v=1/\rho$, such as $\hatpb(v)=\pb(1/v)$.

Provided that the equation of state satisfies suitable convexity
conditions, the solution of \eqref{eq:riemann_problem} is self-similar,
$U(x,t)=W(x/t)$, and consists of four constant states separated by three
waves: a left-facing wave, a contact discontinuity traveling with the
intermediate velocity $u\star$, and a right-facing wave; see
Theorem~\ref{thm:riemann_fan} below. The outer waves are shocks or
rarefaction fans and, for a nonconvex total pressure $p$,
possibly composite waves; see
Section~\ref{sec:composite_waves}. We denote the leftmost and rightmost
wave speeds of the Riemann fan by $\lambda^-_1$ and $\lambda^+_3$,
respectively.

We are not concerned with computing the exact solution of
\eqref{eq:riemann_problem}. Our goal is instead to construct a computable
and guaranteed upper bound on the maximal wave speed
$\lambda_{\max}\,:=\,\max\big(|\lambda^-_1|,\,|\lambda^+_3|\big)$, which is
the quantity entering CFL conditions and, in particular, the
graph-viscosity coefficients of invariant-domain-preserving approximations
\cite{GuermondPopov2016InvariantDomains,GuermondPopov2016WaveSpeed,
ClaytonGuermondPopov2022,ClaytonEtAl2023}. For such computational schemes
an underestimate of $\lambda_{\max}$ can lead to a loss of the invariant
domain property, whereas an overestimate merely results in additional
numerical dissipation \cite{GuermondPopov2016WaveSpeed}. Guaranteed upper
bounds are likewise a prerequisite for approximate Riemann solvers of HLL
type \cite{HartenLaxVanLeer1983,Davis1988,Toro2009}.

\subsection{A review of the solution strategy for Riemann problems}
\label{sec:review}

The classical solution strategy \cite{Godunov1959,Toro2009} for the Riemann
problem rests on the fact that pressure and velocity are continuous across
the contact discontinuity, so that the two star states share the values
$p\star$ and $u\star$; see Theorem~\ref{thm:riemann_fan}. For each of the
two outer waves one introduces a \emph{pressure function} $f(p\star;W)$
that returns the velocity change across the wave in terms of the star
pressure $p\star$. Following \cite[Chaps.~3 and~4]{Toro2009}, it
is given for an \emph{elementary} left wave by
\begin{align}
  \label{eq:pressure_function}
  f(p\star;W\L)\;=\;
  \begin{cases}
    \begin{aligned}
      &f\S(p\star;W\L)\;:=\;\big(p\star-p\L\big)^{1/2}
      \big(v\L-v\star\big)^{1/2}
      &\;&\text{if }p\star>p\L,
      \\[0.25em]
      &f\R(p\star;W\L)\;:=\;\int_{\rho\L}^{\rho\star}
      \frac{a(\rho)}{\rho}\bigg|_{\sh}\,\mathrm{d}\rho
      &\;&\text{if }p\star\le p\L,
    \end{aligned}
  \end{cases}
\end{align}
where $v:=1/\rho$ denotes the specific volume and $\rho\star=1/v\star$ the
density behind the wave, which is determined for given $p\star$ by the
Rankine-Hugoniot conditions on the elementary shock branch $f\S$, and by
constant entropy along the isentrope through $W\L$ on the elementary
rarefaction branch $f\R$. Here, $a$ denotes the speed of sound. For a
composite wave the pressure function is more complicated and has to be
assembled from the contributions of the individual elementary waves, see
Section~\ref{sec:composite_waves}. Now, matching $u\star$ from the left and
from the right \cite[Chaps.~3 and~4]{Toro2009},
\begin{align}
  \label{eq:velocity_matching}
  u\star\;=\;u\L-f\big(p\star;W\L\big)\;=\;u\R+f\big(p\star;W\R\big),
\end{align}
shows that the star pressure is a root of the equation
\begin{align}
  \label{eq:pressure_equation}
  f\big(p\star;W\L\big)\,+\,f\big(p\star;W\R\big)
  \,+\,u\R-u\L\;=\;0.
\end{align}
Once $p\star$ is known, $u\star$ follows from either equality in
\eqref{eq:velocity_matching}, and, provided that both outer waves are
elementary, the extremal wave speeds are recovered from $p\star$ directly,
see \cite[Chaps.~3 and~4]{Toro2009},
\begin{align}
  \label{eq:lambda_one}
  \begin{cases}
    \begin{aligned}
      \lambda^-_1\;&=\;
      \begin{cases}
        \begin{aligned}
          &u\L-v\L\,Q\big(p\star;W\L\big) &\;&\text{if }p\star>p\L,
          \\[0.25em]
          &u\L-a(W\L) &\;&\text{if }p\star\le p\L,
        \end{aligned}
      \end{cases}
      \\[0.5em]
      \lambda^+_3\;&=\;
      \begin{cases}
        \begin{aligned}
          &u\R+v\R\,Q\big(p\star;W\R\big) &\;&\text{if }p\star>p\R,
          \\[0.25em]
          &u\R+a(W\R) &\;&\text{if }p\star\le p\R.
        \end{aligned}
      \end{cases}
    \end{aligned}
  \end{cases}
\end{align}
Here, we have used the mass flux
\begin{align}
  \label{eq:mass_flux}
  Q\big(p\star;W\big)\;:=\;
  \frac{\big(p\star-p\big)^{1/2}}{\big(v-v\star\big)^{1/2}},
\end{align}
and $p$ and $v$ denote the pressure and the specific volume of the state
$W$, and $v\star$ the specific volume behind the respective wave.

\subsection{Strategy and main result}
\label{sec:strategy}

Carrying out the program of Section~\ref{sec:review} for the composite
equation of state \eqref{eq:composite_eos} requires the wave curves of the
full system, which is precisely the effort we wish to avoid. We do so
deliberately, because we want to admit equations of state for which the
pressure functions \eqref{eq:pressure_function} are difficult or impossible
to evaluate analytically. This includes the situations discussed above, in
which the outer waves are no longer elementary but composite, see
Section~\ref{sec:composite_waves}, and it includes equations of state for
which \eqref{eq:pressure_equation} fails to be uniquely solvable, see
Remark~\ref{rem:no_uniqueness}. We make no attempt to resolve any of this.
All we wish to construct is an upper bound on the wavespeeds that remains
valid in all of these cases.

We therefore assume throughout unique solvability of the
\emph{hydrodynamical} problem obtained from \eqref{eq:riemann_problem} by
dropping the barotropic contributions $\pb$ and $\eb$, viz.,
\begin{align}
  \label{eq:riemann_problem_h}
  \begin{cases}
    \begin{aligned}
      \partial_t\rho + \partial_x (\rho u) &= 0, \\[0.25em]
      \partial_t(\rho u) + \partial_x\big(\rho u^2 + \ph\big) &= 0,
      \\[0.25em]
      \partial_t \mathcal{E}_{\mathrm h} +
      \partial_x\big(u (\mathcal{E}_{\mathrm h}+\ph)\big) &= 0,
    \end{aligned}
  \end{cases}
  \qquad
  \mathcal{E}_{\mathrm h}\;=\;\frac{1}{2}\rho\,u^2\,+\,\eh,
\end{align}
with $\ph=\ph(\rho,\eh)$ and subject to the same initial data $W\L$ and
$W\R$; this is the setting of Theorem~\ref{thm:riemann_fan} below, and in
particular both outer waves of \eqref{eq:riemann_problem_h} are
elementary.

For later reference we record the three ingredients of
Section~\ref{sec:review} for the reduced problem
\eqref{eq:riemann_problem_h}. Marking all quantities by a subscript
$\mathrm{h}$, the pressure function \eqref{eq:pressure_function} becomes
\begin{align}
  \label{eq:pressure_function_h}
  f_{\mathrm h}(\tilde p_h\star;W\L)\;=\;
  \begin{cases}
    \begin{aligned}
      &f\S_{\mathrm h}(\tilde p_h\star;W\L)\;:=\;
      \big(\tilde p_h\star-\ph\L\big)^{1/2}
      \big(v\L-\tilde v\star\big)^{1/2}
      &\;&\text{if }\tilde p_h\star>\ph\L,
      \\[0.25em]
      &f\R_{\mathrm h}(\tilde p_h\star;W\L)\;:=\;
      \int_{\rho\L}^{\tilde\rho\star}
      \frac{\ah(\rho,\eh)}{\rho}\bigg|_{\sh}\,\mathrm{d}\rho
      &\;&\text{if }\tilde p_h\star\le\ph\L,
    \end{aligned}
  \end{cases}
\end{align}
with the hydrodynamical star volume $\tilde v\star$ given by the reduced
Rankine-Hugoniot condition
\begin{align*}
  \tag{\texorpdfstring{RH\textsubscript{h}}{RHh}}
  \hateh(\tilde v\star,\tilde p_h\star) -\hateh(v\L,\ph\L)
  \,+\,\frac{1}{2}(\tilde p_h\star+\ph\L)
  \big(\tilde v\star - v\L\big) \;=\;0,
\end{align*}
the matching condition \eqref{eq:pressure_equation} for the star pressure
$\tilde p_h\star$ now reads
\begin{align}
  \label{eq:pressure_equation_h}
  0\;&=\;f_{\mathrm h}\big(\tilde p_h\star;W\L\big)
  \,+\,f_{\mathrm h}\big(\tilde p_h\star;W\R\big)\,+\,u\R-u\L,
\end{align}
and the leftmost wavespeed \eqref{eq:lambda_one} is given by
\begin{align}
  \label{eq:lambda_one_h}
  \begin{cases}
    \begin{aligned}
      \lambda^-_{1,\mathrm h}\;&=\;
      \begin{cases}
        \begin{aligned}
          &u\L-v\L\,Q_{\mathrm h}\big(\tilde p_h\star;W\L\big)
          &\,&\text{if }\tilde p_h\star>\ph\L,
          \\[0.25em]
          &u\L-\ah(W\L)
          &\,&\text{if }\tilde p_h\star\le\ph\L,
        \end{aligned}
      \end{cases}
      \\[0.5em]
      \lambda^+_{3,\mathrm h}\;&=\;
      \begin{cases}
        \begin{aligned}
          &u\R+v\R\,Q_{\mathrm h}\big(\tilde p_h\star;W\R\big)
          &\,&\text{if }\tilde p_h\star>\ph\R,
          \\[0.25em]
          &u\R+\ah(W\R)
          &\,&\text{if }\tilde p_h\star\le\ph\R,
        \end{aligned}
      \end{cases}
    \end{aligned}
  \end{cases}
\end{align}
with the hydrodynamical mass flux $Q_{\mathrm h}(\tilde p_h\star;W)$
defined as in \eqref{eq:mass_flux}, but with $\ph$ in place of $p$ and with
$\tilde v\star$ in place of $v\star$.

In Sections~\ref{sec:structural} and~\ref{sec:matching} we will
establish the following main result.
\begin{theorem}[General wavespeed bounds]
  \label{thm:main}
  Assume that the left and right waves are each either compressive, or
  expansive (see \eqref{eq:composite_compressive} and
  \eqref{eq:composite_expansive}). Let $\overbar q\L$ and $\overbar q\R$,
  as well as $\underbar q\L$ and $\underbar q\R$, be the unique solutions
  of
  \begin{align}
    \label{eq:main_onesided}
    \begin{cases}
      \begin{aligned}
        f\S_{\mathrm h}\big(\overbar q\L;W\L\big)\,
        &=\,\max\big(u\L-u\R,\,0\big)
        =\,f\S_{\mathrm h}\big(\overbar q\R;W\R\big),
        \\[0.25em]
        f\S_{\mathrm h}\big(\underbar q\L;W\L\big)\,
        &=\,\min\big(u\L-u\R,\,0\big)
        \,=\,f\S_{\mathrm h}\big(\underbar q\R;W\R\big),
      \end{aligned}
    \end{cases}
  \end{align}
  where $f\S_{\mathrm h}(\,\cdot\,;W)$ is understood through its
  continuation past $\ph$, see Definition~\ref{def:expansive_continuation},
  so that $\underbar q\L\le\ph\L\le \overbar q\L$ and $\underbar
  q\R\le\ph\R\le\overbar q\R$. Note that exactly one of the two equations
  in \eqref{eq:main_onesided} is nontrivial. Let
  \begin{align}
    \label{eq:main_pressure_bound}
    \begin{cases}
      \begin{aligned}
        \barph\starL\,&:=\,\max\big\{\overbar q\L, \,p\R-\hatpb(v\L)\big\},
        &\;
        \barph\starR\,&:=\,\max\big\{\overbar q\R, \,p\L-\hatpb(v\R)\big\},
        \\[0.25em]
        \underph\starL\,&:=\, \min\big\{\underbar q\L,\,p\R-\hatpb(v\L)\big\},
        &\;
        \underph\starR\,&:=\, \min\big\{\underbar q\R,\,p\L-\hatpb(v\R)\big\}.
      \end{aligned}
    \end{cases}
  \end{align}
  Let $\bar v\starL$ and $\bar v\starR$, as well as
  $\underbar v\starL$ and $\underbar v\starR$, be determined by the
  Rankine-Hugoniot condition \eqref{eq:rankine_hugoniot_h} for
  $\barph\starL$ and $W\L$, and for $\barph\starR$ and $W\R$, respectively
  (and, correspondingly, for $\underph\starL$ and
  $W\L$, and for $\underph\starR$ and $W\R$), so that $\bar v\starL\le
  v\L\le\underbar v\starL$ and $\bar v\starR\le v\R\le\underbar v\starR$,
  and let
  \begin{align}
    \label{eq:main_lambda}
    \bar\lambda^-_{1,\mathrm h}\;:=\;u\L-v\L\,
    Q_{\mathrm h}\big(\barph\starL;W\L\big),
    \qquad
    \bar\lambda^+_{3,\mathrm h}\;:=\;u\R+v\R\,
    Q_{\mathrm h}\big(\barph\starR;W\R\big).
  \end{align}
  Then, under mild structural assumptions on the equation of state
  \eqref{eq:composite_eos} and on the solution of
  \eqref{eq:riemann_problem_h}, all made precise in
  Section~\ref{sec:structural}, the extremal wavespeeds of the full
  problem \eqref{eq:riemann_problem} obey
  \begin{align}
    \label{eq:main_wavespeed}
    \left\{
    \begin{aligned}
      \lambda^-_1\;&\ge\;
      u\L-\sqrt{\big(u\L-\bar\lambda^-_{1,\mathrm h}\big)^2
      \,+\,2\,(v\L)^2\!\!\max_{\bar v\starL\le\tau\le\underbar v\starL}\!
      \frac{\hatab^2(\tau)}{\tau^2}},
      \\[0.75em]
      \lambda^+_3\;&\le\;
      u\R+\sqrt{\big(\bar\lambda^+_{3,\mathrm h}-u\R\big)^2
      \,+\,2\,(v\R)^2\!\!\max_{\bar v\starR\le\tau\le\underbar v\starR}\!
      \frac{\hatab^2(\tau)}{\tau^2}}.
    \end{aligned}
    \right.
  \end{align}
  Importantly, no case distinction on the type of the outermost elementary
  wave of either family is necessary.
\end{theorem}
\begin{remark}
  The lower bounds on the pressures, $\underph\starL$ and $\underph\starR$,
  are well defined as long as $\min(u\L-u\R,0)$ lies in the range of the
  continued shock function $f\S_{\mathrm h}(\,\cdot\,;W)$. Should this
  fail, then $\underbar v\starL$ and $\underbar v\starR$ in
  \eqref{eq:main_wavespeed} may be replaced by any suitable upper bound
  derived from a priori knowledge on the hydrodynamical equation of state.
\end{remark}
The bounds \eqref{eq:main_wavespeed} hold irrespective of whether the outer
waves are elementary or composite, and irrespective of whether they are
compressive or expansive. Under additional smallness and regularity
assumptions, however, both outer waves are \emph{regular}, meaning that
each of them is an elementary wave, described by a single pressure function
of a shock or a rarefaction, along which the fundamental derivative
\eqref{A3} stays positive, and the bounds can be improved by an iteration:
\begin{theorem}[Regular wavespeed bounds]
  \label{thm:main_iterative}
  In the situation of Theorem~\ref{thm:main}, assume that the assumptions
  of Lemma~\ref{lem:matching_volume_bound} are satisfied for $W\L$ and
  $W\R$. Let $\tilde p_h\star$ be the solution of
  \eqref{eq:pressure_equation_h} and set
  \begin{align}
    \label{eq:main_pistar}
    \underbar\pi\star\;:=\;\min\big(\tilde p_h\star,\,\ph\L,\,\ph\R\big),
    \qquad
    \bar\pi\star\;:=\;\max\big(\tilde p_h\star,\,\ph\L,\,\ph\R\big).
  \end{align}
  Set $\bar P^{(0)}_{\mathrm L}:=\barph\starL$, $\bar P^{(0)}_{\mathrm
  R}:=\barph\starR$, $\underbar P^{(0)}_{\mathrm L}:=\underph\starL$ and
  $\underbar P^{(0)}_{\mathrm R}:=\underph\starR$. Assume that the
  smallness conditions $4vm\le d$ and $\Delta_{\mathrm b}^+<d$ with the
  quantities \eqref{eq:matching_volume_data} and
  \eqref{eq:deviation_delta_b}, the positivity $\underbar\Xi(\bar
  P^{(0)};W)>0$ of the lower volume bound \eqref{eq:matching_volume_bound},
  as well as the convexity criterion \eqref{eq:kappa_criterion}, hold true
  for these initial pressure bounds and the states $W\L$ and $W\R$. Then,
  each of the left and right waves is a regular shock
  or a regular rarefaction, and we can define
  recursively, for $n\ge0$:
  \begin{align}
    \label{eq:main_bootstrap}
    \begin{cases}
      \begin{aligned}
        \underbar V^{(n)}_{\mathrm L}
        &:=
        \underbar\Xi\big(\bar P^{(n)}_{\mathrm L};W\L\big),
        \quad
        \bar V^{(n)}_{\mathrm L}
        :=
        \overbar\Xi\big(\underbar P^{(n)}_{\mathrm L};W\L\big),
        \\[0.25em]
        \underbar V^{(n)}_{\mathrm R}
        &:=
        \underbar\Xi\big(\bar P^{(n)}_{\mathrm R};W\R\big),
        \quad
        \bar V^{(n)}_{\mathrm R}
        :=
        \overbar\Xi\big(\underbar P^{(n)}_{\mathrm R};W\R\big),
        \\[0.25em]
        \underbar\delta^{(n)}&:=
          \hatpb\big(\underbar V^{(n)}_{\mathrm L}
          \big)-\hatpb\big(\bar V^{(n)}_{\mathrm R}\big),
        \quad
        \bar\delta^{(n)}:=
          \hatpb\big(\underbar V^{(n)}_{\mathrm R}
          \big)-\hatpb\big(\bar V^{(n)}_{\mathrm L}\big),
        \\[0.25em]
        \varepsilon^{(n)}&:=
        \max\big(\bar\delta^{(n)},\,
        \underbar\delta^{(n)}\big),
        \\[0.25em]
        \underbar\eta^{(n)}&:=
        -\eta\S\big(\bar P^{(n)}_{\mathrm L},\underbar V^{(n)}_{\mathrm L};W\L\big)
        -\eta\S\big(\bar P^{(n)}_{\mathrm R},\underbar V^{(n)}_{\mathrm R};W\R\big),
        \\[0.25em]
        \bar\eta^{(n)}&:=
        \eta\R\big(\underbar P^{(n)}_{\mathrm L},\bar V^{(n)}_{\mathrm L};W\L\big)
        +\eta\R\big(\underbar P^{(n)}_{\mathrm R},\bar V^{(n)}_{\mathrm R};W\R\big),
        \\[0.25em]
        \underbar\pi^{\ast(n)}&:=
        \max\big(\underbar\pi\star,\,\Phi(\underbar\eta^{(n)};W\L,W\R)\big),
        \\[0.25em]
        \bar\pi^{\ast(n)}&:=
        \min\big(\bar\pi\star,\,\Phi(\bar\eta^{(n)};W\L,W\R)\big),
        \\[0.25em]
        \bar P^{(n+1)}_{\mathrm L}&:=
        \min\big(\bar P^{(n)}_{\mathrm L},\,
        \bar\pi^{\ast(n)}+\max(\bar\delta^{(n)},0)\big),
        \\[0.25em]
        \bar P^{(n+1)}_{\mathrm R}&:=
        \min\big(\bar P^{(n)}_{\mathrm R},\,
        \bar\pi^{\ast(n)}+\max(\underbar\delta^{(n)},0)\big),
        \\[0.25em]
        \underbar P^{(n+1)}_{\mathrm L}&:=
        \max\big(\underbar P^{(n)}_{\mathrm L},\,
        \underbar\pi^{\ast(n)}-\max(\underbar\delta^{(n)},0)\big),
        \\[0.25em]
        \underbar P^{(n+1)}_{\mathrm R}&:=
        \max\big(\underbar P^{(n)}_{\mathrm R},\,
        \underbar\pi^{\ast(n)}-\max(\bar\delta^{(n)},0)\big).
      \end{aligned}
    \end{cases}
  \end{align}
  Here, $\underbar\Xi$ and $\overbar\Xi$ denote the computable volume
  bounds defined in Lemma~\ref{lem:matching_volume_bound}, $\eta\S$ and
  $\eta\R$ the deviation bounds \eqref{eq:deviation_eta} of
  Lemma~\ref{lem:deviation_bounds}, and $\Phi(\eta;W\L,W\R)$ the solution
  of the shifted hydrodynamical pressure equation
  \eqref{eq:pressure_equation_shifted}. Then, the sequences $\bar
  P^{(n)}_{\mathrm L/\mathrm R}$ are non-increasing, the sequences
  $\underbar P^{(n)}_{\mathrm L/\mathrm R}$ are non-decreasing, and, for
  every $n\ge0$, the bounds \eqref{eq:main_wavespeed} remain valid with
  $\barph\starL$, $\barph\starR$, $\underph\starL$ and $\underph\starR$
  replaced by $\max(\bar P^{(n)}_{\mathrm L},\ph\L)$, $\max(\bar
  P^{(n)}_{\mathrm R},\ph\R)$, $\min(\underbar P^{(n)}_{\mathrm L},\ph\L)$
  and $\min(\underbar P^{(n)}_{\mathrm R},\ph\R)$.
\end{theorem}
\begin{remark}
  \label{rem:shifted_no_solution}
  If \eqref{eq:pressure_equation_shifted} has no solution for
  $\eta=\underbar\eta^{(n)}$ or $\eta=\bar\eta^{(n)}$, the corresponding
  argument of the maximum or minimum in \eqref{eq:main_bootstrap} shall be
  omitted.
\end{remark}
If the pressure and volume bounds of Theorem~\ref{thm:main_iterative} are
closing in on the star states, the wavespeed calculation itself can be
improved. For every $n\ge0$ the iteration \eqref{eq:main_bootstrap}
guarantees
\begin{align*}
  \underbar P^{(n)}_{\mathrm L}\;\le\;
  \ph\starL\;\le\;\bar P^{(n)}_{\mathrm L},
  \qquad
  \underbar P^{(n)}_{\mathrm R}\;\le\;
  \ph\starR\;\le\;\bar P^{(n)}_{\mathrm R}.
\end{align*}
We call such a bracket \emph{sufficiently small} if it does not contain
the pressure of the adjacent initial state, i.\,e., if $\ph\L\notin[\underbar
P^{(n)}_{\mathrm L},\bar P^{(n)}_{\mathrm L}]$, or $\ph\R\notin[\underbar
P^{(n)}_{\mathrm R},\bar P^{(n)}_{\mathrm R}]$, respectively. A
sufficiently small bracket determines the type of the wave, and the
factor two in \eqref{eq:main_wavespeed}, which is the price for not
knowing the type, can be dropped.
\begin{corollary}[Wavespeed bounds for sufficiently small brackets]
  \label{cor:main_closing}
  In the situation of Theorem~\ref{thm:main_iterative}, let $n\ge0$, and
  let $\bar\lambda^-_{1,\mathrm h}$, $\bar v\starL$,
  $\bar\lambda^+_{3,\mathrm h}$ and $\bar v\starR$ be defined as in
  Theorem~\ref{thm:main}, but with $\bar P^{(n)}_{\mathrm L}$ and $\bar
  P^{(n)}_{\mathrm R}$ in place of $\barph\starL$ and $\barph\starR$.

  (a) \emph{Left shock.} If $\ph\L<\underbar P^{(n)}_{\mathrm L}$, then
  the left wave is a shock, and
  \begin{align*}
    \lambda^-_1 \;\ge\; u\L -\sqrt{
    \Big(1+\frac{\hatpb(\bar v\starL)-\hatpb(v\L)}
    {\underbar P^{(n)}_{\mathrm L}-\ph\L}\Big)
    \big(u\L-\bar\lambda^-_{1,\mathrm{h}}\big)^2+ (v\L)^2
    \max_{\bar v\starL\,\le\,\tau\le v\L}\frac{\hatab^2(\tau)}{\tau^2}}.
  \end{align*}
  (b) \emph{Left rarefaction.} If $\bar P^{(n)}_{\mathrm L}<\ph\L$, then
  the left wave is a rarefaction, and
  \begin{align*}
    \lambda^-_1\;=\;u\L-\sqrt{\ah^2\big(\rho\L,\eh(\rho\L,\ph\L)\big)
    \,+\,\ab^2(\rho\L)}.
  \end{align*}
  (c) \emph{Right shock.} If $\ph\R<\underbar P^{(n)}_{\mathrm R}$, then
  the right wave is a shock, and
  \begin{align*}
    \lambda^+_3 \;\le\; u\R +\sqrt{
    \Big(1+\frac{\hatpb(\bar v\starR)-\hatpb(v\R)}
    {\underbar P^{(n)}_{\mathrm R}-\ph\R}\Big)
    \big(\bar\lambda^+_{3,\mathrm{h}}-u\R\big)^2+ (v\R)^2
    \max_{\bar v\starR\,\le\,\tau\le v\R}\frac{\hatab^2(\tau)}{\tau^2}}.
  \end{align*}
  (d) \emph{Right rarefaction.} If $\bar P^{(n)}_{\mathrm R}<\ph\R$, then
  the right wave is a rarefaction, and
  \begin{align*}
    \lambda^+_3\;=\;u\R+\sqrt{\ah^2\big(\rho\R,\eh(\rho\R,\ph\R)\big)
    \,+\,\ab^2(\rho\R)}.
  \end{align*}
  If $\underbar P^{(n)}_{\mathrm L}\le\ph\L\le\bar P^{(n)}_{\mathrm L}$,
  or $\underbar P^{(n)}_{\mathrm R}\le\ph\R\le\bar P^{(n)}_{\mathrm R}$,
  the respective bound of Theorem~\ref{thm:main_iterative} remains in
  force.
\end{corollary}
Section~\ref{sec:structural} sets up the proof of Theorems~\ref{thm:main}
and~\ref{thm:main_iterative}, and Section~\ref{sec:numerics} discusses how
to implement an efficient wavespeed estimate based on them and showcases a
number of numerical results.


\section{Comparison principles for the left wave}
\label{sec:structural}
This section develops the tools behind Theorems~\ref{thm:main}
and~\ref{thm:main_iterative}. We first collect the structural assumptions
on the two constituents of the composite equation of state
\eqref{eq:composite_eos}. We then take the left wave of the Riemann problem
\eqref{eq:riemann_problem} apart, one wave type at a time: a compressive
shock, an expansive shock, a rarefaction, and a composite wave. For each
type we make the pressure function \eqref{eq:pressure_function} and the
wavespeed explicit and compare them with their counterparts of the reduced
problem \eqref{eq:riemann_problem_h}. Estimates for the right wave follow
analogously and are not discussed separately.

\subsection{Assumptions on the equation of state}
\label{sec:assumptions}
The two constituents $\ph(\rho,\eh)$ and $\pb(\rho)$ of the
composite equation of state \eqref{eq:composite_eos} are subject to the
following assumptions. For a given function $f(\rho)$ depending on the
density we introduce a change of variables to the specific volume
$v=1/\rho$ by setting $\hat f(v):= f\big(\rho(v)\big)=f(1/v)$.
\begin{definition}[Assumptions on $\ph(\rho,\eh)$]
  Let $\ph(\rho,\eh)$ be a given hydrodynamical equation of state with an
  associated specific entropy $\sh(\rho,\eh)$ and a speed of
  sound $\ah(\rho,\eh)$, such that \cite{Toro2009,Smith1979}
  \begin{align}
    \tag{A.\,0}\label{A0}
    \ph(\rho,\eh)>-p_\infty,
    \;\;\;
    \ah^2(\rho,\eh)=-v^2\,\partial_{v}\hatph(v,\eh)\big|_{s_{\mathrm h}}>0,
    \;\;\;
    \partial_{v}\partial_{v}\hatph(v,\eh)\big|_{s_{\mathrm h}}>0.
  \end{align}
  Let $\hateh(v,\ph)$ denote the specific internal energy as a function of
  specific volume $v$ and pressure $\ph$. Furthermore, we assume that
  \begin{gather}
    \tag{A.\,1}\label{A1}
    \partial_v\hateh(v,\ph)\Big|_{\ph}\ge p_\infty,
    \quad
    \partial_p\hateh(v,\ph)\Big|_{v}\ge 0,
    \quad
    \partial_{\sh}\hatph(v,\eh)\Big|_{v}>0,
    \\\notag
    \hateh(v,\ph)\ge q,\quad\text{and}\quad
    \lim_{v\to0+}\hateh(v,\ph)\Big|_{\ph}=q\,\in\,\mathbb{R}.
  \end{gather}
  Note that $\partial_{\sh}\hatph>0$ is a consequence of
  $\partial_p\hateh>0$ and positive temperature $T>0$.
\end{definition}

\begin{remark}
  \label{rem:NASG}
  Assumptions \eqref{A0} and \eqref{A1} are satisfied by a large class of
  equations of state. In the following we will make use of the
  \emph{Noble-Abel stiffened gas} equation of state,
  \begin{align*}
    \ph^{\text{NASG}}(\rho,\eh) = (\gamma - 1) \rho\,(\eh - q) / (1 - b
    \rho) - \gamma p_\infty,
  \end{align*}
  where the ratio of specific heats $\gamma>1$, the covolume $0\le b<1$ and
  $p_\infty\ge0$ are constants. We compute
  \begin{align*}
    \hateh^{\text{NASG}}(v,\ph) = q +
    v\,(1-b/v)\Big(\frac{(\ph+p_\infty)}{\gamma-1}+p_\infty\Big),
  \end{align*}
  which implies $\partial_v\hateh(v,\ph)\Big|_{\ph} \,=\,
  \frac{(\ph+p_\infty)}{\gamma-1}+p_\infty\,=\,v(1-b/v)\,(\eh-q)$. Closed-form
  expressions for the wave curves and wavespeeds of the corresponding
  hydrodynamical Riemann problem \eqref{eq:riemann_problem_h} are collected
  in Section~\ref{sec:nasg}.
\end{remark}
We have the following theorem.
\begin{theorem}[Riemann fan
  \cite{Weyl1949,CourantFriedrichs1977,Smith1979,MenikoffPlohr1989,Toro2009}]
  \label{thm:riemann_fan}
  Let assumptions~\eqref{A0} and \eqref{A1} hold true. Then, the
  hydrodynamical Riemann problem \eqref{eq:riemann_problem_h} is
  well-posed. The solution is self-similar, $U(x,t)=W(x/t)$, consisting of
  four constant states denoted by $W\L$, $W\starL$, $W\starR$, $W\R$ from
  left to right, where $W\L$ and $W\starL$, as well as $W\starR$ and $W\R$
  are connected by an elementary rarefaction wave or an elementary shock
  wave. The states $W\starL$ and $W\starR$ are always connected by a
  contact discontinuity implying
  \begin{align*}
    u\star:=u\starL=u\starR
    \quad\text{and}\quad
    \tilde p_h\star:=p\starL=p\starR.
  \end{align*}
\end{theorem}
The barotropic pressure contribution $\pb(\rho)$ is required to have
the following structural properties.
\begin{definition}[Assumptions on $\eb$ and $\pb$]
  Assume that $\eb(\rho)$ is a three times continuously differentiable
  function of the density $\rho$ satisfying
  $\pb(\rho)\,:=\,\rho^2\,\partial_\rho\eb(\rho)$, such that
  \begin{align*}
    \tag{A.\,2}\label{A2}
    \ab^2(\rho)\,:=\, \partial_{\rho}\pb(\rho)\;\ge\;0.
  \end{align*}
\end{definition}
\begin{remark}
  Assumptions \eqref{A0} and \eqref{A1} are crucial for
  Theorem~\ref{thm:riemann_fan} to hold true. In particular, if the
  convexity condition of \eqref{A0} is violated, then composite
  shock-rarefaction waves can emerge in the Riemann problem
  \cite{Wendroff1972,MenikoffPlohr1989,MullerVoss2006}. Moreover, violating
  condition \eqref{A1} can lead to non-uniqueness in the matching of states
  connected by a shock \cite{Smith1979}. While it is tempting to avoid such
  issues, we will not enforce additional conditions on $\hat\pb(v)$.
  Instead, we will quantify how much concavity of $\hat\pb(v)$ can be
  compensated by the hydrodynamical equation of state and the implications
  for the solution structure of the Riemann problem.
\end{remark}

A consequence of Assumption \eqref{A2} on the barotropic speed of sound
contribution is the fact that it ensures hyperbolicity of the Riemann
problem \eqref{eq:riemann_problem} with full pressure $p$, and creates a
correspondence between its compression and rarefaction conditions and
those emerging in the reduced Riemann problem
\eqref{eq:riemann_problem_h}:
\begin{lemma}
  \label{lem:no_weird_stuff}
  Let $W=[\rho,u,p]$ and let $W\star=[\rho\star,u\star,p\star]$ be two
  given states. Then,
  \begin{align*}
    \rho^\ast\ge\rho,\;\sh\star\ge\sh\;\text{ and }\; p^\ast\ge p
    \qquad&\Leftrightarrow\qquad
    \rho^\ast\ge\rho,\;\sh\star\ge\sh\;\text{ and }\; \ph^\ast\ge\ph,
    \\[0.25em]
    \rho^\ast\le\rho,\;\sh\star\le\sh\;\text{ and }\; p^\ast\le p
    \qquad&\Leftrightarrow\qquad
    \rho^\ast\le\rho,\;\sh\star\le\sh\;\text{ and }\; \ph^\ast\le\ph.
  \end{align*}
  The first case corresponds to a compressive shock,
  or (with $\sh\star=\sh$) to a \emph{compressive} rarefaction; the second
  case corresponds (with $\sh\star=\sh$) to an expansive rarefaction, or an
  \emph{expansive} shock.
\end{lemma}
\begin{proof}
  Since $p=\ph+\pb(\rho)$, we have
  $p\star-p=(\ph\star-\ph)+\big(\pb(\rho\star)-\pb(\rho)\big)$, where the
  last term has the same sign as $\rho\star-\rho$ by \eqref{A2}. Moreover,
  \eqref{A0} and \eqref{A1} imply
  $\partial_v\hatph(v,\eh)\big|_{\sh}=-\ah^2/v^2<0$ and
  $\partial_{\sh}\hatph(v,\eh)\big|_{v}>0$, so that $\rho\star\ge\rho$ and
  $\sh\star\ge\sh$ already imply $\ph\star\ge\ph$, and hence $p\star\ge p$.
  The case $\rho\star\le\rho$ and $\sh\star\le\sh$
  is analogous.
\end{proof}
Composite waves are subject to three additional structural assumptions,
\eqref{A3} to \eqref{A5}, which, for the sake of exposition, we state in
Definition~\ref{def:composite_assumptions} of
Section~\ref{sec:composite_waves}.

\subsection{Left compressive shock}
\label{sec:left_shock}
We start with the case that the left state and the star state are connected
by a single, elementary shock. Our goal is to make the shock branch $f\S$
of the pressure function \eqref{eq:pressure_function} and the mass flux
\eqref{eq:mass_flux} explicit for the composite equation of state
\eqref{eq:composite_eos}, and to compare both quantities with their
hydrodynamical counterparts \eqref{eq:pressure_function_h} and
\eqref{eq:lambda_one_h}.

Let $W\L=[\rho\L,u\L,p\L]$ be a left state and
$W\star=[\rho\star,u\star,p\star]$ a star state (situated right of the left
state) and assume both states are connected by a compressive elementary
shock, i.\,e., $\rho\star\ge\rho\L$ ($v\star\le v\L$), $\sh\star\ge\sh\L$
and $p\star\ge p\L$. We will discuss the case of expansive shocks in
Section~\ref{sec:left_expansive_shock}. Assume that the Rankine-Hugoniot
condition,
\begin{align*}
  \eh\star\,+\,\eb(\rho\star)\,-\,\eh\L\,-\,\eb(\rho\L)
  \,+\,\frac{1}{2}\big(\ph\star+\pb(\rho\star)+\ph\L+\pb(\rho\L)\big)
  \big(v\star-v\L\big)
  \;=\;0,
\end{align*}
is verified. For later use we introduce
\begin{align}
  \label{eq:Ab}
  \hatAb(v\star, v)\;:=\;
  \hateb(v\star)-\hateb(v)\,+\,\frac{1}{2}\big(\hatpb(v\star)+\hatpb(v)\big)
  \big(v\star-v\big)
\end{align}
and rewrite the Rankine-Hugoniot condition as follows:
\begin{align}
  \label{eq:rankine_hugoniot} \tag{RH}
  \hateh(v\star,\ph\star)-\hateh(v\L,\ph\L)\,+\,\frac{1}{2}(\ph\star+\ph\L)
  \big(v\star-v\L\big) \,+\,\hatAb(v\star,v\L) \;=\;0.
\end{align}
Here, $\hatAb$ collects precisely those contributions that are absent in
the reduced problem \eqref{eq:riemann_problem_h}. We will establish that it
is---for fixed data $W\L$ and at least locally around
$(v\star,\ph\star)$---possible to uniquely solve
\eqref{eq:rankine_hugoniot} yielding an implicit function
$v\star=\xi(\ph\star;W\L)$ of solutions to \eqref{eq:rankine_hugoniot}. Our
strategy is to relate \eqref{eq:rankine_hugoniot} to the corresponding
condition for \eqref{eq:riemann_problem_h}. We start by collecting a number
of useful observations.
\begin{lemma}
  \label{lem:sign_on_ab}
  We have
  \begin{align}
    \label{eq:sign_on_ab}
    \frac{\mathrm d}{\mathrm d v\star}\hat A_{\mathrm{b}}(v\star, v\L)
    \;=\;
    \frac{1}{2} \int_{v\star}^{v\L} \!\int_{v\star}^{\tau}
    \!\!\partial_v\partial_v\hatpb(\kappa) \mathrm{d}\kappa\,
    \mathrm{d}\tau.
  \end{align}
  In particular, $\hat A_{\mathrm{b}}(v\star,v\L)\ge0$ for $\hatpb(v)$
  concave, and $\hat A_{\mathrm{b}}(v\star,v\L)\le0$ for $\hatpb(v)$ convex.
  Moreover, $\hat A_{\mathrm{b}}(v\star,v\L)\,\equiv\,0$ holds true if and
  only if $\hatpb(v)$ is a linear function in $v$.
\end{lemma}

\begin{proof}
  Observe that $\hat A_{\mathrm{b}}(v,v)=0$. Moreover,
  \begin{align}
    \label{eq:sign_on_ab_proof}
    \frac{\mathrm d}{\mathrm d v\star}\hat A_{\mathrm{b}}(v\star, v)
    \;&=\;
    \partial_v\hateb(v\star)
    \,+\,\frac{1}{2}\big(\hatpb(v\star)+\hatpb(v)\big)
    \,+\, \frac{1}{2}\partial_v\hatpb(v\star) (v\star-v)
    \\\notag\;&=\;
    \frac{1}{2}\big(\hatpb(v)-\hatpb(v\star)\big)
    \,+\, \frac{1}{2}\partial_v\hatpb(v\star) (v\star-v)
    \\\notag\;&=\;
    \frac{1}{2} \Big( \hatpb(v)\,-\, \big\{\hatpb(v\star)
    \,+\, \partial_v\hatpb(v\star) (v-v\star)\big\}\Big)
    \\\notag\;&=\;
    \frac{1}{2} \int_{v\star}^{v}\big(\partial_v\hatpb(\tau)-
    \partial_v\hatpb(v\star)\big)\mathrm{d}\tau
    \\\notag\;&=\;
    \frac{1}{2} \int_{v\star}^{v}
    \int_{v\star}^{\tau}
    \!\!\partial_v\partial_v\hatpb(\kappa)
    \mathrm{d}\kappa\,
    \mathrm{d}\tau.
  \end{align}
  This implies $\frac{\mathrm d}{\mathrm d v\star}\hat
  A_{\mathrm{b}}(v\star, v)\,\le\,0$ for $\hatpb(v)$ concave, and
  $\frac{\mathrm d}{\mathrm d v\star}\hat A_{\mathrm{b}}(v\star,
  v)\,\ge\,0$ for $\hatpb(v)$ convex. Observing $v\ge v\star$ establishes
  the second part of the Lemma.
\end{proof}

\begin{lemma}
  \label{lem:inequality_on_density}
  Let $\ph\star$ be fixed. Let $v\star$ be given by
  \eqref{eq:rankine_hugoniot} and introduce $\tilde v\star$ satisfying
  \begin{align}
    \label{eq:rankine_hugoniot_h}
    \tag{\texorpdfstring{RH\textsubscript{h}}{RHh}}
    \hateh(\tilde v\star,\ph\star) -\hateh(v\L,\ph\L)
    \,+\,\frac{1}{2}(\ph\star+\ph\L) \big(\tilde v\star - v\L\big) \;=\;0.
  \end{align}
  Then, $\tilde v\star\ge v\star$ ($\tilde\rho\star\le\rho\star$) if
  $A_{\mathrm{b}}(v\star,v\L)\ge0$, and $\tilde v\star\le v\star$
  ($\tilde\rho\star\ge\rho\star$) if $\hat A_{\mathrm{b}}(v\star,v\L)\le0$.
  Furthermore,
  \begin{align}
    \label{eq:specific_volume_factor}
    \frac{v\L-\tilde v\star}{v\L-v\star}
    \;\le\;
    \begin{cases}
      \begin{aligned}
        &1 &\quad& \text{if }\hatAb(v\star,v\L)\ge 0,
        \\[0.25em]
        &1\;+\;\frac{\hatpb(v\star)-\hatpb(v\L)}{\ph\star-\ph\L}
        &\quad& \text{if }\hatAb(v\star,v\L)\le 0,
        \quad\ph\star>\ph\L.
      \end{aligned}
    \end{cases}
  \end{align}
\end{lemma}
\begin{proof}
  Subtracting \eqref{eq:rankine_hugoniot} from
  \eqref{eq:rankine_hugoniot_h} and rearranging:
  \begin{align*}
    \hateh(\tilde v\star,\ph\star) - \hateh(v\star,\ph\star)
    \,+\,\frac{1}{2}(\ph\star+\ph\L) \big(\tilde v\star-v\star\big)
    \,=\,\,\hat A_{\mathrm{b}}(v\star, v\L).
  \end{align*}
  Some further manipulations give
  \begin{align}
    \label{eq:volume_identity}
    \big(\tilde v\star-v\star\big)\,\frac{1}{2}
    \big(\underbrace{2p_\infty+\ph\star+\ph\L}_{>0}\big)
    \;+\;
    \int_{v\star}^{\tilde v\star}\!\!
    \big(\underbrace{\partial_v\hateh(\tau,\ph\star)-p_\infty}_{\ge0}\big)
    \,\mathrm{d}\tau \,=\,\,\hat A_{\mathrm{b}}(v\star, v\L),
  \end{align}
  which proves the first part of the Lemma and the first case of
  \eqref{eq:specific_volume_factor}. In order to show the second case of
  \eqref{eq:specific_volume_factor} we solve the equation for $\big(\tilde
  v\star-v\star\big)$ and substitute:
  \begin{align*}
    \frac{v\L-\tilde v\star}{v\L-v\star} = 1 \,+\,
    \frac{v\star-\tilde v\star}{v\L-v\star}
    = 1 \,+\,
    \frac{\displaystyle
    -\hat A_{\mathrm{b}}(v\star, v\L)\,+\,
    \int_{v\star}^{\tilde v\star}\!\!\!
    \partial_v\hateh(\tau,\ph\star)-p_\infty\,\mathrm{d}\tau
    }{\frac{1}{2}\big(2p_\infty+\ph\star+\ph\L\big)\,\big(v\L-v\star\big)}.
  \end{align*}
  Assuming $v\star < v\L$, a direct computation gives
  \begin{align}
    \label{eq:hatab_estimate}
    \frac{\hat A_{\mathrm{b}}(v\star, v\L)}{v\star-v\L}
    \;&=\;
    \frac{\hateb(v\star)-\hateb(v)}{v\star-v\L}
    \,+\,\frac{1}{2}\big(\hatpb(v\star)+\hatpb(v)\big)
    \\\notag
    \;&=\;
    -\hatpb(\xi) \,+\,\frac{1}{2}\big(\hatpb(v\star)+\hatpb(v)\big)
    \qquad\text{for some }\xi\in(v\star,v\L)
    \\\notag
    \;&\le\; \frac{1}{2}\big(\hatpb(v\star)-\hatpb(v)\big).
  \end{align}
  The inequality follows from the monotonicity assumption \eqref{A2}. Using $v\star\ge\tilde v\star$ and $2p_\infty+\ph\star +
  \ph\L\ge\ph\star-\ph\L>0$ then establishes the second case of
  \eqref{eq:specific_volume_factor}.
\end{proof}
The following lemma verifies that \eqref{eq:rankine_hugoniot} is
locally solvable. For
the sake of completeness, we also establish a \emph{sufficient} convexity
criterion on the pressures that ensures solvability.
\begin{proposition}
  \label{prop:curve_is_fine}
  Let $W\L$ be fixed. Then there exists a neighborhood of $(v\L,\ph\L)$ in
  which the set of solutions $(v\star,\ph\star)$ of
  \eqref{eq:rankine_hugoniot} is described by a continuously differentiable
  function $v\star=\xi(\ph\star)$. This curve can be extended as long as
  $\sh\star\ge\sh\L$ and $v\star\le v\L$ is satisfied for
  $\ph\star\ge\ph\L$ (with reversed inequalities for $\ph\star\le\ph\L$)
  and there exists an entropy $\sh^\hash \in [\sh\L, \sh\star]$ such that
  \begin{align*}
    \min_{\tau\in[v\star,v\L]}
    \partial_v\partial_v\ph\big(\tau,\sh^\hash\big)\big|_{s_{\mathrm{h}}}
    \!+ \min_{\tau\in[v\star,v\L]}\partial_v\partial_v\pb(\tau) \,>\,0.
  \end{align*}
  Note, that this is a sufficient criterion for the solvability of
  \eqref{eq:rankine_hugoniot}.
\end{proposition}
\begin{proof}
  We use the implicit function theorem to characterize the curve described
  by \eqref{eq:rankine_hugoniot}. To this end, denote by
  $\mH(v\star,\ph\star)$ the left-hand side of \eqref{eq:rankine_hugoniot},
  and let us collect and analyze the two relevant partial derivatives:
  \begin{align*}
    \partial_{v\star}\mH \;&=\;\phantom{-}
    \partial_v\eh(v\star,\ph\star)\big|_{\ph}+
    \frac{1}{2}\big(\ph\star+\ph\L\big)+\partial_{v\star}\hatAb(v\star,v\L),
    \\
    \partial_{\ph\star}\mH \;&=\;
    \phantom{-}\partial_{\ph}\eh(v\star,\ph\star)\big|_{v}+
    \frac{1}{2}\big(v\star-v\L\big).
  \end{align*}
  In order to invert \eqref{eq:rankine_hugoniot} resolving
  $v\star=\xi(\ph\star)$ we need $\partial_{v\star}\mH\neq0$ to hold true.
  We have $\partial_{v\star}\mH\ge p_\infty+\frac{1}{2}(\ph\star+\ph\L)>0$
  for $v\star=v\L$, implying that we can invert the curve at least locally
  around $(v\L,\ph\L)$. Let $(v\star,\ph\star)$ be a point on the curve and,
  assuming $\ph\star\ge\ph\L$, rearrange as follows:
  \begin{align*}
    \partial_{v\star}\mH =
    \underbrace{\partial_v\eh(v\star,\ph\star)\big|_{\ph}+\ph\L}_{\ge0}
    \,+\,
    \frac{1}{2}\big(\ph\star-\ph\L\big)+\partial_{v\star}\hatAb(v\star,v\L),
  \end{align*}
  In a second step we look for an expression similar to \eqref{eq:sign_on_ab} for the
  difference $\ph\star-\ph\L$. Heuristically speaking, we want to establish
  that the increase in pressure from $\ph\L$ to $\ph\star$ is at least as
  much as integrating $\partial_v\partial_v\hatph(v,p)\big|_{\hatsh}$ twice
  along an isentrope. By slight abuse of notation, change variables to
  $\ph(v,\sh)$ and integrate:
  \begin{align*}
    \ph&(v\star,\sh\star)-\ph(v\L,\sh\L)
    \\[0.25em]
    &=\; \ph(v\star,\sh\star)-\ph(v\star,\sh^\hash)
    \,+\,\ph(v\star,\sh^\hash)-\ph(v\L,\sh^\hash)
    \,+\,\ph(v\L,\sh^\hash)-\ph(v\L,\sh\L)
    \\&=\;
    \int_{\sh^\hash}^{\sh\star}\underbrace{\partial_{\sh}\ph(v\L,\sigma)\big|_{v}}_{\ge0}\,\mathrm{d}\sigma
    \,+\,\int_{v\L}^{v\star}\partial_v\ph(\tau,\sh^\hash)\big|_{\sh}\,\mathrm{d}\tau
    \,+\,\int_{\sh\L}^{\sh^\hash}\underbrace{\partial_{\sh}\ph(v\L,\sigma)\big|_{v}}_{\ge0}\,\mathrm{d}\sigma
    \\&\ge\;
    \int_{v\L}^{v\star}\partial_v\ph(\tau,\sh^\hash)\big|_{\sh}\,\mathrm{d}\tau
    \\&\ge\;
    \int_{v\L}^{v\star}\partial_v\ph(\tau,\sh^\hash)\big|_{\sh}\,\mathrm{d}\tau
    \,-\, \underbrace{\partial_v\ph(v\L,\sh^\hash)\big|_{\sh}\big(v\star-v\L\big)}_{\ge0}
    \\&=\;
    \int_{v\L}^{v\star}\int_{v\L}^{\tau}\partial_v\partial_v\ph(\kappa,\sh^\hash)\big|_{\sh}\,\mathrm{d}\kappa
    \,\mathrm{d}\tau
    \;=\;
    \int_{v\star}^{v\L}\big(\kappa-v\star\big)\,
    \partial_v\partial_v\ph(\kappa,\sh^\hash)\big|_{\sh}\,\mathrm{d}\kappa.
  \end{align*}
  Finally, using \eqref{eq:sign_on_ab} in the equivalent form
  \begin{align*}
    \partial_{v\star}\hatAb(v\star,v\L)=\frac{1}{2}
    \int_{v\star}^{v\L}(v\L-\kappa)\,\partial_v\partial_v\hatpb(\kappa)\,
    \mathrm{d}\kappa
  \end{align*}
  establishes
  \begin{align*}
    \partial_{v\star}\mH \;&\ge\;
    \frac{1}{2}\int_{v\star}^{v\L}\Big[\big(\kappa-v\star\big)\,
    \partial_v\partial_v\ph(\kappa,\sh^\hash)\big|_{\sh}
    \,+\,\big(v\L-\kappa\big)\,\partial_v\partial_v\pb(\kappa)\Big]
    \,\mathrm{d}\kappa
    \\[0.25em]
    \;&\ge\;
    \frac{1}{4}\big(v\L-v\star\big)^2\,\Big[\min_{\tau\in[v\star,v\L]}
    \partial_v\partial_v\ph(\tau,\sh^\hash)\big|_{\sh}
    \,+\,\min_{\tau\in[v\star,v\L]} \partial_v\partial_v\pb(\tau)\Big]
    \;>\;0.
  \end{align*}
\end{proof}
A more careful analysis of the Hugoniot curve reveals that the assumptions
on specific volume and entropy can be dropped. We summarize:
\begin{corollary}
  \label{cor:curve_is_fine}
  The result of Proposition~\ref{prop:curve_is_fine} already holds true by
  just assuming that
  \vspace{-1em}
  \begin{align*}
    \min_{\tau\in[v\star,v\L]}
    \partial_v\partial_v\ph\big(\tau,\sh\star\big)\big|_{s_{\mathrm{h}}}
    \!+ \min_{\tau\in[v\star,v\L]}\partial_v\partial_v\pb(\tau) \,>\,0
  \end{align*}
  is satisfied along the Hugoniot curve $v\star=\xi(\ph\star)$.
\end{corollary}

\begin{proof}
  By assumption there exists a region in the $(v,\ph)$ plane where the total
  pressure $p$ is convex. Then, by virtue of Lemma~\ref{lem:no_weird_stuff}
  and by adjusting a result by Weyl on the monotonicity of the Hugoniot
  curve to our setting of a compound equation of state, \cite[Sect.~3,
  Thm.~1, eq.\,(12), and Sect.~4, Thms.~2 and 3]{Weyl1949} (see also
  \cite[Sect.~65]{CourantFriedrichs1977} and \cite[Prop.\,2.1 and
  Cor.\,2.4]{Smith1979}) we know that $\sh\star\ge\sh\L$ and $v\star\le
  v\L$ is already satisfied for $\ph\star\ge\ph\L$ (with reversed
  inequalities for $\ph\star\le\ph\L$). We can then refine the construction
  used in the proof of Proposition~\ref{prop:curve_is_fine} by integrating
  over a series of isentropic and isovolume ``step functions'' along the
  Hugoniot curve that stay in the region of convex total pressure. A
  compactness argument establishes that such a procedure is possible in
  finite steps.
\end{proof}
We conclude the discussion on the solvability of
\eqref{eq:rankine_hugoniot} with a computable, purely algebraic solvability
criterion that does not rely on any convexity assumptions on the (partial)
pressures.
\begin{lemma}[Solvability criterion for \eqref{eq:rankine_hugoniot}]
  \label{lem:hugoniot_solvable}
  Let $W\L$ be a state and let $I$ be a closed interval containing its
  specific volume $v\L\in I$. Set $\overbar v:=\text{len}(I)$ and introduce
  \begin{align*}
    \begin{cases}
      \begin{aligned}
        m&:=\max_{\tau\in I} \Big(\frac{\hatab^2(\tau)}{\tau^2}\Big)
               -\min_{\tau\in I} \Big(\frac{\hatab^2(\tau)}{\tau^2}\Big)
          =\max_{\tau\in I} \big(\partial_v\hatpb(\tau)\big)
          -\min_{\tau\in I} \big(\partial_v\hatpb(\tau)\big),
        \\[0.5em]
        d\,&:=\,2\big(p_\infty+\ph\L\big).
      \end{aligned}
    \end{cases}
  \end{align*}
  If the smallness condition $2\overbar vm < d$ holds true, then
  \eqref{eq:rankine_hugoniot} is uniquely solvable for $v\star$ on $I$. In
  addition the monotonicity properties of
  Proposition~\ref{prop:curve_is_fine} hold true.
\end{lemma}
\begin{proof}
  Denote again by $\mH(v\star,\ph\star)$ the left-hand side of
  \eqref{eq:rankine_hugoniot}. The fourth line of
  \eqref{eq:sign_on_ab_proof} reads
  \vspace{-0.5em}
  \begin{align*}
    \frac{\mathrm d}{\mathrm d v\star}\hat A_{\mathrm{b}}(v\star, v\L)
    \,&=\,
    \frac{1}{2} \int_{v\star}^{v\L}\big(\partial_v\hatpb(\tau)-
    \partial_v\hatpb(v\star)\big)\mathrm{d}\tau
  \end{align*}
  and since the integrand is bounded by $m$ in absolute value we obtain,
  \begin{align*}
    \partial_{v\star}\hatAb(v\star,v\L)\;\ge\;
    -\frac{1}{2}\,m\,\big|v\L-v\star\big|
    \;\ge\;-\frac{1}{2}\,m\,\overbar v\;\ge\;-\frac{1}{4}\,d
    \qquad\text{for }v\star\in I.
  \end{align*}
  This implies
  \begin{align*}
    \partial_{v\star}\mH
    \;&=\;
    \partial_v\hateh(v\star,\ph\star)\big|_{\ph}
    \,+\,\frac{1}{2}\big(\ph\star+\ph\L\big)
    \,+\,\partial_{v\star}\hatAb(v\star,v\L)
    \\[0.25em]
    \;&\ge\;
    \frac{1}{2}\big(2p_\infty+\ph\star+\ph\L\big)\,-\,\frac{1}{4}\,d
    \;\ge\;\frac{1}{2}\big(p_\infty+\ph\L\big)\,-\,\frac{1}{4}\,d
    \;>\;0
    \quad\text{for }v\star\in I.
  \end{align*}
  The remainder of the argument is identical to the proof of
  Proposition~\ref{prop:curve_is_fine}.
\end{proof}

With these preparations at hand we relate the shock branches of the
pressure functions \eqref{eq:pressure_function} and
\eqref{eq:pressure_function_h} to each other. Recall that
\begin{align}
  \label{eq:f_shock}
  u\star\;=\;u\L- f\S(\ph\star;W\L),
  \quad\text{with }\; f\S(\ph\star;W\L)\;=\;\big(p\star-p\L\big)^{1/2}
  \big(v\L-v\star\big)^{1/2},
\end{align}
and
\begin{align}
  \label{eq:f_shock_h}
  f_{\mathrm{h}}\S(\ph\star;W\L)\;=\;\big(\ph\star-\ph\L\big)^{1/2}
  \big(v\L-\tilde v\star\big)^{1/2},
\end{align}
where $\rho\star$ and $\tilde\rho\star$ are determined by
\eqref{eq:rankine_hugoniot} and \eqref{eq:rankine_hugoniot_h},
respectively. Here, by slight abuse of notation, we parametrize $f\S$ by
the hydrodynamical star pressure $\ph\star$ instead of the total star
pressure $p\star$ used in \eqref{eq:pressure_function}; the two determine
each other through $p\star=\ph\star+\pb(\rho\star)$. Note that
\eqref{eq:f_shock_h} is well-defined because $p\star\ge p\L$ implies
$\ph\star\ge\ph\L$ by virtue of
Lemma~\ref{lem:no_weird_stuff}.
\begin{proposition}
  \label{prop:shock_function}
  We have
  \begin{align*}
    f\S(\ph\star;W\L)
    \;&=\;
    \Bigg[1\;+\frac{\hatpb(v\star)-\hatpb(v\L)}{\ph\star-\ph\L}\Bigg]^{1/2}
    \Bigg[\frac{v\L-v\star}{v\L-\tilde v\star}\Bigg]^{1/2}
    f_{\mathrm{h}}\S(\ph\star;W\L)
    \\[0.50em]
    &\ge\,f_{\mathrm{h}}\S(\ph\star;W\L).
  \end{align*}
\end{proposition}

\begin{proof}
  The equality is a direct consequence of the definition of both pressure
  functions. The inequality is established by estimating the second term
  from below with the help of Lemma~\ref{lem:inequality_on_density}.
\end{proof}
A short remark is in order.
\begin{remark}
  \label{rem:no_uniqueness}
  The assumptions made in Proposition~\ref{prop:curve_is_fine} are not
  sufficient to establish a stronger variant of
  Proposition~\ref{prop:shock_function}, namely, that for any given left
  state $W\L$ and any $\Delta\in\mathbb{R}^+_0$ the equation
  $\Delta=f\S(\ph\star;W\L)$ admits at most one solution. For this to hold
  true, additional properties have to be verified for the composite
  equation of state given by $p=\ph(\rho,\eh)+\pb(\rho)$; see
  \cite{Smith1979}.
\end{remark}

In the same manner we relate the shock branches of the wavespeeds
\eqref{eq:lambda_one} and \eqref{eq:lambda_one_h} to each other. Recall
that
\begin{align}
  \label{eq:Q_shock}
  \lambda^-_1\;=\;u\L-v\L Q(\ph\star;W\L);
  \quad&\text{with }\; Q(\ph\star;W)\;=\;
  \frac{\big(p\star-p\L\big)^{1/2}}{\big(v\L-v\star\big)^{1/2}},
  \quad\text{and}
  \\
  \label{eq:Q_shock_h}
  \lambda^-_{1,\mathrm h}\;=\;u\L-v\L Q_{\mathrm h}(\ph\star;W\L);
  \quad&\text{with }\;
  Q_{\mathrm h}(\ph\star;W\L)\;=\;
  \frac{\big(\ph\star-\ph\L\big)^{1/2}}{\big(v\L-\tilde v\star\big)^{1/2}},
\end{align}
where $v\star$ and $\tilde v\star$ are again determined by
\eqref{eq:rankine_hugoniot} and \eqref{eq:rankine_hugoniot_h},
respectively, and where the mass flux $Q$ of \eqref{eq:mass_flux} is
parametrized by $\ph\star$.
\begin{proposition}
  \label{prop:shock_wavespeed}
  We have
  \begin{align}
    \label{eq:Q_shock_identity}
    \big(\lambda^-_1-u\L\big)^2
    \;&=\;
    \Bigg[\frac{v\L-\tilde v\star} {v\L-v\star}\Bigg]
    \big(\lambda^-_{1,\mathrm{h}}-u\L\big)^2
    \;+\;
    (v\L)^2\, \frac{\hatpb(v\star)-\hatpb(v\L)}{v\L-v\star}.
  \end{align}
  For the special case of $\hatpb(v)$ being purely concave on $(v\star,v\L)$
  the equation admits the following lower bound:
  \begin{align*}
    \lambda^-_1 \;\ge\;
    u\L -\sqrt{\big(u\L-\lambda^-_{1,\mathrm{h}}\big)^2+\hatab^2(v\L)}
    \quad\text{($\hatpb(v)$ concave)},
  \end{align*}
  and, conversely, for the special case of $\hatpb(v)$ being purely convex
  on $(v\star,v\L)$ we have
  \begin{align*}
    \lambda^-_1 \;\ge\; u\L -\sqrt{ \Big(1+\frac{\hatpb(\tilde
    v\star)-\hatpb(v\L)}{\ph\star-\ph\L}\Big)
    \big(u\L-\lambda^-_{1,\mathrm{h}}\big)^2+ \frac{(v\L)^2}{(\tilde v\star)^2}\,
    \hatab^2(\tilde v\star)}.
  \end{align*}
\end{proposition}

\begin{proof}
  Using the definitions of $Q$ and $Q_{\mathrm{h}}$ and after some algebra:
  \begin{align*}
    (v\L)^2\,Q^2(\ph\star;W)
    \;&=\;
    \Bigg[\frac{v\L-\tilde v\star} {v\L-v\star}\Bigg]
    \,(v\L)^2 Q_{\mathrm{h}}^2(\ph\star;W)
    \;+\;
    (v\L)^2 \frac{\hatpb(v\star)-\hatpb(v\L)}{v\L-v\star}.
  \end{align*}
  Equation \eqref{eq:Q_shock_identity} now follows from definitions
  \eqref{eq:Q_shock} and \eqref{eq:Q_shock_h}. The two special cases are a
  consequence of Lemma~\ref{lem:inequality_on_density} and the observation
  that we have for some $v\star\le\xi\le v\L$:
  \begin{align*}
    \frac{\hatpb(v\star)-\hatpb(v\L)}{v\L-v\star}
    \;=\; -\,\frac{\mathrm d}{\mathrm d v}\hatpb(\xi)
    \;=\; \Big(\underbrace{\frac{\mathrm d}{\mathrm d v}\hatpb(v)
    -\,\frac{\mathrm d}{\mathrm d v}\hatpb(\xi)}_{=:\text{(I)}}
    \Big) -\,\frac{\mathrm d}{\mathrm d v}\hatpb(v).
  \end{align*}
  Choosing $v=v\L$ for concave $\hatpb(v)$, we have $\text{(I)}\le0$ and
  $-\partial_v\hatpb(v\L)=(v\L)^{-2}\hatab^2(v\L)$. In case of a convex
  $\hatpb(v)$ we set $v=v\star$ and have again $\text{(I)}\le0$. The
  argument now follows from $\tilde v\star\le v\star$ by virtue of
  Lemma~\ref{lem:inequality_on_density}, and the fact that $\hatpb(\tilde
  v\star)\ge\hatpb(v\star)$.
\end{proof}
For the general case we first collect an observation.
\begin{lemma}
  \label{lem:estimate_on_quotient}
  Let $\ph\star\ge\ph\L$ be given and let $v\star\le v\L$ be a solution of
  \eqref{eq:rankine_hugoniot}. Assume that it is the only one to its right,
  i.\,e., that no $w\in(v\star,v\L]$ satisfies \eqref{eq:rankine_hugoniot}
  as well. Then,
  \begin{align}
    \label{eq:estimate_on_quotient}
    \frac{\hatpb(v\star)-\hatpb(v\L)}{v\L-v\star}
    \;\le\;
    \max_{\tilde v\star\,\le\,\tau\le v\L}\big(-\partial_v\hatpb(\tau)\big)
  \end{align}
  holds true throughout. Note that the quotient on the left is an
  expression in $v\star$, but the bound on the right uses $\tilde v\star$.
\end{lemma}
\begin{proof}
  The inequality is trivially satisfied for $\hatAb(v\star,v\L)\le0$,
  i.\,e., when $\tilde v\star\le v\star$: by the mean value theorem the
  left-hand side of \eqref{eq:estimate_on_quotient} is equal to
  $-\partial_v\hatpb(\xi)$ for some
  $\xi\in(v\star,v\L)\subset[\tilde v\star,v\L]$. Thus, assume
  $v\star<\tilde v\star$, i.\,e., $\hatAb(v\star,v\L)>0$, and suppose for
  the sake of contradiction that
  \begin{align*}
    g\;:=\;\frac{\hatpb(v\star)-\hatpb(v\L)}{v\L-v\star}
    \;>\;
    \max_{\tilde v\star\,\le\,\tau\le
    v\L}\big(-\partial_v\hatpb(\tau)\big).
  \end{align*}
  This assumption contradicts the absence of a second
  solution of \eqref{eq:rankine_hugoniot}, as we show next. Let $\ell(w)$
  denote the affine secant of $\hatpb(v)$ with slope $-g$ through the points
  $\big(v\star,\hatpb(v\star)\big)$ and $\big(v\L,\hatpb(v\L)\big)$, and
  introduce the deviation $h(w):=\hatpb(w)-\ell(w)$ satisfying
  $h(v\star)=h(v\L)=0$. For $w\in[\tilde v\star,v\L]$ we compute
  $\partial_w h(w)=\partial_v\hatpb(w)+g > 0$. Consequently,
  \begin{align}
    \label{eq:lemma_h_negative}
    h(w)\;<\;0\quad\text{for }w\in[\tilde v\star,v\L),
    \qquad
    h(v\L)=0.
  \end{align}
  Moreover, using $\partial_v\hateb(v)=-\hatpb(v)$ we rewrite Definition
  \eqref{eq:Ab} as follows:
  \begin{multline}
    \label{eq:lemma_Ab_as_h}
    \hatAb(w,v\L)
    \;=\;
    \int_w^{v\L}\hatpb(\tau)\,\mathrm{d}\tau
    \,-\,\frac{\hatpb(w)+\hatpb(v\L)}{2}\big(v\L-w\big)
    \\
    \;=\;
    \int_w^{v\L}h(\tau)\,\mathrm{d}\tau\,-\,\frac{h(w)}{2}\big(v\L-w\big).
  \end{multline}
  At this point the uniqueness assumption enters. Recall that $\mH(w;\ph\star)$
  denotes the left-hand side of \eqref{eq:rankine_hugoniot} with $v\star$
  replaced by $w$. We claim that
  \begin{align}
    \label{eq:lemma_H_positive}
    \mH(w;\ph\star)\;>\;0\qquad\text{for }w\in(v\star,v\L].
  \end{align}
  Indeed, $\mH(v\L;\ph\star)=\hateh(v\L,\ph\star)-\hateh(v\L,\ph\L)\ge0$
  holds true by assumption \eqref{A1} and $\ph\star\ge\ph\L$. If
  $\mH(\,\cdot\,;\ph\star)$ vanished at $v\L$, or vanished or attained a
  negative value at some $w\in(v\star,v\L)$, then the intermediate value
  theorem would produce a second solution of
  \eqref{eq:rankine_hugoniot} in $(v\star,v\L]$, which we have excluded.
  We write out $\mH(w;\ph\star)$ using definition
  \eqref{eq:rankine_hugoniot}:
  \begin{align*}
    \mH(w;\ph\star)
    \;=\;
    \hateh(w,\ph\star)-\hateh(v\L,\ph\L)
    \,+\,\frac{1}{2}\big(\ph\star+\ph\L\big)\big(w-v\L\big)
    \,+\,\hatAb(w,v\L).
  \end{align*}
  Subtracting \eqref{eq:rankine_hugoniot_h} evaluated at $\tilde v\star$
  from this identity leads to
  \begin{align*}
    \mH(w;\ph\star)
    \;&=\;
    \hateh(w,\ph\star)-\hateh(\tilde v\star,\ph\star)
    \,+\,\frac{1}{2}\big(\ph\star+\ph\L\big)\big(w-\tilde v\star\big)
    \,+\,\hatAb(w,v\L)
    \\
    \;&=\;
    -\int_w^{\tilde v\star}\Big(\partial_v\hateh(\tau,\ph\star)
    +\frac{1}{2}\big(\ph\star+\ph\L\big)\Big)\,\mathrm{d}\tau
    \,+\,\hatAb(w,v\L).
  \end{align*}
  Rearranging terms and using
  $\partial_v\hateh(\tau,\ph\star)\ge p_\infty$, see \eqref{A1}, we
  obtain for $w\in[v\star,\tilde v\star]$:
  \begin{align}
    \label{eq:lemma_Ab_lower_bound}
    \hatAb(w,v\L)
    \;&=\;
    \mH(w;\ph\star)
    \,+\,\int_w^{\tilde v\star}\Big(\partial_v\hateh(\tau,\ph\star)
    +\frac{1}{2}\big(\ph\star+\ph\L\big)\Big)\,\mathrm{d}\tau
    \\\notag
    \;&\ge\;
    \mH(w;\ph\star)
    \,+\,\frac{1}{2}\big(2p_\infty+\ph\star+\ph\L\big)
    \big(\tilde v\star-w\big).
  \end{align}
  Finally, set $w_0:=\sup\big\{w\in[v\star,\tilde v\star]\,:\,
  h(w)\ge0\big\}$. The set is nonempty due to $h(v\star)=0$, and
  \eqref{eq:lemma_h_negative} implies $w_0<\tilde v\star$ as well as, by
  continuity, $h(w_0)=0$. By construction, $h<0$ on
  $(w_0,\tilde v\star]$ and, by virtue of \eqref{eq:lemma_h_negative}, on
  $[\tilde v\star,v\L)$. Identity \eqref{eq:lemma_Ab_as_h} thus yields
  \vspace{-1em}
  \begin{align*}
    \hatAb(w_0,v\L)\;=\;\int_{w_0}^{v\L}h(\tau)\,\mathrm{d}\tau\;<\;0.
  \end{align*}
  On the other hand, evaluating \eqref{eq:lemma_Ab_lower_bound} at $w=w_0$
  and using \eqref{eq:lemma_H_positive} together with
  $\mH(v\star;\ph\star)=0$, $w_0<\tilde v\star$, and
  $2p_\infty+\ph\star+\ph\L>0$ shows
  $\hatAb(w_0,v\L)>0$; a contradiction. This establishes
  \eqref{eq:estimate_on_quotient}.
\end{proof}
This puts us in a position to complement
Proposition~\ref{prop:shock_wavespeed} for the general case.
\begin{proposition}
  \label{prop:shock_wavespeed_general}
  Let $\ph\star>\ph\L$ and assume that \eqref{eq:rankine_hugoniot} is
  solved uniquely by $v\star<v\L$, i.\,e., that no $w\in(v\star,v\L]$
  satisfies \eqref{eq:rankine_hugoniot} as well. Then, without imposing any
  convexity or concavity condition on $\hatpb(v)$, we have
  \begin{align}
    \label{eq:wavespeed_bound_general}
    \lambda^-_1 \;\ge\; u\L -\sqrt{ \Big(1+\frac{\hatpb(\tilde
    v\star)-\hatpb(v\L)}{\ph\star-\ph\L}\Big)
    \big(u\L-\lambda^-_{1,\mathrm{h}}\big)^2+ (v\L)^2
    \max_{\tilde v\star\,\le\,\tau\le v\L}\frac{\hatab^2(\tau)}{\tau^2}}.
  \end{align}
\end{proposition}
\begin{proof}
  We estimate both terms on the right-hand side of
  \eqref{eq:Q_shock_identity} from above. For the second term,
  Lemma~\ref{lem:estimate_on_quotient} in combination with
  $-\partial_v\hatpb(\tau)=\tau^{-2}\hatab^2(\tau)$ gives
  \begin{align*}
    \frac{\hatpb(v\star)-\hatpb(v\L)}{v\L-v\star}
    \;\le\;
    \max_{\tilde v\star\,\le\,\tau\le v\L}\frac{\hatab^2(\tau)}{\tau^2}.
  \end{align*}
  For the first term, recall that assumption \eqref{A2} implies
  $\partial_v\hatpb(v)=-v^{-2}\hatab^2(v)\le0$, i.\,e., $\hatpb(v)$ is
  non-increasing. Thus, in case $\hatAb(v\star,v\L)\ge0$ we conclude
  $\hatpb(\tilde v\star)\ge\hatpb(v\L)$ from $\tilde v\star\le v\L$ and the
  first case of \eqref{eq:specific_volume_factor} yields
  \begin{align*}
    \frac{v\L-\tilde v\star}{v\L-v\star}
    \;\le\;1\;\le\;
    1\,+\,\frac{\hatpb(\tilde v\star)-\hatpb(v\L)}{\ph\star-\ph\L}.
  \end{align*}
  In case $\hatAb(v\star,v\L)\le0$ we have $\tilde v\star\le v\star$ by
  virtue of Lemma~\ref{lem:inequality_on_density} and thus
  $\hatpb(\tilde v\star)\ge\hatpb(v\star)$, so that the second case of
  \eqref{eq:specific_volume_factor} implies the very same bound. The
  assertion now follows by substituting both estimates into
  \eqref{eq:Q_shock_identity} and taking the square root, observing
  $\lambda^-_1\le u\L$, see \eqref{eq:Q_shock}.
\end{proof}
The prefactor multiplying $(u\L-\lambda^-_{1,\mathrm{h}})^2$ in
\eqref{eq:wavespeed_bound_general} is not monotone in $\ph\star$, which is
inconvenient whenever $\ph\star$ itself is only known up to an upper bound.
Trading a factor of two for monotonicity removes the difficulty.
\begin{corollary}
  \label{cor:shock_wavespeed_monotone}
  Under the assumptions of
  Proposition~\ref{prop:shock_wavespeed_general},
  \begin{align}
    \label{eq:wavespeed_bound_monotone}
    \lambda^-_1 \;\ge\; u\L -\sqrt{\big(u\L-\lambda^-_{1,\mathrm{h}}\big)^2
    \;+\;2\,(v\L)^2
    \max_{\tilde v\star\,\le\,\tau\le v\L}\frac{\hatab^2(\tau)}{\tau^2}}.
  \end{align}
\end{corollary}
\begin{proof}
  By \eqref{eq:Q_shock_h} we have $\big(u\L-\lambda^-_{1,\mathrm
  h}\big)^2=(v\L)^2\,(\ph\star-\ph\L)\,\big(v\L-\tilde v\star\big)^{-1}$,
  so that the first term of \eqref{eq:wavespeed_bound_general} splits as
  \begin{align*}
    \Big(1+\frac{\hatpb(\tilde v\star)-\hatpb(v\L)}{\ph\star-\ph\L}\Big)
    \big(u\L-\lambda^-_{1,\mathrm{h}}\big)^2
    \;=\;
    \big(u\L-\lambda^-_{1,\mathrm{h}}\big)^2
    \,+\,(v\L)^2\,\frac{\hatpb(\tilde v\star)-\hatpb(v\L)}{v\L-\tilde v\star}.
  \end{align*}
  The mean value theorem applied to the last quotient produces
  $-\partial_v\hatpb(\xi)=\xi^{-2}\hatab^2(\xi)$ for some
  $\xi\in(\tilde v\star,v\L)$, which is bounded by the maximum in
  \eqref{eq:wavespeed_bound_monotone}. Adding the second term of
  \eqref{eq:wavespeed_bound_general} gives the factor two.
\end{proof}

\subsection{Left expansive shock}
\label{sec:left_expansive_shock}
The estimates in Propositions~\ref{prop:shock_wavespeed},
\ref{prop:shock_wavespeed_general}, and
Corollary~\ref{cor:shock_wavespeed_monotone} were derived for the case
$\ph\star\ge\ph\L$, i.\,e., for a \emph{compressive} shock. We thus collect
a number of crucial properties for the case $\ph\star\le\ph\L$ that we will
need later. We start by discussing the hydrodynamical system.
\begin{definition}[Continuation past $\ph\L$]
  \label{def:expansive_continuation} The solution $\tilde
  v\star=\xi_h(\ph\star;W\L)$ of \eqref{eq:rankine_hugoniot_h} is defined
  on both sides of $\ph\L$, see Proposition~\ref{prop:curve_is_fine}, and
  is strictly decreasing with $\xi_h(\ph\L;W\L)=v\L$. Accordingly, we
  extend \eqref{eq:f_shock_h} and \eqref{eq:Q_shock_h} to
  $\ph\star\le\ph\L$, where $\tilde v\star\ge v\L$, by setting
  \begin{align}
    \label{eq:f_shock_h_continued}
    f\S_{\mathrm h}(\ph\star;W\L) \;&:=\;-\big|\ph\star-\ph\L\big|^{1/2}\,
    \big|v\L-\tilde v\star\big|^{1/2},
    \\[0.25em]
    \label{eq:Q_shock_h_continued}
    Q_{\mathrm h}(\ph\star;W\L) \;&:=\;\bigg[\frac{\big|\ph\star-\ph\L\big|}
    {\big|v\L-\tilde v\star\big|}\bigg]^{1/2}\;\ge\;0,
  \end{align}
  which agrees with \eqref{eq:f_shock_h} and
  \eqref{eq:Q_shock_h} for $\ph\star\ge\ph\L$.
\end{definition}
With this reading $f\S_{\mathrm h}(\,\cdot\,;W\L)$ is real valued,
strictly increasing, and changes sign at $\ph\L$. We note:
\begin{lemma}[Monotonicity of the hydrodynamical mass flux]
  \label{lem:Q_monotone}
  Let $W\L$ be fixed. Let $Q_{\mathrm h}(\ph\star;W\L)$ be the
  hydrodynamical mass flux \eqref{eq:Q_shock_h}, read through its
  continuation \eqref{eq:Q_shock_h_continued} for $\ph\star\le\ph\L$ and as
  a limit for $\ph\star=\ph\L$. Then the map $\ph\star\mapsto
  v\L\,Q_{\mathrm h}(\ph\star;W\L)$ is non-decreasing along the entire
  Hugoniot curve through $W\L$, and
  \begin{align}
    \label{eq:Q_monotone}
    \begin{cases}
      \begin{aligned}
        v\L\,Q_{\mathrm h}(\ph\L;W\L)\;&=\;\ah(W\L),
        \\[0.25em]
        v\L\,Q_{\mathrm h}(\ph\star;W\L)\;&\ge\;\ah(W\L)
        &\;&\text{for }\ph\star\ge\ph\L,
        \\[0.25em]
        v\L\,Q_{\mathrm h}(\ph\star;W\L)\;&\le\;\ah(W\L)
        &\;&\text{for }\ph\star\le\ph\L.
      \end{aligned}
    \end{cases}
  \end{align}
\end{lemma}
\begin{proof}
  Under assumptions \eqref{A0} and \eqref{A1} the hydrodynamical Hugoniot
  curve through $W\L$ is convex in the $(v,\ph)$ plane
  \cite{Weyl1949,CourantFriedrichs1977,Smith1979}. The square
  $Q_{\mathrm h}^2(\ph\star;W\L)$ is the negative slope of the chord
  connecting $(v\L,\ph\L)$ to $(\tilde v\star,\ph\star)$ on this curve,
  which is non-decreasing in $\ph\star$ by convexity. As $\ph\star\to\ph\L$
  the chord degenerates into the tangent, and the Hugoniot curve touches
  the isentrope through $W\L$ there. Hence $(v\L)^2\,Q_{\mathrm
  h}^2(\ph\star;W\L)\to-(v\L)^2\,\partial_v\hatph(v\L,\eh\L)\big|_{\sh}
  =\ah^2(W\L)$ by \eqref{A0}, which is \eqref{eq:Q_monotone}; the remaining
  two lines follow from the monotonicity.
\end{proof}
We turn to the composite system.
\begin{lemma}[Expansive counterpart of Proposition~\ref{prop:shock_function}]
  \label{lem:shock_function_expansive}
  Let $\ph\star<\ph\L$, let $v\star>v\L$ solve \eqref{eq:rankine_hugoniot},
  and let $\tilde v\star=\xi_h(\ph\star;W\L)$. Let $f\S_{\mathrm h}$ be
  given by \eqref{eq:f_shock_h_continued} and set
  $f\S(\ph\star;W\L):=-\big|p\star-p\L\big|^{1/2}\big|v\L-v\star\big|^{1/2}$.
  Then the identity of Proposition~\ref{prop:shock_function} holds
  verbatim, and
  \vspace{-1em}
  \begin{align*}
    f\S(\ph\star;W\L)\;\le\;f\S_{\mathrm h}(\ph\star;W\L)\;\le\;0.
  \end{align*}
\end{lemma}
\begin{proof}
  Both quotients in the identity are non-negative, since numerator and
  denominator change sign simultaneously, and the identity itself is a
  direct consequence of the definitions. As $f\S_{\mathrm h}\le0$, the
  ordering of the pressure functions is equivalent to showing
  \vspace{-1em}
  \begin{align}
    \label{eq:expansive_volume_quotient}
    \frac{\tilde v\star-v\L}{v\star-v\L}
    \;\le\;1+\frac{\hatpb(v\star)-\hatpb(v\L)}{\ph\star-\ph\L}.
  \end{align}
  This is the second case of \eqref{eq:specific_volume_factor} with
  reversed sign: identity \eqref{eq:volume_identity} and estimate
  \eqref{eq:hatab_estimate} of Lemma~\ref{lem:inequality_on_density}
  remain valid for $v\star>v\L$ with every inequality in $v$ reversed, and
  so does the proof of Lemma~\ref{lem:inequality_on_density}.
\end{proof}
Wavespeed bounds for the expansive case follow the same route. We again
start with the general statement of Proposition~\ref{prop:shock_wavespeed}:
\begin{lemma}[Expansive counterpart of
  Corollary~\ref{cor:shock_wavespeed_monotone}]
  \label{lem:shock_wavespeed_expansive}
  Let $\ph\star<\ph\L$, let $v\star>v\L$ solve \eqref{eq:rankine_hugoniot},
  and let $\tilde v\star=\xi_h(\ph\star;W\L)$. Then identity
  \eqref{eq:Q_shock_identity} holds verbatim, and
  \begin{align}
    \label{eq:shock_wavespeed_expansive}
    \lambda^-_1 \;\ge\;
    u\L -\sqrt{\big(u\L-\lambda^-_{1,\mathrm{h}}\big)^2
    \;+\;2\,(v\L)^2
    \max_{v\L\,\le\,\tau\le v\star}\frac{\hatab^2(\tau)}{\tau^2}}.
  \end{align}
\end{lemma}
\begin{proof}
  The identity is algebraic and does not rely on an ordering of the
  states. The bound follows as in the proof of
  Corollary~\ref{cor:shock_wavespeed_monotone}, with
  \eqref{eq:expansive_volume_quotient} in place of the estimate on the
  first quotient of \eqref{eq:Q_shock_identity}.
\end{proof}
At first glance, the bound \eqref{eq:shock_wavespeed_expansive} of
Lemma~\ref{lem:shock_wavespeed_expansive} looks like a major inconvenience:
it is an expression in the star volume $v\star$, and the inequality chain
$v\L\le\tilde v\star\le v\star$ bounds $v\star$ from the wrong side. Worse,
the mechanism that allowed Lemma~\ref{lem:estimate_on_quotient} to dispense
with the star state in the compressive case is not available here.
Fortunately, the star volume can be bounded by information that is readily
at hand. Borrowing a structural assumption from
Definition~\ref{def:composite_assumptions} that we discuss for composite
waves we formulate:

\begin{lemma}[Hydrodynamical bracket for an expansive shock]
  \label{lem:expansive_hydro_bracket}
  Let $\ph\star<\ph\L$ and $v\star>v\L$ solve \eqref{eq:rankine_hugoniot},
  with velocity jump $\Delta:=-f\S(\ph\star;W\L)>0$ and total pressure jump
  $\Pi:=p\L-p\star>0$. Assume that the chord condition \eqref{A4} holds
  true at the front state, that is
  \begin{align}
    \label{eq:expansive_condition}
    v\L\,Q(p\star;W\L)\;\ge\;a(W\L)\;:=\;\sqrt{\ah^2(W\L)+\ab^2(\rho\L)}.
  \end{align}
  Let $q\le\ph\L$ and set $\underbar v:=\xi_h(q;W\L)$ by virtue of
  Definition~\ref{def:expansive_continuation}. Then,
  $v\star\le\underbar v$, provided that
  \begin{align*}
    -f\S_{\mathrm h}(q;W\L)\ge\Delta,
    \quad\text{or}\quad
    \ph\L-q\ge\Pi
    \quad\text{holds true.}
  \end{align*}
\end{lemma}
\begin{proof}
  With the sign convention of Lemma~\ref{lem:shock_function_expansive},
  the shock function \eqref{eq:f_shock} and the mass flux
  \eqref{eq:Q_shock} of the expansive shock read
  \begin{align*}
    \Delta=\Pi^{1/2}\big(v\star-v\L\big)^{1/2},
    \quad
    Q:= Q(p\star;W\L)\,=\, \frac{\Pi^{1/2}}{(v\star-v\L)^{1/2}},
    \quad
    v\star-v\L=\frac{\Delta}{Q}=\frac{\Pi}{Q^2}.
  \end{align*}
  The chord condition \eqref{eq:expansive_condition} bounds $Q$ from below:
  \begin{align*}
    v\star-v\L\;=\;\frac{\Delta}{Q}\;\le\;\frac{v\L}{a(W\L)}\,\Delta,
    \qquad
    v\star-v\L\;=\;\frac{\Pi}{Q^2}\;\le\;\frac{(v\L)^2}{a^2(W\L)}\,\Pi.
  \end{align*}
  On the hydrodynamical side, the continued shock function
  \eqref{eq:f_shock_h_continued} and mass flux
  \eqref{eq:Q_shock_h_continued} at the pressure $q\le\ph\L$ are related
  to the volume $\underbar v=\xi_h(q;W\L)\ge v\L$ in the same way,
  \begin{align*}
    -f\S_{\mathrm h}(q;W\L)\;=\;Q_{\mathrm h}(q;W\L)\,\big(\underbar v-v\L\big),
    \qquad
    \ph\L-q\;=\;Q_{\mathrm h}^2(q;W\L)\,\big(\underbar v-v\L\big),
  \end{align*}
  and \eqref{eq:Q_monotone} bounds the hydrodynamical mass flux from above,
  \begin{align*}
    v\L\,Q_{\mathrm h}(q;W\L)\;\le\;\ah(W\L)\;\le\;a(W\L).
  \end{align*}
  In the first case this yields
  \begin{align*}
    \underbar v-v\L\;=\;\frac{-f\S_{\mathrm h}(q;W\L)}{Q_{\mathrm h}(q;W\L)}
    \;\ge\;\frac{\Delta}{Q_{\mathrm h}(q;W\L)}
    \;\ge\;\frac{v\L}{\ah(W\L)}\,\Delta
    \;\ge\;\frac{v\L}{a(W\L)}\,\Delta,
  \end{align*}
  and in the second case
  \begin{align*}
    \underbar v-v\L\;=\;\frac{\ph\L-q}{Q_{\mathrm h}^2(q;W\L)}
    \;\ge\;\frac{\Pi}{Q_{\mathrm h}^2(q;W\L)}
    \;\ge\;\frac{(v\L)^2}{\ah^2(W\L)}\,\Pi
    \;\ge\;\frac{(v\L)^2}{a^2(W\L)}\,\Pi.
  \end{align*}
\end{proof}
We remark that \eqref{eq:expansive_condition} follows from \eqref{A4}
by the same tangent limit that established the first line of
\eqref{eq:Q_monotone}.

\subsection{Elementary left rarefaction}
\label{sec:left_rarefaction}
Let $W\L=[\rho\L,u\L,p\L]$ again be the left state and let
$W\star=[\rho\star,u\star,p\star]$ be a star state situated to the right of
it, this time connected to $W\L$ by a rarefaction wave, i.\,e.,
$\rho\star\le\rho\L$ and $p\star\le p\L$. In contrast to the shock case, the
solution is smooth across the wave and $\rho\star$ is therefore not
determined by a jump condition, but implicitly by the requirement that the
entropy stay constant across the fan:
\begin{align}
  \label{eq:constant_entropy}
  s_{\mathrm h}\big(\rho\star,\eh(\rho\star,\ph\star)\big)\;=\; s_{\mathrm
  h}\big(\rho\L,\eh(\rho\L,\ph\L)\big).
\end{align}
The rarefaction branches of the pressure functions
\eqref{eq:pressure_function} and \eqref{eq:pressure_function_h} can be
made explicit. They are obtained by integrating the Riemann invariant of the
left-facing characteristic field along the isentrope
\eqref{eq:constant_entropy}. Recall that
\begin{multline}
  \label{eq:f_rarefaction}
  u\star\;=\;u\L- f\R(\ph\star;W\L),
  \\
  \quad\text{with }\; f\R(\ph\star;W\L)\;=\;\int_{\rho\L}^{\rho\star}
  \frac{\sqrt{\ah^2(\rho,\eh)\,+\,\ab^2(\rho)}}{\rho}
  \Bigg|_{s_{\mathrm h}}\mathrm{d}\rho,
\end{multline}
and
\begin{align}
  \label{eq:f_rarefaction_h}
  f\R_{\mathrm{h}}(\ph\star;W\L)\;=\;\int_{\rho\L}^{\rho\star}
  \frac{\ah(\rho,\eh)}{\rho}\Bigg|_{s_{\mathrm h}}\mathrm{d}\rho,
\end{align}
where the internal energy $\eh$ is understood as a function of $\rho$ along
that isentrope, and where the full speed of sound
$\sqrt{\ah^2+\ab^2}$ is real and strictly positive by virtue of \eqref{A0}
and \eqref{A2}. In contrast to the shock case, the star density of the
reduced problem \eqref{eq:riemann_problem_h} coincides with $\rho\star$
because \eqref{eq:constant_entropy} involves hydrodynamical quantities
only; we may thus write $\rho\star$ in place of $\tilde\rho\star$ in
\eqref{eq:f_rarefaction_h}, so that the two pressure functions differ
solely in the integrand. They are ordered as follows.
\begin{proposition}
\label{prop:rarefaction}
  Let $\rho\star\le\rho\L$. Then
  \begin{align}
    \label{eq:f_rarefaction_estimate}
    f\R(\ph\star;W\L)\;\le\;f\R_{\mathrm{h}}(\ph\star;W\L)\;\le\;0,
  \end{align}
  and the head of the rarefaction fan travels at speed
  \begin{align}
    \label{eq:lambda_rarefaction}
    \lambda^-_1\;=\;u\L-\sqrt{\ah^2\big(\rho\L,\eh(\rho\L,\ph\L)\big)
    \,+\,\ab^2(\rho\L)}.
  \end{align}
\end{proposition}
\begin{proof}
  Assumption \eqref{A2} gives $\sqrt{\ah^2+\ab^2}\ge\ah>0$ pointwise along
  the isentrope. Both integrals run from $\rho\L$ to the smaller value
  $\rho\star$, so both are non-positive and the pointwise ordering of the
  two integrands reverses upon integration, which is
  \eqref{eq:f_rarefaction_estimate}. Identity
  \eqref{eq:lambda_rarefaction} is the characteristic speed
  $u-\sqrt{\ah^2+\ab^2}$ of the left-facing wave family evaluated at
  $W\L$, the state that bounds the fan to the left.
\end{proof}
\begin{corollary}
  \label{cor:compressive_rarefaction}
  Proposition~\ref{prop:rarefaction} holds true analogously for a
  \emph{compressive} rarefaction, i.\,e., for $\rho\star\ge\rho\L$: both
  integrals in \eqref{eq:f_rarefaction} and \eqref{eq:f_rarefaction_h} then
  run in the direction of increasing $\rho$, so the ordering
  \eqref{eq:f_rarefaction_estimate} reverses to $0\le f\R_{\mathrm
  h}(\ph\star;W\L)\le f\R(\ph\star;W\L)$, while
  \eqref{eq:lambda_rarefaction} is unaffected, being evaluated at $W\L$
  alone.
\end{corollary}

\subsection{Left composite wave}
\label{sec:composite_waves}
Whenever the total pressure, viz., $p=\ph(\rho,\eh)+\pb(\rho)$,
fails to be convex, the classification of Theorem~\ref{thm:riemann_fan}
breaks down and the left state $W\L$ need no longer be joined to the star
state $W\star$ by a single elementary wave. Instead, the first
characteristic family can produce a finite sequence of shocks and
rarefaction fans separated by constant states
\cite{Wendroff1972,MenikoffPlohr1989,MullerVoss2006}; see
Figure~\ref{fig:composite_wave}. We make this precise as follows.
\begin{figure}
  \centering
  \begin{tikzpicture}[scale=0.9,every node/.style={font=\small}]
    \coordinate (O) at (0,0);
    \foreach \ah/\at in {158/126, 126/88}{%
      \fill[black!12] (O) -- (\ah:4.2) arc (\ah:\at:4.2) -- cycle;
      \foreach \a in {1,...,5}{%
        \draw[thin,dashed,black!60]
          (O) -- ({\ah+(\at-\ah)*\a/6}:4.2);
      }
    }
    \draw[very thick] (O) -- (158:4.6);
    \draw[very thick] (O) -- (126:4.6);
    \draw[very thick] (O) -- (88:4.6);
    \draw[very thick,-{Stealth}] (-6.4,0) -- (2.8,0)
      node[below left] {$x$};
    \node at (169:5.1) {$W\L=W^{(0)}$};
    \node at (152:5.05) {$W^{(1)}$};
    \node at (131:5.05) {$W^{(2)}$};
    \node at (120:5.05) {$W^{(3)}$};
    \node at (93.5:5.0) {$W^{(4)}$};
    \node at (74:5.2) {$W^{(5)}=W\star$};
  \end{tikzpicture}
  \caption{Schematic representation of a left-facing compressive composite
  wave in the $(x,t)$-plane consisting of three shocks (thick lines) and
  two attached rarefaction fans (shaded, with dashed characteristics).
  Adjacent waves touch: each shock travels at the characteristic speed of
  the neighboring fan edge, and the intermediate states $W^{(k)}$ of
  Definition~\ref{def:composite_wave} are attained along the attachment
  rays.}
  \label{fig:composite_wave}
\end{figure}
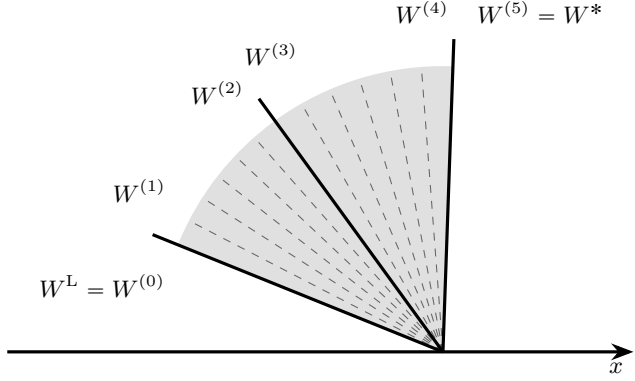
\begin{definition}[left composite wave]
  \label{def:composite_wave}
  A \emph{left composite wave} connecting $W\L$ to $W\star$ is a finite
  sequence of states
  \begin{align*}
    W\L\;=\;W^{(0)},\;W^{(1)},\;\ldots,\;W^{(n)}\;=\;W\star,
    \qquad n\ge1,
  \end{align*}
  such that for every $k$ the states $W^{(k-1)}$ and $W^{(k)}$ are
  connected by an elementary left-facing wave, that is, either by a shock
  satisfying \eqref{eq:rankine_hugoniot} or by a rarefaction fan along the
  isentrope \eqref{eq:constant_entropy}, and such that the speeds of
  consecutive elementary waves are ordered so that the resulting
  self-similar profile is well defined; see Figure~\ref{fig:composite_wave}.
  We call the composite wave \emph{compressive} if
  \begin{align}
    \label{eq:composite_compressive}
    \ph\L\;=\;\ph^{(0)}\le\;\ldots\;\le\ph^{(n)}\;=\;\ph\star,
    \qquad
    v\L\;=\;v^{(0)}\ge \;\ldots\;\ge v^{(n)}\;=\;v\star.
  \end{align}
  Conversely, we call the composite wave \emph{expansive} if the
  inequalities in \eqref{eq:composite_compressive} are inverted, viz.,
  \begin{align}
    \label{eq:composite_expansive}
    \ph\L\;=\;\ph^{(0)}\ge\;\ldots\;\ge\ph^{(n)}\;=\;\ph\star,
    \qquad
    v\L\;=\;v^{(0)}\le \;\ldots\;\le v^{(n)}\;=\;v\star.
  \end{align}
\end{definition}
\begin{remark}
  \label{rem:composite_wave}
  Two observations on Definition~\ref{def:composite_wave} are in order.
  First, telescoping the velocity jumps \eqref{eq:f_shock} and
  \eqref{eq:f_rarefaction} across the individual waves shows that the
  pressure function of a composite wave is additive \cite{MullerVoss2006},
  \begin{align}
    \label{eq:f_composite}
    u\star\;=\;u\L-f^{\mathrm{C}}(\ph\star;W\L),
    \quad\text{with }\;
    f^{\mathrm{C}}(\ph\star;W\L)\;:=\;\sum_{k=1}^{n}
    f^{(k)}\big(\ph^{(k)};W^{(k-1)}\big),
  \end{align}
  where $f^{(k)}$ denotes $f\S$ or $f\R$ according to the type of the
  $k$-th elementary wave.

  Second, a rarefaction fan occurring in a compressive composite wave is
  itself compressive: the density \emph{increases} across it. A fan
  requires the characteristic speed $u-\sqrt{\ah^2+\ab^2}$ to increase from
  head to tail, and in a region where
  $\partial_v\partial_v\hatph|_{\sh}+\partial_v\partial_v\hatpb<0$, i.\,e.,
  where the fundamental derivative of the composite equation of state in
  the sense of \cite{MenikoffPlohr1989} is negative, the speed of sound
  drops as the density grows, so that a compression can spread out
  \cite{Wendroff1972,KulikovskiiPogorelovSemenov2000}.
  This is the setting of
  Corollary~\ref{cor:compressive_rarefaction}.
  By the same token, a shock occurring within an \emph{expansive} composite
  wave, Definition~\ref{def:composite_wave}, is itself expansive: the
  density \emph{decreases} across it. In a region
  where $\partial_v\partial_v\hatph|_{\sh}+\partial_v\partial_v\hatpb<0$
  the Rankine-Hugoniot condition \eqref{eq:rankine_hugoniot} may be
  satisfied by such a nonclassical shock connecting states with
  $\ph^+\le\ph^-$ and $v^+\ge v^-$ \cite{Wendroff1972, MenikoffPlohr1989,
  KulikovskiiPogorelovSemenov2000}; this is the setting of
  Section~\ref{sec:left_expansive_shock}.
\end{remark}
\begin{definition}[Assumptions on left composite waves]
  \label{def:composite_assumptions}
  We assume that a composite wave satisfies the
  following three structural properties \cite{Wendroff1972,
  MenikoffPlohr1989, KulikovskiiPogorelovSemenov2000, MullerVoss2006}.
  For simplicity we state them for a left wave.
  First, the fundamental derivative \cite{MenikoffPlohr1989,Thompson1971} of
  the composite equation of state \eqref{eq:composite_eos}, viz.,
  \begin{align}
    \tag{A.\,3}\label{A3}
    \mathcal{G}\,:=\,\frac{v^3\,\big(\partial_v\partial_v\hatph(v,\eh)
    \big|_{\sh}\,+\,\partial_v\partial_v\hatpb(v)\big)}
    {2\,\big(\ah^2+\ab^2\big)}\,\ne\,0
    \quad\text{along each fan of the wave.}
  \end{align}
  For the next two conditions let $\lambda(W)$ denote the
  characteristic speed of a fan with left state $W$, and let
  $\sigma(W^-,W^+)$ denote the speed of a left-facing discontinuity
  connecting $W^-$ and $W^+$:
  \begin{align*}
    \lambda(W)\,:=\,u-\sqrt{\ah^2\big(\rho,\eh\big)+\ab^2(\rho)},
    \qquad
    \sigma\big(W^-,W^+\big)\,:=\,u^--v^-\,Q\big(p^+;W^-\big).
  \end{align*}
  The second condition is known as \emph{Ole{\u\i}nik--Liu chord condition}
  \cite{Oleinik1954,Oleinik1957,Oleinik1959,Liu1975,Liu1976}, which is a
  generalization of the Lax entropy condition: Assume that $W^{(k-1)}$ and
  $W^{(k)}$ are connected by a shock. Then, for every $W$ on the Hugoniot
  locus between $W^{(k-1)}$ and $W^{(k)}$:
  \begin{align}
    \tag{A.\,4}\label{A4}
    \sigma\big(W^{(k-1)},W\big)\;\ge\;\sigma\big(W^{(k-1)},W^{(k)}\big).
  \end{align}
  Finally, adjacent elementary waves have to touch, that is to say
  \begin{align}
    \tag{A.\,5}\label{A5}
    \sigma\big(W^{(k-1)},W^{(k)}\big)\;=\;\lambda\big(W^{(k)}\big)
    \quad\text{resp.}\quad
    \sigma\big(W^{(k)},W^{(k+1)}\big)\;=\;\lambda\big(W^{(k)}\big),
  \end{align}
  whenever the $k$-th wave is a shock followed by a fan, resp. a fan
  followed by a shock.

  A consequence of the three conditions is the fact that a composite wave
  is either compressive or expansive, and (using the fact that $\eb(\rho)$
  is three times continuously differentiable) that shocks and rarefaction
  fans have to alternate along the sequence.
\end{definition}
The following lemma covers both types of composite waves at once.
\begin{lemma}
  \label{lem:elementary_comparison}
  (a) Let $W^-$ and $W^+$ be connected by a compressive elementary left
  wave, i.\,e., $\ph^+\ge\ph^-$ and $v^+\le v^-$. Then
  \begin{align}
    \label{eq:elementary_comparison}
    f\big(\ph^+;W^-\big)\;\ge\;f_{\mathrm{h}}\big(\ph^+;W^-\big)\;\ge\;0,
  \end{align}
  where $f$ and $f_{\mathrm{h}}$ denote the pressure function of the wave
  and its hydrodynamical counterpart, \eqref{eq:f_shock} and
  \eqref{eq:f_shock_h} for a shock, \eqref{eq:f_rarefaction} and
  \eqref{eq:f_rarefaction_h} for a rarefaction fan.

  (b) Conversely, let $W^-$ and $W^+$ instead be
  connected by an \emph{expansive} elementary left wave, i.\,e.,
  $\ph^+\le\ph^-$ and $v^+\ge v^-$. For a shock,
  $f\S_{\mathrm h}(\,\cdot\,;W^-)$ and $f\S(\,\cdot\,;W^-)$ are then
  read through their continuation past $\ph^-$ given in
  Definition~\ref{def:expansive_continuation} and in
  Lemma~\ref{lem:shock_function_expansive}. Then
  \begin{align}
    \label{eq:elementary_comparison_expansive}
    f\big(\ph^+;W^-\big)\;\le\;f_{\mathrm{h}}\big(\ph^+;W^-\big)\;\le\;0.
  \end{align}
\end{lemma}
\begin{proof}
  For a shock, (a) is
  Proposition~\ref{prop:shock_function} together with
  $f\S_{\mathrm{h}}\ge0$ by \eqref{eq:f_shock_h}, and (b) is
  Lemma~\ref{lem:shock_function_expansive}. For a rarefaction fan, (a) is
  Corollary~\ref{cor:compressive_rarefaction} and (b) is
  Proposition~\ref{prop:rarefaction}.
\end{proof}
The point of \eqref{eq:elementary_comparison} and
\eqref{eq:elementary_comparison_expansive} is that every summand in
\eqref{eq:f_composite} is either non-negative in the compressive case or
non-positive in the expansive one. The leading summand is therefore
controlled by the total, which is what makes the outermost wave accessible
without any knowledge of the waves behind it.
\begin{lemma}
  \label{lem:composite_pressure_bound}
  Let $W\L$ be a given left state, let $\Delta\in\mathbb{R}$, and let $W\L$
  and $W\star$ be connected by a compressive ($\Delta\ge0$) or expansive
  ($\Delta\le0$) left composite wave in the sense of
  Definition~\ref{def:composite_wave} with
  $f^{\mathrm{C}}(\ph\star;W\L)=\Delta$. By virtue of
  Theorem~\ref{thm:riemann_fan} and
  Definition~\ref{def:expansive_continuation}, let $\tilde p_h\star$ be
  the unique solution of $\Delta=f\S_{\mathrm{h}}(\tilde p_h\star;W\L)$. If the
  outermost elementary wave is a shock, then the pressure $\ph^{(1)}$
  behind it obeys
  \begin{align}
    \label{eq:composite_pressure_bound}
    \min\big\{\ph\L,\tilde p_h\star\big\} \;\le\;\ph^{(1)}\;\le\;
    \max\big\{\ph\L,\tilde p_h\star\big\}.
  \end{align}
\end{lemma}
\begin{proof}
  By Lemma~\ref{lem:elementary_comparison} every
  summand in \eqref{eq:f_composite} has the sign of $\Delta$, so that the
  first summand is controlled by the total,
  \begin{align}
    \label{eq:composite_first_summand}
    f^{(1)}\big(\ph^{(1)};W\L\big)\;\le\;\Delta
    \quad\text{if }\Delta\ge0,
    \qquad
    f^{(1)}\big(\ph^{(1)};W\L\big)\;\ge\;\Delta
    \quad\text{if }\Delta\le0.
  \end{align}
  Since the outermost wave is a shock, $f^{(1)}=f\S$,
  and \eqref{eq:elementary_comparison}, respectively
  \eqref{eq:elementary_comparison_expansive}, turns
  \eqref{eq:composite_first_summand} into
  $f\S_{\mathrm h}(\ph^{(1)};W\L)\le\Delta=f\S_{\mathrm h}(\tilde
  p_h\star;W\L)$, respectively
  $f\S_{\mathrm h}(\ph^{(1)};W\L)\ge\Delta=f\S_{\mathrm h}(\tilde
  p_h\star;W\L)$. As $f\S_{\mathrm h}(\,\cdot\,;W\L)$ is strictly
  increasing on either side of $\ph\L$, see Theorem~\ref{thm:riemann_fan}
  and Definition~\ref{def:expansive_continuation}, we obtain
  $\ph^{(1)}\le\tilde p_h\star$ in the compressive and $\ph^{(1)}\ge\tilde
  p_h\star$ in the expansive case. The remaining inequality of
  \eqref{eq:composite_pressure_bound}, $\ph\L\le\ph^{(1)}$ respectively
  $\ph^{(1)}\le\ph\L$, is \eqref{eq:composite_compressive} respectively
  \eqref{eq:composite_expansive}.
\end{proof}
\begin{proposition}
  \label{prop:composite_wavespeed}
  Under the assumptions of Lemma~\ref{lem:composite_pressure_bound}, let
  $\tilde v_h\star$ be determined by \eqref{eq:rankine_hugoniot_h} for
  $\tilde p_h\star$, given by
  Definition~\ref{def:expansive_continuation} if $\Delta\le0$, and, in
  analogy to \eqref{eq:Q_shock_h}, set $\tilde\lambda^-_{1,\mathrm{h}}
  :=u\L-v\L\,Q_{\mathrm h}(\tilde p_h\star;W\L)$. Then the leftmost
  wavespeed $\lambda^-_1$ of the composite wave is bounded from below by
  \begin{align}
    \label{eq:composite_wavespeed_bound}
    \lambda^-_1\ge
    \begin{cases}
      \begin{aligned}
        &u\L-\sqrt{\ah^2\big(\rho\L,\eh(\rho\L,\ph\L)\big)
        \,+\,\ab^2(\rho\L)}
        && \text{if }f^{(1)}=f\R,
        \\[0.25em]
        &u\L-\sqrt{\big(u\L-\tilde\lambda^-_{1,\mathrm{h}}\big)^2
        +2(v\L)^2\!\!\max_{\tilde v_h\star\le\tau\le v\L}\!
        \frac{\hatab^2(\tau)}{\tau^2}}
        && \text{if }f^{(1)}=f\S,\,\Delta\ge0,
        \\[0.25em]
        &u\L-\sqrt{\ah^2(W\L)
        +2(v\L)^2\!\!\max_{v\L\le\tau\le\tilde v_h\star}\!
        \frac{\hatab^2(\tau)}{\tau^2}}
        && \text{if }f^{(1)}=f\S,\,\Delta\le0.
      \end{aligned}
    \end{cases}
  \end{align}
\end{proposition}
\begin{proof}
  If the outermost wave is a rarefaction fan, its
  head is located at $W\L$, and \eqref{eq:lambda_rarefaction} holds by
  Proposition~\ref{prop:rarefaction} in the expansive and by
  Corollary~\ref{cor:compressive_rarefaction} in the compressive case;
  this is the first case of \eqref{eq:composite_wavespeed_bound}. Let the
  outermost wave be a shock, and let $\tilde v^{(1)}:=\xi_h(\ph^{(1)};W\L)$
  be determined by \eqref{eq:rankine_hugoniot_h} for $\ph^{(1)}$.

  If $\Delta\ge0$, the shock is compressive, and
  Corollary~\ref{cor:shock_wavespeed_monotone} applied to the pair $W\L$,
  $W^{(1)}$ yields, in squared form and using $\lambda^-_1\le u\L$ and
  \eqref{eq:Q_shock_h},
  \begin{align*}
    \big(u\L-\lambda^-_1\big)^2
    \;\le\;
    (v\L)^2\,Q_{\mathrm h}^2\big(\ph^{(1)};W\L\big)
    \;+\;2\,(v\L)^2\!\!\max_{\tilde v^{(1)}\le\tau\le v\L}\!
    \frac{\hatab^2(\tau)}{\tau^2}.
  \end{align*}
  Both terms are non-decreasing in $\ph^{(1)}$: the
  first by Lemma~\ref{lem:Q_monotone}, the second because $\tilde v^{(1)}$
  decreases with $\ph^{(1)}$, see Proposition~\ref{prop:curve_is_fine}.
  Evaluating both at the upper bound $\ph^{(1)}\le\tilde p_h\star$ of
  \eqref{eq:composite_pressure_bound} gives the second case of
  \eqref{eq:composite_wavespeed_bound}.
  If $\Delta\le0$, the shock is expansive, and
  Lemma~\ref{lem:shock_wavespeed_expansive} applied to it, with
  $\ph^{(1)}$ and $v^{(1)}$ in place of $\ph\star$ and $v\star$, yields
  \begin{align*}
    \big(u\L-\lambda^-_1\big)^2
    \;\le\;
    (v\L)^2\,Q_{\mathrm h}^2\big(\ph^{(1)};W\L\big)
    \;+\;2\,(v\L)^2\!\!\max_{v\L\le\tau\le v^{(1)}}\!
    \frac{\hatab^2(\tau)}{\tau^2}.
  \end{align*}
  Here, the first term is bounded by $\ah^2(W\L)$ by virtue of
  \eqref{eq:Q_monotone}, as $\ph^{(1)}\le\ph\L$. For the second term, the
  velocity jump $\Delta_1:=-f\S(\ph^{(1)};W\L)$ across the shock obeys
  $0\le\Delta_1\le-\Delta=-f\S_{\mathrm h}(\tilde p_h\star;W\L)$ by
  \eqref{eq:composite_first_summand}, so that
  Lemma~\ref{lem:expansive_hydro_bracket} with $q=\tilde p_h\star$ gives
  $v^{(1)}\le\tilde v_h\star$. This is the third case of
  \eqref{eq:composite_wavespeed_bound}.
\end{proof}


\section{Matching conditions}
\label{sec:matching}

Let $W\L$ and $W\R$ be given initial states and adopt the parametrization
by the hydrodynamical star pressure of Section~\ref{sec:structural}. The
star states of the full Riemann problem \eqref{eq:riemann_problem} are
characterized by a \emph{pair} $\big(\ph\starL,\ph\starR\big)$ of
hydrodynamical star pressures. They satisfy the matching conditions
\begin{align}
  \label{eq:matching}
  \begin{cases}
    \begin{aligned}
    f\big(\ph\starL;W\L\big)\,+\,f\big(\ph\starR;W\R\big)
    \,+\,u\R-u\L\;=\;0,
    \\[0.5em]
    \ph\starL\,+\,\hatpb\big(v\starL\big)
    \;=\;
    \ph\starR\,+\,\hatpb\big(v\starR\big)\;=\;p\star,
    \end{aligned}
  \end{cases}
\end{align}
consisting of the pressure equation \eqref{eq:pressure_equation} and an
algebraic relation expressing continuity of the total pressure $p\star$
across the contact discontinuity. Here, $v\starL$ and $v\starR$ are
determined by $\ph\starL$ and $\ph\starR$ through the appropriate matching
condition depending on the wave type.

\subsection{General bounds}
In this subsection we establish wavespeed bounds on $\lambda_1^{-}$ and
$\lambda_3^{+}$ that do not require solving \eqref{eq:matching}, at the expense of a
slightly larger wavespeed bound. Throughout, the outer waves are assumed to
satisfy \eqref{A3} to \eqref{A5} of
Definition~\ref{def:composite_assumptions}.
\begin{proposition}[Compressive--expansive configuration]
  \label{prop:matching_compression_expansion}
  Let $\ph\starL$ and $\ph\starR$ be given by \eqref{eq:matching}. Assume
  that the left wave is compressive in the sense of
  \eqref{eq:composite_compressive}, and that the right wave is an expansive
  composite wave in the sense of \eqref{eq:composite_expansive}. Set
  \begin{align}
    \label{eq:matching_pressure_bound}
    \begin{cases}
      \begin{aligned}
        \barph\starL\;&:=\;p\R\,-\,\hatpb(v\L)
        \;=\;\ph\R\,+\,\hatpb(v\R)\,-\,\hatpb(v\L),
        \\[0.5em]
        \underph\starR\;&:=\;p\L\,-\,\hatpb(v\R)\;=\;
        \ph\L\,+\,\hatpb(v\L)\,-\,\hatpb(v\R),
      \end{aligned}
    \end{cases}
  \end{align}
  let $\bar v\starL$ and $\underbar v\starR$ be determined by
  \eqref{eq:rankine_hugoniot_h} for $\barph\starL$ and $W\L$, and for
  $\underph\starR$ and $W\R$, respectively, and let
  \vspace{-1em}
  \begin{align*}
    \bar\lambda^-_{1,\mathrm h}\;:=\;u\L-v\L\,
    Q_{\mathrm h}\big(\barph\starL;W\L\big).
  \end{align*}
  Then $\ph\L\le\ph\starL\le\barph\starL$ and
  $\underph\starR\le\ph\starR\le\ph\R$; if the outermost
  elementary wave of the right family is a shock, the specific volume
  $v^{(1)}$ behind it obeys $v\R\le v^{(1)}\le
  \underbar v\starR$, and the extremal wavespeeds of
  \eqref{eq:riemann_problem} obey
  \begin{align}
    \label{eq:matching_wavespeed_bound}
    \left\{\;
    \begin{aligned}
      \lambda^-_1\;&\ge\;
      \begin{cases}
        \begin{aligned}
          &u\L-\sqrt{\ah^2\big(W\L\big)
          \,+\,\ab^2(\rho\L)}
          &\;&\text{if }f^{(1)}=f\R,
          \\[0.25em]
          &u\L-\sqrt{\big(u\L-\bar\lambda^-_{1,\mathrm h}\big)^2
          \,+\,2\,(v\L)^2\!\!\max_{\bar v\star\le\tau\le v\L}\!
          \frac{\hatab^2(\tau)}{\tau^2}}
          &\;&\text{if }f^{(1)}=f\S,
        \end{aligned}
      \end{cases}
      \\[0.75em]
      \lambda^+_3\;&\le\;
      \begin{cases}
        \begin{aligned}
          &u\R+\sqrt{\ah^2\big(W\R\big) \,+\,\ab^2(\rho\R)}
          &\;&\text{if }f^{(1)}=f\R,
          \\[0.25em]
          &u\R+\sqrt{\ah^2\big(W\R\big)
          \,+\,2\,(v\R)^2\!\!\max_{v\R\le\tau\le\underbar v\star}\!
          \frac{\hatab^2(\tau)}{\tau^2}}
          &\;&\text{if }f^{(1)}=f\S,
        \end{aligned}
      \end{cases}
    \end{aligned}
    \right.
  \end{align}
  where $f^{(1)}$ refers to the type of the outermost elementary wave of
  the respective family.
\end{proposition}
\begin{proof}
  Assumption \eqref{A2} implies $\hatpb(v)$ is non-increasing. Since the
  right wave is expansive we have $v\starR\ge v\R$ and thus
  $\hatpb(v\starR)\le\hatpb(v\R)$, so that the second equation of
  \eqref{eq:matching} yields
  \begin{align*}
    p\star\;=\;\ph\starR+\hatpb(v\starR)
    \;\le\;\ph\R+\hatpb(v\R)\;=\;p\R.
  \end{align*}
  The left wave, on the other hand, is compressive, implying $v\starL\le
  v\L$ and $\hatpb(v\starL)\ge\hatpb(v\L)$. Using the same
  equation once more,
  \begin{align*}
    \ph\starL\;=\;p\star-\hatpb(v\starL)
    \;\le\;p\R-\hatpb(v\L)\;=\;\barph\starL.
  \end{align*}
  Together with $\ph\L\le\ph\starL$, which is again
  \eqref{eq:composite_compressive}, this proves the first part of the
  assertion. Symmetrically, the same two monotonicities give
  $p\star=\ph\starL+\hatpb(v\starL)\ge \ph\L+\hatpb(v\L)=p\L$, implying,
  \begin{align*}
    \ph\starR\;=\;p\star-\hatpb(v\starR)
    \;\ge\;p\L-\hatpb(v\R)\;=\;\underph\starR,
  \end{align*}
  which, together with $\ph\starR\le\ph\R$ from
  \eqref{eq:composite_expansive}, proves the second part. In particular
  $\barph\star\ge\ph\L$,
  so that $\bar v\star$ and $\bar\lambda^-_{1,\mathrm h}$
  are well defined by virtue of
  Theorem~\ref{thm:riemann_fan}.

  For the left wavespeed, let $W^{(1)}$ denote the
  state behind the outermost elementary wave of the left family (with
  $W^{(1)}=W\starL$ if the left wave is elementary). If that wave is a
  rarefaction fan, the first case of \eqref{eq:matching_wavespeed_bound}
  is the first case of \eqref{eq:composite_wavespeed_bound}. If it is a
  shock, it is compressive by \eqref{eq:composite_compressive}, and the
  second case follows as in the proof of
  Proposition~\ref{prop:composite_wavespeed}: the right-hand side of
  Corollary~\ref{cor:shock_wavespeed_monotone} applied to the pair $W\L$,
  $W^{(1)}$ is non-decreasing in $\ph^{(1)}$ and may thus be evaluated at
  $\ph^{(1)}\le\ph\starL\le\barph\starL$. The right wavespeed is the
  mirrored argument: with $W^{(1)}$ now the state behind the outermost
  wave of the right family, a rarefaction fan gives the first case as
  before; a shock is expansive by \eqref{eq:composite_expansive}, and the
  second case follows as in the third case of
  \eqref{eq:composite_wavespeed_bound}, with $\underbar v\star$ in place of
  $\tilde v_h\star$, once $v^{(1)}\le\underbar v\star$ is known. As $\ph$
  and $\hatpb$ are non-increasing along the expansive wave, $p^{(1)}\ge
  p\star\ge p\L$, so that the total pressure jump across the shock obeys
  $\Pi=p\R-p^{(1)}\le p\R-p\L=\ph\R-\underph\starR$ by
  \eqref{eq:matching_pressure_bound}, and
  Lemma~\ref{lem:expansive_hydro_bracket} with $q=\underph\starR$ yields
  $v^{(1)}\le\underbar v\star$.
\end{proof}

For two compressive outer waves, a key observation is that both velocity
jumps entering the pressure equation are
non-negative, so that neither of them can exceed the total velocity jump
$u\L-u\R$.
\begin{proposition}[Compressive--compressive]
  \label{prop:matching_compression_compression}
  Let $\ph\starL$ and $\ph\starR$ be given by \eqref{eq:matching}. Assume
  that both the left and right wave are compressive in the sense of
  \eqref{eq:composite_compressive}. Let $\barph\starL$ and $\barph\starR$
  be the unique solutions of
  \begin{align}
    \label{eq:matching_onesided}
    f\S_{\mathrm h}\big(\barph\starL;W\L\big)\;=\;u\L-u\R\;=\;
    f\S_{\mathrm h}\big(\barph\starR;W\R\big),
  \end{align}
  let $\bar v\starL$ and $\bar v\starR$ be determined by
  \eqref{eq:rankine_hugoniot_h} for $\barph\starL$ and $W\L$, and for
  $\barph\starR$ and $W\R$, respectively, and let
  \begin{align*}
    \bar\lambda^-_{1,\mathrm h}\;:=\;u\L-v\L\,
    Q_{\mathrm h}\big(\barph\starL;W\L\big),
    \qquad
    \bar\lambda^+_{3,\mathrm h}\;:=\;u\R+v\R\,
    Q_{\mathrm h}\big(\barph\starR;W\R\big).
  \end{align*}
  Then $u\L\ge u\R$, the pressures behind the outermost elementary wave of
  the left and of the right family obey $\ph^{(1)}\le\barph\starL$ and
  $\ph^{(1)}\le\barph\starR$, respectively, and the extremal wavespeeds of
  \eqref{eq:riemann_problem} are bounded by
  \begin{align}
    \label{eq:matching_compression_compression}
    \left\{
    \begin{aligned}
      \lambda^-_1\,&\ge\,
      \begin{cases}
        \begin{aligned}
          &u\L-\sqrt{\ah^2\big(W\L\big) \,+\,\ab^2(\rho\L)}
          &\;&\text{if }f^{(1)}=f\R,
          \\[0.25em]
          &u\L-\sqrt{\big(u\L-\bar\lambda^-_{1,\mathrm h}\big)^2
          \,+\,2\,(v\L)^2\!\!\max_{\bar v\starL\le\tau
          \le v\L}\!
          \frac{\hatab^2(\tau)}{\tau^2}}
          &\;&\text{if }f^{(1)}=f\S,
        \end{aligned}
      \end{cases}
      \\[0.75em]
      \lambda^+_3\,&\le\,
      \begin{cases}
        \begin{aligned}
          &u\R+\sqrt{\ah^2\big(W\R\big) \,+\,\ab^2(\rho\R)}
          &\;&\text{if }f^{(1)}=f\R,
          \\[0.25em]
          &u\R+\sqrt{\big(\bar\lambda^+_{3,\mathrm h}-u\R\big)^2
          \,+\,2\,(v\R)^2\!\!\max_{\bar v\starR\le\tau
          \le v\R}\!
          \frac{\hatab^2(\tau)}{\tau^2}}
          &\;&\text{if }f^{(1)}=f\S,
        \end{aligned}
      \end{cases}
    \end{aligned}
    \right.
  \end{align}
  where the case distinction refers to the type of the outermost elementary
  wave of the respective family.
\end{proposition}
\begin{proof}
  By Lemma~\ref{lem:elementary_comparison}(a) every
  summand in the decomposition \eqref{eq:f_composite} of a compressive
  wave is non-negative. The two velocity jumps
  $\Delta\L:=f\big(\ph\starL;W\L\big)$ and
  $\Delta\R:=f\big(\ph\starR;W\R\big)$ are therefore non-negative and the
  first equation of \eqref{eq:matching} reads $\Delta\L+\Delta\R=u\L-u\R$,
  so that $u\L\ge u\R$ and $\Delta\L\le u\L-u\R$. Let $\tilde p_{\mathrm
  L}$ be the solution of $f\S_{\mathrm h}(\tilde p_{\mathrm
  L};W\L)=\Delta\L$. Since $f\S_{\mathrm h}(\,\cdot\,;W\L)$ is strictly
  increasing, \eqref{eq:matching_onesided} gives $\tilde p_{\mathrm
  L}\le\barph\starL$, and Lemma~\ref{lem:composite_pressure_bound} yields
  $\ph^{(1)}\le\tilde p_{\mathrm L}\le\barph\starL$ for the pressure
  behind the outermost elementary wave of the left family. Its wavespeed
  obeys \eqref{eq:composite_wavespeed_bound} of
  Proposition~\ref{prop:composite_wavespeed} with $\Delta=\Delta\L$ and
  $\tilde p_h\star=\tilde p_{\mathrm L}$. The first case of
  \eqref{eq:composite_wavespeed_bound} is the first case of
  \eqref{eq:matching_compression_compression}; the second case is
  non-decreasing in $\tilde p_{\mathrm L}$, by Lemma~\ref{lem:Q_monotone}
  and because $\xi_h(\,\cdot\,;W\L)$ is decreasing, see
  Proposition~\ref{prop:curve_is_fine}, and evaluating it at
  $\barph\starL$ gives the second case of
  \eqref{eq:matching_compression_compression}. The same argument applies
  to the right family.
\end{proof}

We repeat the same argument for the expansive--expansive case.

\begin{proposition}[Expansive--expansive]
  \label{prop:matching_expansion_expansion}
  Let $\ph\starL$ and $\ph\starR$ be given by \eqref{eq:matching}. Assume
  that both the left and right wave are expansive in the sense of
  \eqref{eq:composite_expansive}. Let $\underph\starL$ and $\underph\starR$
  be the unique solutions of
  \begin{align}
    \label{eq:matching_onesided_expansive}
    f\S_{\mathrm h}\big(\underph\starL;W\L\big)
    \;=\;u\L-u\R\;=\;
    f\S_{\mathrm h}\big(\underph\starR;W\R\big),
  \end{align}
  where $f\S_{\mathrm h}(\,\cdot\,;W)$ is read through its continuation
  past $\ph$, see Definition~\ref{def:expansive_continuation}, and let
  $\underbar v\starL$ and $\underbar v\starR$ be determined by
  \eqref{eq:rankine_hugoniot_h}: $\underbar
  v\starL:=\xi_h\big(\underph\starL;W\L\big)$, $\underbar
  v\starR:=\xi_h\big(\underph\starR;W\R\big)$.
  Then $u\L\le u\R$, the hydrodynamical pressures behind the outermost elementary wave
  of the left and of the right family obey
  $\underph\starL\le\ph^{(1)}\le\ph\L$ and
  $\underph\starR\le\ph^{(1)}\le\ph\R$, respectively, and if that wave is
  an expansive shock, the specific volume $v^{(1)}$ behind it obeys $v\L\le
  v^{(1)}\le\underbar v\starL$, respectively $v\R\le v^{(1)}\le\underbar
  v\starR$. The extremal wavespeeds of \eqref{eq:riemann_problem} are
  bounded by
  \begin{align}
    \label{eq:matching_expansion_expansion}
    \left\{
    \begin{aligned}
      \lambda^-_1\,&\ge\,
      \begin{cases}
        \begin{aligned}
          &u\L-\sqrt{\ah^2\big(W\L\big) \,+\,\ab^2(\rho\L)}
          &\;&\text{if }f^{(1)}=f\R,
          \\[0.25em]
          &u\L-\sqrt{\ah^2\big(W\L\big)
          \,+\,2\,(v\L)^2\!\!\max_{v\L\le\tau\le\underbar v\starL}\!
          \frac{\hatab^2(\tau)}{\tau^2}}
          &\;&\text{if }f^{(1)}=f\S,
        \end{aligned}
      \end{cases}
      \\[0.75em]
      \lambda^+_3\,&\le\,
      \begin{cases}
        \begin{aligned}
          &u\R+\sqrt{\ah^2\big(W\R\big)
          \,+\,\ab^2(\rho\R)}
          &\;&\text{if }f^{(1)}=f\R,
          \\[0.25em]
          &u\R+\sqrt{\ah^2\big(W\R\big)
          \,+\,2\,(v\R)^2\!\!\max_{v\R\le\tau\le\underbar v\starR}\!
          \frac{\hatab^2(\tau)}{\tau^2}}
          &\;&\text{if }f^{(1)}=f\S,
        \end{aligned}
      \end{cases}
    \end{aligned}
    \right.
  \end{align}
  where the case distinction refers to the type of the outermost elementary
  wave of the respective family.
\end{proposition}
\begin{proof}
  By Lemma~\ref{lem:elementary_comparison}(b) every
  summand in the decomposition \eqref{eq:f_composite} of an expansive
  wave is non-positive. The two velocity jumps
  $\Delta\L:=f\big(\ph\starL;W\L\big)$ and
  $\Delta\R:=f\big(\ph\starR;W\R\big)$ are therefore non-positive and the
  first equation of \eqref{eq:matching} reads $\Delta\L+\Delta\R=u\L-u\R$,
  so that $u\L\le u\R$ and $\Delta\L\ge u\L-u\R$. Let $\tilde p_{\mathrm
  L}$ be the solution of $f\S_{\mathrm h}(\tilde p_{\mathrm
  L};W\L)=\Delta\L$, is given by
  Definition~\ref{def:expansive_continuation}. Since $f\S_{\mathrm
  h}(\,\cdot\,;W\L)$ is strictly increasing,
  \eqref{eq:matching_onesided_expansive} gives $\tilde p_{\mathrm
  L}\ge\underph\starL$, and Lemma~\ref{lem:composite_pressure_bound} yields
  $\ph^{(1)}\ge\tilde p_{\mathrm L}\ge\underph\starL$ for the pressure
  behind the outermost elementary wave of the left family, whereas
  $\ph^{(1)}\le\ph\L$ is \eqref{eq:composite_expansive}. If that wave is a
  shock, the proof of Proposition~\ref{prop:composite_wavespeed} shows
  $v^{(1)}\le\xi_h(\tilde p_{\mathrm
  L};W\L)\le\xi_h(\underph\starL;W\L)=\underbar v\starL$, as
  $\xi_h(\,\cdot\,;W\L)$ is decreasing. The wavespeed bound is
  \eqref{eq:composite_wavespeed_bound} of
  Proposition~\ref{prop:composite_wavespeed} with $\Delta=\Delta\L$ and
  $\tilde p_h\star=\tilde p_{\mathrm L}$: its first case is the first case
  of \eqref{eq:matching_expansion_expansion}, and its third case implies
  the second case of \eqref{eq:matching_expansion_expansion}, since
  $[v\L,\xi_h(\tilde p_{\mathrm L};W\L)]\subset[v\L,\underbar v\starL]$.
  The same arguments apply to the right family.
\end{proof}
The four possible configurations of the two outer waves
are now covered, and it remains to collect the pieces:
\begin{proof}[Proof of Theorem~\ref{thm:main}]
  We discuss the left wave; the right wave is treated
  in the same way. In every configuration the shock case ($f^{(1)}=f\S$)
  of the respective estimate,
  \eqref{eq:matching_wavespeed_bound}, \eqref{eq:matching_compression_compression}
  or \eqref{eq:matching_expansion_expansion}, is an upper bound for the
  rarefaction case ($f^{(1)}=f\R$), so it suffices to dominate the shock
  case by \eqref{eq:main_wavespeed}.

  Let the left wave be compressive, so that
  Proposition~\ref{prop:matching_compression_expansion}
  or~\ref{prop:matching_compression_compression} applies. The pressure
  $\barph\starL$ of \eqref{eq:main_pressure_bound} dominates $\barph\star$
  of \eqref{eq:matching_pressure_bound}, respectively $\barph\starL$ of
  \eqref{eq:matching_onesided}, as $u\L\ge u\R$ there. Hence, by
  Lemma~\ref{lem:Q_monotone}, the hydrodynamical term of the shock case is
  dominated by that of \eqref{eq:main_wavespeed}, and $\bar v\starL\le
  v\L\le\underbar v\starL$ encloses the interval of the maximum.

  Let the left wave be expansive, so that
  Proposition~\ref{prop:matching_compression_expansion} or
  Proposition~\ref{prop:matching_expansion_expansion} applies. Then
  $\overbar q\L\ge\ph\L$ by \eqref{eq:main_onesided}, so that
  $\barph\starL\ge\ph\L$ and, by \eqref{eq:Q_monotone},
  \begin{align*}
    \big(u\L-\bar\lambda^-_{1,\mathrm h}\big)^2
    \;=\;(v\L)^2\,Q_{\mathrm h}^2\big(\barph\starL;W\L\big)
    \;\ge\;\ah^2(W\L).
  \end{align*}
  Moreover, the pressure $\underph\starL$ of \eqref{eq:main_pressure_bound}
  is dominated by the lower pressure of the respective proposition,
  $p\R-\hatpb(v\L)$ of the mirrored \eqref{eq:matching_pressure_bound},
  respectively $\underph\starL$ of \eqref{eq:matching_onesided_expansive},
  as $u\L\le u\R$ there. Since $\xi_h(\,\cdot\,;W\L)$ is decreasing,
  $\underbar v\starL$ of Theorem~\ref{thm:main} dominates the star volume
  of the proposition, and $[\bar v\starL,\underbar v\starL]$ encloses the
  interval of the maximum.
\end{proof}

\subsection{Iterative estimates for regular waves}
Proposition~\ref{prop:matching_compression_compression} holds in
considerable generality, but having to resort to the ``one-sided'' matching
condition \eqref{eq:matching_onesided} leaves something to be desired. The
bounds can be improved significantly by an iterative procedure if both
outer waves are regular waves. Here, we call an
elementary wave \emph{regular} if the fundamental derivative $\mathcal G$
of \eqref{A3} is positive along it, so that, by
Remark~\ref{rem:composite_wave}, a regular shock is compressive and a
regular rarefaction is expansive. The procedure is motivated by the
following consequence of the matching condition \eqref{eq:matching}:
\begin{lemma}[Separation of the star pressures]
  \label{lem:matching_separation}
  Let $\ph\starL$ and $\ph\starR$ be given by \eqref{eq:matching} and
  assume that both the left and the right wave are regular waves. Let
  $\tilde p_h\star$ be the solution of \eqref{eq:pressure_equation_h} and
  introduce the \emph{barotropic spread}
  \begin{align}
    \label{eq:matching_spread}
    \sigma\;:=\;\big|\hatpb\big(v\starR\big)-\hatpb\big(v\starL\big)\big|.
  \end{align}
  Then, in the case of a shock-shock configuration we have
  \begin{align}
    \label{eq:matching_separation}
    \min\big(\ph\starL,\,\ph\starR\big)\;\le\;\tilde p_h\star,
    \qquad
    \max\big(\ph\starL,\,\ph\starR\big)\;\le\;\tilde p_h\star\,+\,\sigma,
  \end{align}
  and in the case of a rarefaction-rarefaction configuration
  \begin{align}
    \label{eq:matching_separation_rr}
    \tilde p_h\star\,-\,\sigma\;\le\; \min\big(\ph\starL,\,\ph\starR\big),
    \qquad
    \tilde p_h\star\;\le\;\max\big(\ph\starL,\,\ph\starR\big).
  \end{align}
  In the case of a shock-rarefaction configuration we get
  \begin{align*}
    \min(\ph\L,\ph\R)-\sigma
    \,\le\,
    \min\big(\ph\starL,\ph\starR\big)
    \,\le\,
    \max\big(\ph\starL,\ph\starR\big)\;\le\;\max(\ph\L,\ph\R)+\sigma.
  \end{align*}
\end{lemma}
\begin{proof}
  In the shock-shock case, Proposition~\ref{prop:shock_function} bounds
  both pressure functions of the first equation of \eqref{eq:matching} from
  below by their hydrodynamical counterparts, so that $f\S_{\mathrm
  h}(\ph\starL;W\L)+f\S_{\mathrm h}(\ph\starR;W\R)\le u\L-u\R$. Both waves
  are shocks, hence $\ph\starL\ge\ph\L$ and $\ph\starR\ge\ph\R$ and the two
  summands coincide with $f_{\mathrm h}(\ph\starL;W\L)$ and $f_{\mathrm
  h}(\ph\starR;W\R)$. Evaluating both at $m:=\min(\ph\starL,\ph\starR)$
  decreases them, therefore
  \begin{align*}
    f_{\mathrm h}\big(m;W\L\big)+f_{\mathrm h}\big(m;W\R\big)
    \;\le\;u\L-u\R\;=\;
    f_{\mathrm h}\big(\tilde p_h\star;W\L\big)
    +f_{\mathrm h}\big(\tilde p_h\star;W\R\big),
  \end{align*}
  and strict monotonicity of the left-hand side in $m$ implies $m\le\tilde
  p_h\star$. The second equation of \eqref{eq:matching} reads
  $\ph\starL-\ph\starR=\hatpb(v\starR)-\hatpb(v\starL)$, so that
  $\max(\ph\starL,\ph\starR)=m+\sigma$$\;\le\tilde p_h\star+\sigma$. In the
  rarefaction-rarefaction case Proposition~\ref{prop:rarefaction} bounds
  both pressure functions from above by their hydrodynamical counterparts,
  so that the chain of inequalities holds with all signs reversed and with
  $M:=\max(\ph\starL,\ph\starR)$ in place of $m$. This gives $M\ge\tilde
  p_h\star$ and $\min(\ph\starL,\ph\starR)=M-\sigma\ge\tilde
  p_h\star-\sigma$. The pressure bounds for the shock-rarefaction case are
  already established in Proposition
  \ref{prop:matching_compression_expansion}.
\end{proof}

\begin{remark}
  Note that \eqref{eq:matching_separation} and
  \eqref{eq:matching_separation_rr} do not hold true for the case of
  composite waves. The step of establishing
  $\min\big(\ph\starL,\ph\starR\big)\le\tilde p_h\star$ fails for a
  composite compressive wave. Similarly, the step of establishing $\tilde
  p_h\star\le\max\big(\ph\starL,\ph\starR\big)$ fails for a composite
  expansive one.
\end{remark}

Much of the remainder this section is devoted to quantifying the barotropic
spread $\sigma$ in \eqref{eq:matching_spread} in a computable manner and
then incorporating this result into an iterative procedure that succesively
improves the bounds established in
Propositions~\ref{prop:matching_compression_expansion}
and~\ref{prop:matching_compression_compression}.
In order to guide the reader through a series of technical Lemmas we start
by summarizing the overall strategy how we want to achieve this. In case of
regular waves Propositions~\ref{prop:matching_compression_expansion}
and~\ref{prop:matching_compression_compression} yield bounds on the two
hydrodynamical star pressures, which we now seek to improve iteratively.
Say after the $n$th iteration we have found:
\vspace{-1em}
\begin{align*}
  \underbar P^{(n)}_{\mathrm L}\,\le\,
  \ph\starL\,\le\,\bar P^{(n)}_{\mathrm L},
  \qquad
  \underbar P^{(n)}_{\mathrm R}\,\le\,
  \ph\starR\,\le\,\bar P^{(n)}_{\mathrm R}.
\end{align*}
In Lemma~\ref{lem:inequality_on_density} we then convert the pressure
bounds into bounds on the specific star volumes:
\vspace{-1em}
\begin{align*}
  \underbar V^{(n)}_{\mathrm L}\;
  \le\;v\starL\;\le\;
  \bar V^{(n)}_{\mathrm L},
  \quad
  \underbar V^{(n)}_{\mathrm R}\;
  \le\;v\starR\;\le\;
  \bar V^{(n)}_{\mathrm R},
\end{align*}
We then convert the pressure and volume bounds into an estimate quantifying
the mismatch between the hydrodynamical and full pressure functions (Lemma
~\ref{lem:deviation_bounds}):
\begin{align*}
  \begin{cases}
    \begin{aligned}
      -\eta\R(\underbar P^{(n)}_{\mathrm L},\overbar V^{(n)}_{\mathrm L};W\L)
      \,\le\,
      f\big(\ph\starL;W\L\big) - f_{\mathrm h}\big(\ph\starL;W\L\big)
      \,\le\,
      \eta\S(\overbar P^{(n)}_{\mathrm L},\underbar V^{(n)}_{\mathrm L};W\R),
      \\[0.50em]
      -\eta\R(\underbar P^{(n)}_{\mathrm R},\overbar V^{(n)}_{\mathrm R};W\R)
      \,\le\,
      f\big(\ph\starR;W\R\big) - f_{\mathrm h}\big(\ph\starR;W\R\big)
      \,\le\,
      \eta\S(\overbar P^{(n)}_{\mathrm R},\underbar V^{(n)}_{\mathrm R};W\R).
    \end{aligned}
  \end{cases}
\end{align*}
Which in turn allows us to construct new bounds on the pressure quantifying
this mismatch and the barotropic spread; see
Corollary~\ref{cor:deviation_star_pressures} and
Proposition~\ref{prop:matching_bootstrap}:
\begin{align*}
  \underbar\pi^{\ast(n)}-\max\big(\underbar\delta^{(n)},0\big)
  \,&\le\,\ph\starL\,\le\,
  \bar\pi^{\ast(n)}+\max\big(\bar\delta^{(n)},0\big),
  \\[0.25em]
  \underbar\pi^{\ast(n)}-\max\big(\bar\delta^{(n)},0\big)
  \,&\le\,\ph\starR\,\le\,
  \bar\pi^{\ast(n)}+\max\big(\underbar\delta^{(n)},0\big).
\end{align*}
These new bounds are then used to improve the bounds on the pressure by
taking the better estimate. The construction can then be iterated as long
as the pressure bounds are improving; see
Proposition~\ref{prop:matching_bootstrap}.
The outlined iterative procedure requires various smallness assumptions on
the barotropic component to hold true. We start by establishing computable
bounds on $v\star$:
\begin{lemma}[Computable bounds on star volumes]
  \label{lem:matching_volume_bound}
  Let $W$ be a state with specific volume $v$ and hydrodynamical pressure
  $\ph$, connected to a star state by a regular shock with
  hydrodynamical star pressure $\ph\star\ge\ph$, and let $v\star$ and
  $\tilde v\star$ be given by \eqref{eq:rankine_hugoniot} and
  \eqref{eq:rankine_hugoniot_h}. Denote by $\xi_h(\,\cdot\,;W)$ the
  implicit solution function of \eqref{eq:rankine_hugoniot_h} and assume
  that it is non-increasing. Let $\underbar P\,\le\,\ph\star\le\overbar P$
  be bounds on the pressure and introduce the following quantities:
  \begin{align}
    \label{eq:matching_volume_data}
    \begin{cases}
      \begin{aligned}
        \underbar c\,&:=\,v-\xi_h(\underbar P;W),
        &\quad
        \overbar c\,&:=\,v-\xi_h(\overbar P;W),
        \\[0.25em]
        m\;&:=\!\!\max_{\xi_h(\overbar P;W)\,\le\,\tau\le v}\!
        \frac{\hatab^2(\tau)}{\tau^2},
        &\quad
        d\;&:=\;2\big(p_\infty+\ph\big).
      \end{aligned}
    \end{cases}
  \end{align}
  Assume that the hydrodynamical pressure $\ph$ is sufficiently large such
  that $4vm\le d$ holds true. Then,
  \begin{align}
    \label{eq:matching_volume_bound}
    v\star\;\ge\;\underbar\Xi(\overbar P;W):=v-
    \frac{2\,\overbar c}{1+\sqrt{1-4\,m\,\overbar c/d}}
    \qquad\qquad\text{for}\quad4vm\le d.
  \end{align}
  Without any smallness condition we have, in
  addition, the computable upper bound
  \begin{align}
    \label{eq:matching_volume_bound_upper}
    v\star\;\le\;\overbar\Xi(\underbar P;W):=v-
    \frac{2\,\underbar c}{1+\sqrt{1+4\,m\,\underbar c/d}}.
  \end{align}
  If, instead, $W$ is connected to the star state by
  a regular rarefaction, then $v\star=\xi_h(\ph\star;W)$, where
  $\xi_h(\,\cdot\,;W)$ now denotes the implicit solution function of
  \eqref{eq:constant_entropy}, so that the pressure bounds
  $\underbar P\le\ph\star\le\overbar P$ yield the
  computable bounds
  \begin{align}
    \label{eq:matching_volume_bound_rarefaction}
    \underbar\Xi(\overbar P;W)\;:=\;\xi_h(\overbar P;W)\;\le\;v\star\;\le\;\xi_h(\underbar P;W)\;=:\;\overbar\Xi(\underbar P;W),
  \end{align}
  irrespective of any further assumption on
  $\hatpb$.
\end{lemma}
\begin{proof}
  \eqref{eq:rankine_hugoniot} is uniquely solvable, as the smallness
  condition $4vm\le d$ implies the criterion of
  Lemma~\ref{lem:hugoniot_solvable}. We proceed by first establishing
  \begin{align}
    \label{eq:matching_volume_quadratic}
    \tilde v\star-v\star\;\le\;\frac{m}{d}\,\big(v-v\star\big)^2.
  \end{align}
  This is trivial for $\tilde v\star\le v\star$. Thus assume that $v\star <
  \tilde v\star$. A careful inspection of the proof of
  Lemma~\ref{lem:inequality_on_density} establishes the inequalities
  \begin{align*}
    \big(\tilde v\star-v\star\big)\,
    \underbrace{\big(2p_\infty+\ph\star+\ph\big)}_{\ge\,d}
    \;\le\;2\,\hatAb(v\star,v)\;\le\;
    \big(v-v\star\big)\, \big(\hatpb(v\star)-\hatpb(v)\big).
  \end{align*}
  The second inequality is a consequence of \eqref{eq:hatab_estimate}, but
  using $\hatpb(\xi)\le\hatpb(v\star)$ to establish a bound from below.
  Finally, Lemma~\ref{lem:estimate_on_quotient} establishes
  \begin{align*}
    \frac{\hatpb(v\star)-\hatpb(v)}{v-v\star}
    \;\le\;
    \max_{\tilde v\star\le\tau\le v}\hatab^2(\tau)/\tau^2
    \;\le\;m.
  \end{align*}
  Since $\xi_h(\overbar P;W)\le\tilde v\star$ we also have $v-\tilde v\star\le\overbar c$,
  so that \eqref{eq:matching_volume_quadratic} turns into
  \begin{align*}
    \frac{m}{d}\,\big(v-v\star\big)^2\,-\,\big(v-v\star\big)
    \,+\,\overbar c\;=:\;\delta\ge\;0.
  \end{align*}
  Solving this quadratic equation in $v-v\star$ establishes
  \begin{align*}
    v\,-\,v\star\;=\;\frac{2\,(\overbar c-\delta)}
    {1+\sqrt{1-4\,m\,(\overbar c-\delta)/d}}
    \;\le\;\frac{2\,\overbar c}
    {1+\sqrt{1-4\,m\,\overbar c/d}}
    \;=\;v\,-\,\underbar\Xi(\overbar P;W).
  \end{align*}
  The upper bound
  \eqref{eq:matching_volume_bound_upper} follows along the same lines. For
  $\hatAb(v\star,v)\ge0$ we have $v\star\le\tilde v\star\le\xi_h(\underbar
  P;W)=v-\underbar c\le\overbar\Xi(\underbar P;W)$. For
  $\hatAb(v\star,v)\le0$, in turn, $\tilde v\star\le v\star$ and the above
  chain of inequalities, applied to $-\hatAb(v\star,v)$, yields
  $v\star-\tilde v\star\le\frac{m}{d}(v-v\star)^2$. Since $\underbar c\le
  v-\tilde v\star$ this results in
  \begin{align*}
    \frac{m}{d}\,\big(v-v\star\big)^2
    \,+\,\big(v-v\star\big)\,-\,\underbar c\;\ge\;0,
  \end{align*}
  and solving this quadratic equation in $v-v\star$
  gives $v-v\star\ge2\,\underbar c\,/\,(1+\sqrt{1+4\,m\,\underbar c/d})$.
  For a regular rarefaction, in turn, the full
  and the reduced problem share the isentrope \eqref{eq:constant_entropy},
  which involves hydrodynamical quantities only, see
  Section~\ref{sec:left_rarefaction}. Hence $v\star=\xi_h(\ph\star;W)$ and
  \eqref{eq:matching_volume_bound_rarefaction} follows from the
  monotonicity of $\xi_h$.
\end{proof}
Lemma~\ref{lem:matching_volume_bound} also yields computable bounds on the
deviation of the pressure functions \eqref{eq:f_shock} and
\eqref{eq:f_rarefaction} from their hydrodynamical counterparts. The bounds
do not require any knowledge of the wave type.
\begin{lemma}[Deviation of the pressure functions]
  \label{lem:deviation_bounds}
  Let $W$ be a state with specific volume $v$ and hydrodynamical pressure
  $\ph$, connected to a star state by a regular wave with
  hydrodynamical star pressure $\ph\star$ and specific volume $v\star$. Let
  $\underbar P\le\ph\star\le\overbar P$ be bounds on the pressure, let
  $0<\underbar V\le v\star\le\overbar V$ be the volume bounds of
  Lemma~\ref{lem:matching_volume_bound}, and let $m$ and $d$ be given by
  \eqref{eq:matching_volume_data}. Introduce the barotropic pressure
  changes across the volume bracket,
  \begin{align}
    \label{eq:deviation_delta_b}
    \Delta_{\mathrm b}^+\;:=\;\max\big(\hatpb(\underbar V)-\hatpb(v),\,0\big),
    \qquad
    \Delta_{\mathrm b}^-\;:=\;\max\big(\hatpb(v)-\hatpb(\overbar V),\,0\big),
  \end{align}
  as well as
  \vspace{-1em}
  \begin{align}
    \label{eq:deviation_eta}
    \begin{cases}
      \begin{aligned}
        \eta\S(\overbar P,\underbar V;W)\;&:=\;\frac{v}{2\,\ah(W)}\,
        \frac{d\,+\,\overbar P-\ph}{d\,-\,\Delta_{\mathrm
        b}^+}\,\Delta_{\mathrm b}^+,
        \\[0.5em]
        \eta\R(\underbar P,\overbar V;W)\;&:=\;\frac{\overbar V}{2\,\ah(\overbar V,\underbar P)}\,
        \Delta_{\mathrm b}^-,
      \end{aligned}
    \end{cases}
  \end{align}
  where $\ah(\overbar V,\underbar P)$ denotes the hydrodynamical speed of
  sound at specific volume $\overbar V$ and hydrodynamical pressure
  $\underbar P$. Note that $\eta\R(\underbar P,\overbar V;W)=0$ for
  $\underbar P\ge\ph$ and correspondingly $\eta\S(\overbar P,\underbar
  V;W)=0$ for $\overbar P\le\ph$. Assume that $4vm\le d$ and
  $\Delta_{\mathrm b}^+<d$. Then,
  \begin{align}
    \label{eq:deviation_bound}
    \begin{cases}
      \begin{aligned}
        0\;&\le\;f\S(\ph\star;W)-f\S_{\mathrm h}(\ph\star;W)\;\le\;\eta\S(\overbar P,\underbar V;W)
        &\quad&\text{for a shock},
        \\[0.50em]
        0\;&\le\;f\R_{\mathrm h}(\ph\star;W)-f\R(\ph\star;W)\;\le\;\eta\R(\underbar P,\overbar V;W)
        &\quad&\text{for a rarefaction}.
      \end{aligned}
    \end{cases}
  \end{align}
\end{lemma}
\begin{proof}
  The lower bounds are Propositions~\ref{prop:shock_function}
  and~\ref{prop:rarefaction}. We first discuss the shock case. Set
  $\Delta_{\mathrm b}:=\hatpb(v\star)-\hatpb(v)$. As $\hatpb$ is
  non-increasing and $\underbar V\le v\star\le v$, we have
  $0\le\Delta_{\mathrm b}\le\Delta_{\mathrm b}^+<d$. Definitions
  \eqref{eq:f_shock} and \eqref{eq:f_shock_h} give
  \begin{align*}
    \big(f\S\big)^2-\big(f\S_{\mathrm h}\big)^2
    \;=\;(\ph\star-\ph)\big(\tilde v\star-v\star\big)
    \,+\,\Delta_{\mathrm b}\,\big(v-v\star\big).
  \end{align*}
  A close inspection of \eqref{eq:volume_identity} and
  \eqref{eq:hatab_estimate} reveals
  \begin{align*}
    \tilde v\star-v\star \,\le\,
    (v-v\star)\,\frac{\hatpb(v\star)-\hatpb(v)}{2p_\infty+\ph\star+\ph},
    \quad\text{implying}\quad
    \max(\tilde v\star-v\star,0)\le(v-v\star)\frac{\Delta_{\mathrm b}}{d}.
  \end{align*}
  We thus get
  \vspace{-1em}
  \begin{align*}
    \big(f\S\big)^2-\big(f\S_{\mathrm h}\big)^2
    \;\le\;\big(v-v\star\big)\,
    \frac{d+\ph\star-\ph}{d}\,\Delta_{\mathrm b}.
  \end{align*}
  On the other hand, the shock speed is bounded from below by the speed of
  sound, see \eqref{eq:Q_monotone} and Lemma~\ref{lem:Q_monotone}. This implies
  \begin{align*}
    f\S_{\mathrm h}\;=\;Q_{\mathrm h}\,\big(v-\tilde v\star\big)
    \;\ge\;\frac{\ah(W)}{v}\,\big(v-\tilde v\star\big)
    \;\ge\;\frac{\ah(W)}{v}\,\big(v-v\star\big)\,
    \frac{d-\Delta_{\mathrm b}}{d},
  \end{align*}
  where we have used $v-\tilde v\star\ge(v-v\star)-\max(\tilde
  v\star-v\star,0)$ in the last step. Since $f\S\ge f\S_{\mathrm h}\ge0$,
  \vspace{-1em}
  \begin{align*}
    f\S-f\S_{\mathrm h}
    \;=\;\frac{\big(f\S\big)^2-\big(f\S_{\mathrm h}\big)^2}
    {f\S+f\S_{\mathrm h}}
    \;\le\;\frac{\big(f\S\big)^2-\big(f\S_{\mathrm h}\big)^2}
    {2f\S_{\mathrm h}}
    \;\le\;\frac{v}{2\,\ah(W)}\,
    \frac{d+\ph\star-\ph}{d-\Delta_{\mathrm b}}\,\Delta_{\mathrm b}.
  \end{align*}
  The right-hand side is non-decreasing in
  $\Delta_{\mathrm b}\le\Delta_{\mathrm b}^+$ and in $\ph\star\le\overbar
  P$, which is the first bound in \eqref{eq:deviation_bound}.

  For the rarefaction case, let $\hatah(\tau)$ denote the hydrodynamical
  speed of sound along the isentrope through $W$ as a function of specific
  volume $\tau$. Using the definitions \eqref{eq:f_rarefaction} and
  \eqref{eq:f_rarefaction_h}, together with the observation that
  \begin{align*}
    \sqrt{\hatah^2(\tau)+\hatab^2(\tau)}-\hatah(\tau)\,\le\,
    \frac{\hatab^2(\tau)}{2\hatah(\tau)}
    \,=\,
    \frac{-\tau^2\,\partial_v\hatpb(\tau)}{2\hatah(\tau)}
  \end{align*}
  we calculate
  \vspace{-1em}
  \begin{align*}
    f\R_{\mathrm h}-f\R
    \;&=\;\int_v^{v\star}
    \frac{\sqrt{\hatah^2(\tau)+\hatab^2(\tau)}-\hatah(\tau)}{\tau}
    \,\mathrm d\tau
    \\
    \;&\le\;\int_v^{v\star}\big(-\partial_v\hatpb(\tau)\big)\,
    \frac{\tau}{2\,\hatah(\tau)}\,\mathrm d\tau
    \\
    \;&\le\;\frac{v\star}{2\,\hatah(v\star)}\,
    \big(\hatpb(v)-\hatpb(v\star)\big).
  \end{align*}
  where the last step uses the fact that
  $\hatah(\tau)/\tau=(-\partial_v\hatph)^{1/2}$ is non-increasing in $\tau$
  along the isentrope by \eqref{A0}. The same monotonicity,
  $v\star\le\overbar V$, and $\hatpb(v)-\hatpb(v\star)\le\Delta_{\mathrm
  b}^-$ give the second bound in \eqref{eq:deviation_bound}, since
  $\overbar V=\xi_h(\underbar P;W)$ lies on the isentrope at pressure
  $\underbar P$, see \eqref{eq:matching_volume_bound_rarefaction}, so that
  $\hatah(\overbar V)=\ah(\overbar V,\underbar P)$.
\end{proof}
\begin{corollary}[Computable bounds on star pressures]
  \label{cor:deviation_star_pressures}
  Let $\ph\starL$ and $\ph\starR$ be given by \eqref{eq:matching} and
  assume that both the left and the right wave are regular waves. Let
  the assumptions of Lemma~\ref{lem:deviation_bounds} be satisfied for
  $W\L$ with bounds $\underbar P_{\mathrm L}$, $\overbar P_{\mathrm L}$,
  $\underbar V_{\mathrm L}$, $\overbar V_{\mathrm L}$ and for $W\R$ with
  bounds $\underbar P_{\mathrm R}$, $\overbar P_{\mathrm R}$, $\underbar
  V_{\mathrm R}$, $\overbar V_{\mathrm R}$, and set
  \begin{align}
    \label{eq:deviation_eta_sum}
    \begin{cases}
      \begin{aligned}
        \underbar\eta\;&:=\;
        -\eta\S(\overbar P_{\mathrm L},\underbar V_{\mathrm L};W\L)
        -\eta\S(\overbar P_{\mathrm R},\underbar V_{\mathrm R};W\R),
        \\[0.25em]
        \bar\eta\;&:=\;
        \eta\R(\underbar P_{\mathrm L},\overbar V_{\mathrm L};W\L)
        +\eta\R(\underbar P_{\mathrm R},\overbar V_{\mathrm R};W\R).
      \end{aligned}
    \end{cases}
  \end{align}
  For $\eta\in\mathbb R$ let $\Phi(\eta;W\L,W\R)$
  denote the solution of the shifted hydrodynamical pressure equation
  \begin{align}
    \label{eq:pressure_equation_shifted}
    f_{\mathrm h}\big(\Phi(\eta;W\L,W\R);W\L\big)
    \,+\,f_{\mathrm h}\big(\Phi(\eta;W\L,W\R);W\R\big)
    \,+\,u\R-u\L\,-\,\eta\;=\;0,
  \end{align}
  so that $\Phi(0;W\L,W\R)=\tilde p_h\star$. Then,
  provided the respective shifted equation admits a solution,
  \begin{align}
    \label{eq:deviation_star_pressures}
    \Phi(\underbar\eta;W\L,W\R)\;\le\;\max\big(\ph\starL,\ph\starR\big),
    \qquad
    \min\big(\ph\starL,\ph\starR\big)\;\le\;\Phi(\bar\eta;W\L,W\R).
  \end{align}
\end{corollary}
\begin{proof}
  By Lemma~\ref{lem:deviation_bounds}, each of the
  two pressure functions in the first equation of \eqref{eq:matching}
  satisfies
  \begin{align*}
    f_{\mathrm h}\big(\ph\star;W\big)-\eta\R(\underbar P,\overbar V;W)
    \;\le\;f\big(\ph\star;W\big)\;\le\;
    f_{\mathrm h}\big(\ph\star;W\big)+\eta\S(\overbar P,\underbar V;W),
  \end{align*}
  where $f=f\S$ and $f_{\mathrm h}=f\S_{\mathrm h}$
  for a shock, $f=f\R$ and $f_{\mathrm h}=f\R_{\mathrm h}$ for a
  rarefaction fan, and where both deviation bounds are non-negative.
  Summing up both sides and using that $f_{\mathrm h}(\,\cdot\,;W)$ is
  non-decreasing we get
  \begin{align*}
    f_{\mathrm h}\big(\max(\ph\starL,\ph\starR);W\L\big)
    +f_{\mathrm h}\big(\max(\ph\starL,\ph\starR);W\R\big)
    \;&\ge\;u\L-u\R+\underbar\eta,
    \\[0.25em]
    f_{\mathrm h}\big(\min(\ph\starL,\ph\starR);W\L\big)
    +f_{\mathrm h}\big(\min(\ph\starL,\ph\starR);W\R\big)
    \;&\le\;u\L-u\R+\bar\eta.
  \end{align*}
  Comparing with \eqref{eq:pressure_equation_shifted} for $\eta=\underbar\eta$
  and $\eta=\bar\eta$, respectively, strict monotonicity of the left-hand
  side in the star pressure yields \eqref{eq:deviation_star_pressures}.
\end{proof}
As a by-product of Lemma~\ref{lem:matching_volume_bound} we can formulate a
computable criterion on the left and right waves that ensures that they are
regular.
\begin{lemma}[Regular wave criterion]
  \label{lem:elementary_criterion}
  Let a state $W$ with specific volume $v$ and specific entropy $\sh$ be
  connected to a star state $W\star$ with specific volume $v\star$ and
  specific entropy $\sh\star$ by a compressive or expansive wave. Let
  $\underbar V\le v\star\le \overbar V$ be bounds on $v\star$. Introduce
  the interval $I=[\min(v,\underbar V),\max(v,\overbar V)]$ and assume that
  \begin{align}
    \label{eq:kappa_criterion}
    \kappa:=
    \min_{\tau\in I}
    \partial_v\partial_v\ph\big(\tau,\sh^\hash\big)\big|_{s_{\mathrm{h}}}
    \!+ \min_{\tau\in I}\partial_v\partial_v\pb(\tau) \,>\,0.
  \end{align}
  where $\sh^\hash\in[\sh,\sh\star]$ is arbitrary. Then, the wave is a
  regular shock or a regular rarefaction.
\end{lemma}
\begin{proof}
  The numerator of the fundamental derivative $\mathcal G$ in \eqref{A3} is
  bounded below by $\kappa>0$ on $I$: along a fan because the specific
  entropy is constant, and along a shock curve because, by
  Proposition~\ref{prop:curve_is_fine} and
  Corollary~\ref{cor:curve_is_fine}, the Hugoniot locus stays monotone in
  $v$ and $\sh$ and thus within $I\times[\sh,\sh\star]$. Hence $\mathcal
  G>0$ throughout the wave. Were the wave composite, then shocks and fans
  would have to alternate by virtue of \eqref{A3}--\eqref{A5}, so the wave
  would contain a compressive fan or an expansive shock, both of which
  require $\mathcal G<0$, see Remark~\ref{rem:composite_wave}.
\end{proof}

The pressure and volume estimates of
Proposition~\ref{prop:matching_compression_compression} serve as a
starting point for an iterative procedure for computing successively tighter bounds on
$\sigma$.
\begin{proposition}[Iterative bounds for regular waves]
  \label{prop:matching_bootstrap} Assume that both the left and the right
  wave are regular waves and let $\ph\starL$ and
  $\ph\starR$ be given by \eqref{eq:matching}, and let $\tilde p_h\star$
  solve \eqref{eq:pressure_equation_h}.
  Denote by $\xi_h(\,\cdot\,;W)$
  the specific volume along the hydrodynamical wave
  curve through $W$, i.\,e., the implicit solution function of
  \eqref{eq:rankine_hugoniot_h} for a shock and of
  \eqref{eq:constant_entropy} for a rarefaction fan, and assume that it is
  non-increasing.
  For the shock-shock or rarefaction-rarefaction
  case we set
  \vspace{-0.5em}
  \begin{align*}
    \begin{cases}
      \begin{aligned}
        \underbar\pi\star\,&:=\,\min\big(\tilde p_h\star,\ph\L,\ph\R\big),
        &\quad
        \underbar P^{(0)}_{\mathrm L}\,&:=\,\underph\starL,
        &\quad
        \underbar P^{(0)}_{\mathrm R}\,&:=\,\underph\starR,
        \\[0.25em]
        \bar\pi\star\,&:=\,\max\big(\tilde p_h\star,\ph\L,\ph\R\big),
        &\quad
        \bar P^{(0)}_{\mathrm L}\,&:=\,\barph\starL,
        &\quad
        \bar P^{(0)}_{\mathrm R}\,&:=\,\barph\starR,
      \end{aligned}
    \end{cases}
  \end{align*}
  and for the shock-rarefaction case we set
  \begin{align*}
    \begin{cases}
      \begin{aligned}
        \underbar\pi\star\,&:=\,\min\big(\ph\L,\ph\R\big),
        &\quad
        \underbar P^{(0)}_{\mathrm L}\,&:=\,\min\big(\ph\L,\barph\star\big),
        &\quad
        \underbar P^{(0)}_{\mathrm R}\,&:=\,\min\big(\ph\R,\underph\star\big),
        \\[0.25em]
        \bar\pi\star\,&:=\,\max\big(\ph\L,\ph\R\big),
        &\quad
        \bar P^{(0)}_{\mathrm L}\,&:=\,\max\big(\ph\L,\barph\star\big),
        &\quad
        \bar P^{(0)}_{\mathrm R}\,&:=\,\max\big(\ph\R,\underph\star\big),
      \end{aligned}
    \end{cases}
  \end{align*}
  with $\barph\star$, $\underph\star$ as in
  \eqref{eq:matching_pressure_bound} and $\barph\starL$, $\barph\starR$,
  $\underph\starL$, $\underph\starR$ as in \eqref{eq:matching_onesided},
  see Lemma~\ref{lem:matching_separation} and
  Propositions~\ref{prop:matching_compression_expansion}
  and~\ref{prop:matching_compression_compression}.
  Let $\underbar\Xi$ and $\overbar\Xi$ be the volume
  bounds of Lemma~\ref{lem:matching_volume_bound}, let $\eta\S$ and
  $\eta\R$ be the deviation bounds \eqref{eq:deviation_eta} of
  Lemma~\ref{lem:deviation_bounds}, and let $\Phi$ be the solution of the
  shifted pressure equation \eqref{eq:pressure_equation_shifted}.

  Assume that the smallness conditions $4vm\le d$ and $\Delta_{\mathrm
  b}^+<d$ of Lemmas~\ref{lem:matching_volume_bound}
  and~\ref{lem:deviation_bounds}, as well as $\underbar\Xi(\bar
  P^{(0)};W)>0$, hold true for $\bar P^{(0)}_{\mathrm L}$ and $W\L$, and
  for $\bar P^{(0)}_{\mathrm R}$ and $W\R$. Define recursively, for
  $n\ge0$,
  \begin{align}
    \label{eq:matching_bootstrap}
    \begin{cases}
      \begin{aligned}
        \underbar V^{(n)}_{\mathrm L}
        &:=
        \underbar\Xi\big(\bar P^{(n)}_{\mathrm L};W\L\big),
        \quad
        \bar V^{(n)}_{\mathrm L}
        :=
        \overbar\Xi\big(\underbar
        P^{(n)}_{\mathrm L};W\L\big),
        \\[0.25em]
        \underbar V^{(n)}_{\mathrm R}
        &:=
        \underbar\Xi\big(\bar P^{(n)}_{\mathrm R};W\R\big),
        \quad
        \bar V^{(n)}_{\mathrm R}
        :=
        \overbar\Xi\big(\underbar
        P^{(n)}_{\mathrm R};W\R\big),
        \\[0.25em]
        \underbar\delta^{(n)}&:=
          \hatpb\big(\underbar V^{(n)}_{\mathrm L}
          \big)-\hatpb\big(\bar V^{(n)}_{\mathrm R}\big),
        \quad
        \bar\delta^{(n)}:=
          \hatpb\big(\underbar V^{(n)}_{\mathrm R}
          \big)-\hatpb\big(\bar V^{(n)}_{\mathrm L}\big),
        \\[0.25em]
        \varepsilon^{(n)}&:=
        \max\big(\bar\delta^{(n)},\,
        \underbar\delta^{(n)}\big),
        \\[0.25em]
        \underbar\eta^{(n)}&:=
        -\eta\S\big(\bar P^{(n)}_{\mathrm L},\underbar V^{(n)}_{\mathrm L};W\L\big)
        -\eta\S\big(\bar P^{(n)}_{\mathrm R},\underbar V^{(n)}_{\mathrm R};W\R\big),
        \\[0.25em]
        \bar\eta^{(n)}&:=
        \eta\R\big(\underbar P^{(n)}_{\mathrm L},\bar V^{(n)}_{\mathrm L};W\L\big)
        +\eta\R\big(\underbar P^{(n)}_{\mathrm R},\bar V^{(n)}_{\mathrm R};W\R\big),
        \\[0.25em]
        \underbar\pi^{\ast(n)}&:=
        \max\big(\underbar\pi\star,\,\Phi(\underbar\eta^{(n)};W\L,W\R)\big),
        \\[0.25em]
        \bar\pi^{\ast(n)}&:=
        \min\big(\bar\pi\star,\,\Phi(\bar\eta^{(n)};W\L,W\R)\big),
        \\[0.25em]
        \bar P^{(n+1)}_{\mathrm L}&:=
        \min\big(\bar P^{(n)}_{\mathrm L},\,
        \bar\pi^{\ast(n)}+\max(\bar\delta^{(n)},0)\big),
        \\[0.25em]
        \bar P^{(n+1)}_{\mathrm R}&:=
        \min\big(\bar P^{(n)}_{\mathrm R},\,
        \bar\pi^{\ast(n)}+\max(\underbar\delta^{(n)},0)\big),
        \\[0.25em]
        \underbar P^{(n+1)}_{\mathrm L}&:=
        \max\big(\underbar P^{(n)}_{\mathrm L},\,
        \underbar\pi^{\ast(n)}-\max(\underbar\delta^{(n)},0)\big),
        \\[0.25em]
        \underbar P^{(n+1)}_{\mathrm R}&:=
        \max\big(\underbar P^{(n)}_{\mathrm R},\,
        \underbar\pi^{\ast(n)}-\max(\bar\delta^{(n)},0)\big).
      \end{aligned}
    \end{cases}
  \end{align}
  Then, for every $n\ge0$,
  \begin{align}
    \label{eq:matching_bootstrap_bound}
    \begin{cases}
      \begin{aligned}
        -\bar\delta^{(n)}
        \;&\le\;
        \ph\starR-\ph\starL\;\le\;\underbar\delta^{(n)},
        &\quad
        \sigma\;&\le\;\varepsilon^{(n)},
        \\[0.25em]
        \underbar P^{(n)}_{\mathrm L}\;&\le\;
        \ph\starL\;\le\;\bar P^{(n)}_{\mathrm L},
        &\quad
        \underbar P^{(n)}_{\mathrm R}\;&\le\;
        \ph\starR\;\le\;\bar P^{(n)}_{\mathrm R},
        \\[0.25em]
        \underbar V^{(n)}_{\mathrm L}\;&
        \le\;v\starL\;\le\;
        \bar V^{(n)}_{\mathrm L},
        &\quad
        \underbar V^{(n)}_{\mathrm R}\;&
        \le\;v\starR\;\le\;
        \bar V^{(n)}_{\mathrm R},
      \end{aligned}
    \end{cases}
  \end{align}
  and $\varepsilon^{(n)}$, $\underbar\delta^{(n)}$,
  $\bar\delta^{(n)}$,
  $\bar\eta^{(n)}$, $\bar\pi^{\ast(n)}$, $\bar
  P^{(n)}_{\mathrm L/\mathrm R}$ and $\bar V^{(n)}_{\mathrm L/\mathrm R}$
  are non-increasing, while
  $\underbar\eta^{(n)}$, $\underbar\pi^{\ast(n)}$,
  $\underbar P^{(n)}_{\mathrm L/\mathrm R}$ and
  $\underbar V^{(n)}_{\mathrm L/\mathrm R}$ are non-decreasing. In
  particular
  $\varepsilon^{(n)}$ converges to a limit
  $\varepsilon^{(\infty)}\ge\sigma$, and the iteration may be stopped after
  any number of steps.
\end{proposition}
\begin{proof}
  Both waves being regular,
  Propositions~\ref{prop:matching_compression_expansion}
  and~\ref{prop:matching_compression_compression} are the case $n=0$ of
  the pressure bounds in \eqref{eq:matching_bootstrap_bound}. Assume they
  hold for some $n\ge0$. Since $\bar P^{(n)}_{\mathrm L/\mathrm R}\le\bar
  P^{(0)}_{\mathrm L/\mathrm R}$ and $\underbar\Xi$ is non-increasing, the
  smallness conditions of Lemmas~\ref{lem:matching_volume_bound}
  and~\ref{lem:deviation_bounds} and the positivity of the lower volume
  bound remain valid, and Lemma~\ref{lem:matching_volume_bound} yields the
  volume bounds in \eqref{eq:matching_bootstrap_bound}. As $\hatpb$ is
  non-increasing by \eqref{A2}, they imply
  $\hatpb(v\starR)-\hatpb(v\starL)\le\bar\delta^{(n)}$
  and $\hatpb(v\starL)-\hatpb(v\starR)\le\underbar\delta^{(n)}$, where the
  first difference equals $\ph\starL-\ph\starR$ by the second equation
  of \eqref{eq:matching}; this is the first line of
  \eqref{eq:matching_bootstrap_bound} and gives
  $\sigma\le\varepsilon^{(n)}$. Next,
  $\min(\ph\starL,\ph\starR)\le\bar\pi^{\ast(n)}$ and
  $\underbar\pi^{\ast(n)}\le\max(\ph\starL,\ph\starR)$: for
  $\underbar\pi\star$ and $\bar\pi\star$ this is
  Lemma~\ref{lem:matching_separation} together with $\ph\star\ge\ph$ across
  a shock and $\ph\star\le\ph$ across a rarefaction fan, for
  $\Phi(\underbar\eta^{(n)};W\L,W\R)$ and $\Phi(\bar\eta^{(n)};W\L,W\R)$ it
  is Corollary~\ref{cor:deviation_star_pressures}. Distinguishing which of
  the two star pressures is the larger and inserting the bounds on
  $\ph\starL-\ph\starR$ gives
  \begin{align*}
    \underbar\pi^{\ast(n)}-\max\big(\underbar\delta^{(n)},0\big)
    \;&\le\;\ph\starL\;\le\;
    \bar\pi^{\ast(n)}+\max\big(\bar\delta^{(n)},0\big),
    \\[0.25em]
    \underbar\pi^{\ast(n)}-\max\big(\bar\delta^{(n)},0\big)
    \;&\le\;\ph\starR\;\le\;
    \bar\pi^{\ast(n)}+\max\big(\underbar\delta^{(n)},0\big),
  \end{align*}
  hence the pressure bounds for $n+1$. The
  monotonicity of $\bar P^{(n)}_{\mathrm L/\mathrm R}$ and $\underbar
  P^{(n)}_{\mathrm L/\mathrm R}$ holds by definition. It carries over to
  the volume bounds, $\underbar\delta^{(n)}$, $\bar\delta^{(n)}$ and
  $\varepsilon^{(n)}$ because $\underbar\Xi$, $\overbar\Xi$ and $\hatpb$
  are non-increasing; to $\underbar\eta^{(n)}$ and $\bar\eta^{(n)}$ because
  none of $\overbar P$, $\Delta_{\mathrm b}^+$ and $\Delta_{\mathrm b}^-$
  increases, and neither does $\overbar V/\ah(\overbar V,\underbar P)$, as
  $\hatah(\tau)/\tau$ is non-increasing along the isentrope on which
  $\overbar V$ lies whenever $\eta\R\neq0$; and to
  $\underbar\pi^{\ast(n)}$ and $\bar\pi^{\ast(n)}$ because $\Phi$ is
  non-decreasing in its first argument.
\end{proof}
\begin{proof}[Proof of Theorem~\ref{thm:main_iterative}]
  Theorem~\ref{thm:main} provides the initial bracket $\underbar
  P^{(0)}_{\mathrm L}\le\ph\starL\le\bar P^{(0)}_{\mathrm L}$, and
  correspondingly on the right. As the convexity criterion
  \eqref{eq:kappa_criterion} is assumed for these brackets,
  Lemma~\ref{lem:elementary_criterion} guarantees that both outer waves are
  regular. Together with the smallness conditions
  $4vm\le d$ and $\Delta_{\mathrm b}^+<d$, this puts
  Lemmas~\ref{lem:matching_volume_bound} and~\ref{lem:deviation_bounds}
  and Corollary~\ref{cor:deviation_star_pressures} at our disposal for
  $W\L$ and $W\R$, and with them the quantities $\underbar\Xi$,
  $\overbar\Xi$, $\eta\S$, $\eta\R$ and $\Phi$ entering
  \eqref{eq:main_bootstrap}. The choice \eqref{eq:main_pistar} satisfies
  $\underbar\pi\star\le\max(\ph\starL,\ph\starR)$ and
  $\min(\ph\starL,\ph\starR)\le\bar\pi\star$ in every configuration, which
  is all the proof of Proposition~\ref{prop:matching_bootstrap} uses.
  Consequently, Proposition~\ref{prop:matching_bootstrap} applies with
  \eqref{eq:main_bootstrap} in place of \eqref{eq:matching_bootstrap},
  which establishes the monotonicity of the two sequences. The wavespeed
  estimates of Propositions~\ref{prop:matching_compression_expansion}
  and~\ref{prop:matching_compression_compression} rest on the pressure
  bracket alone and are monotone in it by \eqref{eq:Q_monotone}
  (see Lemma~\ref{lem:Q_monotone}), as long as
  the bracket contains $\ph$: the maximum in \eqref{eq:main_wavespeed} has
  to cover the shock, which runs from $v$ to $v\star$, and the state $v$
  itself, at which the head of a fan travels. The iterated bracket
  \eqref{eq:matching_bootstrap_bound} may leave $\ph$, while the clipped
  one, $[\min(\underbar P^{(n)},\ph),\max(\bar P^{(n)},\ph)]$, is still
  tighter than the initial one and contains $\ph$. Replacing the initial
  bracket by the clipped one therefore preserves \eqref{eq:main_wavespeed}.
\end{proof}
\begin{proof}[Proof of Corollary~\ref{cor:main_closing}]
  By Theorem~\ref{thm:main_iterative} both outer waves are
  regular, so that a compressive wave is a shock and an
  expansive wave is a rarefaction. The bracket
  \eqref{eq:matching_bootstrap_bound} thus decides the type in each of
  the four cases. In cases (b) and (d) the head of the fan travels at
  the characteristic speed \eqref{eq:lambda_rarefaction} of
  Proposition~\ref{prop:rarefaction}. In case (a),
  Proposition~\ref{prop:shock_wavespeed_general} holds for
  $\ph\star=\ph\starL$, and each term in
  \eqref{eq:wavespeed_bound_general} is monotone in the bracket: $\tilde
  v\star=\xi_h(\ph\starL;W\L)\ge\bar v\starL$ because $\xi_h$ is
  decreasing, hence $\hatpb(\tilde v\star)\le\hatpb(\bar v\starL)$ by
  \eqref{A2} and the maximum over $[\tilde v\star,v\L]$ is bounded by
  the one over $[\bar v\starL,v\L]$; $\ph\starL-\ph\L\ge\underbar
  P^{(n)}_{\mathrm L}-\ph\L>0$; and $u\L-\lambda^-_{1,\mathrm h}\le
  u\L-\bar\lambda^-_{1,\mathrm h}$ by the monotonicity of the mass flux
  (Lemma~\ref{lem:Q_monotone}). Case (c) is the
  mirrored case.
\end{proof}


\section{Numerical illustrations}
\label{sec:numerics}
We close by illustrating the performance of the algorithm of
Theorems~\ref{thm:main} and \ref{thm:main_iterative} by comparing the wavespeed bounds it
produces with the wavespeeds of the exact Riemann solution, on a selected
synthetic test. We choose the Noble-Abel stiffened gas equation of state
(see Section~\ref{sec:nasg}) as hydrodynamical constituent and complete it
with a synthetic barotropic constituent $\pb$ (see
Section~\ref{sec:barotropic}) that concentrates the barotropic speed of
sound in a narrow bump, so that the total pressure $p=\ph+\pb$ is
potentially nonconvex and the outer waves of \eqref{eq:riemann_problem} may
no longer be elementary.

\subsection{The Noble-Abel stiffened gas equation of state}
\label{sec:nasg}
We collect closed-form expressions for the case that the hydrodynamical
constituent is given by the Noble-Abel stiffened gas equation of state
\cite{LeMetayerSaurel2016,LeMetayerMassoniSaurel2004}. For $p_\infty=0$ all
formulas below reduce to the corresponding formulas for the covolume gas
stated in \cite[Sect.~4.7]{Toro2009} and
\cite[Sect.~4.1]{ClaytonGuermondPopov2022}, see also
\cite{GuermondPopov2016WaveSpeed}. The equation of state and the specific
internal energy read:
\begin{align}
  \label{eq:nasg_eos}
  \ph(\rho,\eh)\,=\,(\gamma-1)\frac{\rho(\eh-q)}{1-b\rho}
  -\gamma\,p_\infty,
  \quad
  \hateh(v,\ph)\,=\,q+(v-b)
  \frac{\ph+\gamma\,p_\infty}{\gamma-1},
\end{align}
where $\gamma>1$, $b\ge0$, $p_\infty\ge0$ and $q$ are constants. The
isentropes of \eqref{eq:nasg_eos} are the curves
$(\ph+p_\infty)\,(v-b)^{\gamma}=\text{const}$, and consequently
\begin{align*}
  \ah^2(\rho,\eh)\;=\;
  \gamma\,\frac{\ph+p_\infty}{\rho\,(1-b\rho)}
  \;=\;\gamma\,\frac{\big(\ph+p_\infty\big)\,v^2}{v-b}.
\end{align*}
For a state $W=[\rho,u,\ph]$ we abbreviate
\begin{align*}
  A(W):=\frac{2(v-b)}{\gamma+1},
  \;\;
  B(W):=p_\infty+\frac{\gamma-1}{\gamma+1}\big(\ph+p_\infty\big),
  \;\;
  \bar a(W):=\ah(W)\frac{v-b}{v}.
\end{align*}
The pressure function \eqref{eq:pressure_function_h} then takes the closed
form
\begin{align*}
  f_{\mathrm h}(\tilde p_h\star;W)\;=\;
  \begin{cases}
    \begin{aligned}
      &f\S_{\mathrm h}(\tilde p_h\star;W)\;=\;
      \big(\tilde p_h\star-\ph\big)
      \bigg(\frac{A(W)}{\tilde p_h\star+B(W)}\bigg)^{1/2}
      &\;&\text{if }\tilde p_h\star>\ph,
      \\[0.5em]
      &f\R_{\mathrm h}(\tilde p_h\star;W)\;=\;
      \frac{2\,\bar a(W)}{\gamma-1}
      \Bigg[\bigg(\frac{\tilde p_h\star+p_\infty}{\ph+p_\infty}
      \bigg)^{\frac{\gamma-1}{2\gamma}}-1\Bigg]
      &\;&\text{if }\tilde p_h\star\le\ph,
    \end{aligned}
  \end{cases}
\end{align*}
so that the matching condition \eqref{eq:pressure_equation_h}, viz.
$0=f_{\mathrm h}\big(\tilde p_h\star;W\L\big)
+f_{\mathrm h}\big(\tilde p_h\star;W\R\big)+u\R-u\L$,
has a unique solution $\tilde p_h\star>-p_\infty$ provided the nonvacuum
condition
\begin{align*}
  u\R-u\L\;<\;\frac{2}{\gamma-1}\,
  \big(\bar a(W\L)\,+\,\bar a(W\R)\big)
\end{align*}
holds \cite[Eq.~(4.40)]{Toro2009},
\cite[Eq.~(4.3) and Lem.~4.1]{ClaytonGuermondPopov2022}.
The wavespeeds \eqref{eq:lambda_one_h} are:
\begin{align*}
  \begin{cases}
    \begin{aligned}
      \lambda^-_{1,\mathrm h}\;&=\;u\L-\ah(W\L)
      \bigg(1+\frac{\gamma+1}{2\gamma}\,
      \frac{\big(\tilde p_h\star-\ph\L\big)_+}{\ph\L+p_\infty}
      \bigg)^{1/2},
      \\[0.5em]
      \lambda^+_{3,\mathrm h}\;&=\;u\R+\ah(W\R)
      \bigg(1+\frac{\gamma+1}{2\gamma}\,
      \frac{\big(\tilde p_h\star-\ph\R\big)_+}{\ph\R+p_\infty}
      \bigg)^{1/2},
    \end{aligned}
  \end{cases}
\end{align*}
where $(\,\cdot\,)_+:=\max(\,\cdot\,,0)$, which covers both cases of
\eqref{eq:lambda_one_h} at once. Finally, the specific volume $\tilde v\star$
behind a hydrodynamical shock, i.\,e., the solution of the Rankine-Hugoniot
condition \eqref{eq:rankine_hugoniot_h}, is available in closed form,
\begin{align*}
  \xi_h(\tilde p_h\star;W)\;=\;
  b+(v-b)\,
  \frac{(\gamma+1)\big(\ph+p_\infty\big)
  +(\gamma-1)\big(\tilde p_h\star+p_\infty\big)}
  {(\gamma-1)\big(\ph+p_\infty\big)
  +(\gamma+1)\big(\tilde p_h\star+p_\infty\big)}.
\end{align*}
It equals $v$ at $\tilde p_h\star=\ph$, is strictly decreasing in $\tilde
p_h\star$, and is read through its continuation past $\ph$, see
Definition~\ref{def:expansive_continuation}.

\begin{definition}[Fast wavespeed estimate for NASG]
  The \emph{two-rarefaction approximation}
  \cite[Sect.~5.1]{ClaytonGuermondPopov2022},
  \cite[Lem.~4.2]{GuermondPopov2016WaveSpeed} furnishes a computable upper
  bound on the pressure $\tilde\ph\star$. Setting
  $z:=\frac{\gamma-1}{2\gamma}$,
  we define
  \begin{align}
    \label{eq:nasg_two_rarefaction}
    \tilde p_h\starRR\;:=\; \Bigg(\frac{\bar a(W\L)+\bar a(W\R)
    -\tfrac{\gamma-1}{2}\,\big(u\R-u\L\big)}
    {\bar a(W\L)\,\big(\ph\L+p_\infty\big)^{-z}
    +\bar a(W\R)\,\big(\ph\R+p_\infty\big)^{-z}}\Bigg)^{1/z}
    \,-\,p_\infty.
  \end{align}
  Assuming that $\gamma\in(1,\tfrac53]$ and under the nonvacuum condition
  stated above we have $\tilde p_h\star\le\tilde p_h\starRR$. Equation
  \eqref{eq:nasg_two_rarefaction} is thus a computable upper bound on the
  star pressure. For $\gamma>\tfrac53$ the two-rarefaction formulation has
  to be modified slightly, see \cite[Sect.~5,
  App.~A]{ClaytonGuermondPopov2022}.
\end{definition}

\subsection{The barotropic pressure component}
\label{sec:barotropic}
We continue by introducing a synthetic pressure component $\pb$ that we use
in our numerical tests. Rather than modeling a particular material, it is
constructed to exhibit all the complications the theory has to cope with:
the total pressure $p=\ph+\pb$ loses convexity so that composite waves
appear, while \eqref{A2} is maintained and the Riemann problem
\eqref{eq:riemann_problem} remains hyperbolic. Both requirements are met by
concentrating the barotropic speed of sound in a narrow bump of width
$\varepsilon$ centered at a reference density $\rho_0$.
\begin{definition}[Bumpy barotropic pressure]
  \label{def:bump}
  Let $\rho_0>0$, $c_0>0$ and $\varepsilon>0$ and define the barotropic
  constituent $\hatpb$ by fixing the barotropic speed of sound to
  \begin{align*}
    \hatab^2(v)\;=\;
    \frac{c_0^2}{\big((v^{-1}-\rho_0)^2+\varepsilon\big)^{3/2}}\,.
  \end{align*}
  Integrating $\ab^2=\partial_\rho\pb$ subject to $\pb(0)=0$, and then
  $\pb=\rho^2\,\partial_\rho\eb$, yields
  \begin{align*}
    \hatpb(v)\;&=\;\frac{c_0^2}{\varepsilon}\,
    \Bigg(
      \frac{v^{-1}-\rho_0}{\sqrt{(v^{-1}-\rho_0)^2+\varepsilon}}
      \,+\,\frac{\rho_0}{\sqrt{\rho_0^2+\varepsilon}}
    \Bigg),
    \\[0.75em]
    \hateb(v)\;&=\;
    \frac{c_0^2\,\rho_0\,v}{\varepsilon\,\big(\rho_0^2+\varepsilon\big)}\,
    \Big(\sqrt{(v^{-1}-\rho_0)^2+\varepsilon}-\sqrt{\rho_0^2+\varepsilon}\Big)
    \\[0.25em]\notag
    &\qquad\;+\;
    \frac{2\,c_0^2}{\big(\rho_0^2+\varepsilon\big)^{3/2}}\,
    \operatorname{artanh}\!
    \Bigg(
      \frac{v^{-1}-\sqrt{(v^{-1}-\rho_0)^2+\varepsilon}}
           {\sqrt{\rho_0^2+\varepsilon}}
    \Bigg)
    \;+\;e_0,
  \end{align*}
  with an arbitrary additive constant $e_0$.
\end{definition}
Taking $c_0$ large enough forces a composite wave whenever the two initial
densities are placed symmetrically around the bump, $\rho_0-\rho\L =
\rho\R-\rho_0$, at constant entropy.

\subsection{Algorithmic realization}
\label{sec:algorithm}
Theorems~\ref{thm:main} and~\ref{thm:main_iterative} are spelled out as
Algorithms~\ref{alg:general} to~\ref{alg:elementary} on
pages~\pageref{alg:general} to~\pageref{alg:elementary}, in the form in
which they are used for the tests below. Here,
$\xi_h\S(\,\cdot\,;W)$ denotes the solution of the reduced
Rankine--Hugoniot condition \eqref{eq:rankine_hugoniot_h} and
$\xi_h\R(\,\cdot\,;W)$ the solution of the isentrope
\eqref{eq:constant_entropy}; $f\S_{\mathrm h}$, $Q_{\mathrm h}$ and
$\xi_h\S$ are read through their continuation past $\ph$, see
Definition~\ref{def:expansive_continuation}. Should either admissibility check
of Algorithm~\ref{alg:elementary} fail, the pressure and wavespeed bounds
of Algorithm~\ref{alg:general} are returned unchanged.
For the Noble--Abel stiffened gas all hydrodynamical
ingredients are available in closed form, see Section~\ref{sec:nasg}; we
solve for the root of \eqref{eq:pressure_equation_h} with a simple
bisection. However, better techniques exist: Guermond and
Popov~\cite{GuermondPopov2016WaveSpeed} describe a third-order convergent
quadratic Newton iteration.

\subsection{Elementary waves test cases}
\label{sec:elementary_tests}
Throughout this section we take the hydrodynamical constituent
\eqref{eq:nasg_eos} to be an ideal gas, $\gamma=1.4$ and $b=q=p_\infty=0$,
and we place the barotropic bump of Definition~\ref{def:bump} at $\rho_0=1$
with width $\varepsilon=10^{-2}$, which leaves its strength $c_0$ as a
free parameter that we will be varying. Our first set of tests concerns a family of
configurations for which both outer waves of \eqref{eq:riemann_problem}
remain elementary. The ten initial states E1--E10 are collected in
Table~\ref{tab:elementary_configurations}. The third component of a state
is the hydrodynamical pressure $\ph$; the barotropic constituent is
superimposed on it and its strength $c_0$ is varied separately. The
hydrodynamic type records the wave pattern of the reduced hydrodynamical
Riemann problem \eqref{eq:pressure_equation_h}, the first letter naming the
left and the second the right wave, with \textup{R} for a rarefaction and
\textup{S} for a shock. The last column of
Table~\ref{tab:elementary_configurations} lists the cut-off $c_0^{\max}$,
the largest strength at which the prerequisites of
Theorem~\ref{thm:main_iterative}, namely the smallness conditions $4vm\le
d$ and $\Delta_{\mathrm b}^+<d$, the positivity of the lower volume bound
and the convexity criterion \eqref{eq:kappa_criterion}, hold true for the
initial bounds of Theorem~\ref{thm:main}. The value is obtained by
bisection and reported to three significant digits. All bump strengths used
below are stated as fractions of this configuration-dependent cut-off.
\begin{table}[t]
  \centering
  \setlength{\tabcolsep}{5.5pt}
  \begin{tabular}{clllr}
    \toprule
    & $W\L=[\rho\L,u\L,\ph\L]$ & $W\R=[\rho\R,u\R,\ph\R]$ & type & $c_0^{\max}$ \\
    \midrule
    E1   & $[1,\,0,\,0.0666667]$ & $[0.001,\,0,\,6.66667\cdot10^{-11}]$ & RS & $8.19\cdot10^{-7}$ \\
    E2   & $[1,\,0,\,1]$ & $[0.125,\,0,\,0.1]$ & RS & $2.24\cdot10^{-2}$ \\
    E3   & $[0.445,\,0.698,\,3.528]$ & $[0.5,\,0,\,0.571]$ & RS & $7.29\cdot10^{-3}$ \\
    E4   & $[1,\,10,\,1000]$ & $[1,\,10,\,0.01]$ & RS & $1.82\cdot10^{-3}$ \\
    E5   & $[1,\,0.02,\,1]$ & $[1,\,-0.02,\,1]$ & SS & $2.18\cdot10^{-2}$ \\
    E6   & $[2,\,1,\,1]$ & $[1,\,-1,\,10]$ & SS & $3.73\cdot10^{-3}$ \\
    E7   & $[10,\,0.5,\,0.1]$ & $[2,\,-0.5,\,10]$ & SS & $4.55\cdot10^{-4}$ \\
    E8   & $[5,\,0.25,\,0.1]$ & $[2,\,-0.25,\,1]$ & SS & $2.36\cdot10^{-3}$ \\
    E9   & $[0.02,\,-0.25,\,10]$ & $[0.02,\,0.25,\,10]$ & RR & $15.5$ \\
    E10  & $[1,\,-0.25,\,10]$ & $[10,\,0.25,\,10]$ & RR & $6.15\cdot10^{-2}$ \\
    \bottomrule
  \end{tabular}
  \caption{The ten elementary-wave configurations E1--E10, with the wave
  pattern of the underlying hydrodynamical Riemann problem and the largest
  bump strength $c_0^{\max}$ admitted by the
  prerequisites of Theorem~\ref{thm:main_iterative}.}
  \label{tab:elementary_configurations}
\end{table}
\paragraph{Hydrodynamical approximation}
We vary the bump strength $c_0$ over $0\%$, $10\%$, $50\%$ and $99\%$
of the cut-off $c_0^{\max}$ and report in
Table~\ref{tab:elementary_wavespeeds}
(page~\pageref{tab:elementary_wavespeeds}) the two star pressures
$\ph\starL$ and $\ph\starR$ and the extremal wavespeeds $\lambda^-_1$,
$\lambda^+_3$ of the exact Riemann problem \eqref{eq:riemann_problem},
together with the star pressure $\tilde p_h\star$ and the wavespeeds
$\lambda^-_{1,\mathrm h}$, $\lambda^+_{3,\mathrm h}$ of the reduced
hydrodynamical problem \eqref{eq:riemann_problem_h}. The last column of
each subtable records the relative deviation $\Delta\lambda_{\max}$ of
$\max\{|\lambda^-_{1,\mathrm h}|,|\lambda^+_{3,\mathrm h}|\}$ from its
exact counterpart $\max\{|\lambda^-_1|,|\lambda^+_3|\}$, in percent. At
$c_0=0$ the two problems coincide. As the bump is turned up, the two exact
star pressures separate, and $\Delta\lambda_{\max}$ turns negative
throughout: the hydrodynamical wavespeeds fall short of the exact ones and
are thus no upper bound for them. The deficit grows with $c_0$; at
$c_0=0.99\,c_0^{\max}$ it reaches up to $13.8\%$.

\paragraph{Full wavespeed bounds}
Tables~\ref{tab:initial_bounds} (page~\pageref{tab:initial_bounds}) and
\ref{tab:bootstrapped_bounds} (page~\pageref{tab:bootstrapped_bounds})
report the guaranteed estimates themselves, for the same configurations and
bump strengths. Table~\ref{tab:initial_bounds}
(page~\pageref{tab:initial_bounds}) lists the one-sided pressure bounds
\eqref{eq:main_pressure_bound} of Theorem~\ref{thm:main} and the wavespeed
bounds \eqref{eq:main_wavespeed} they imply;
Table~\ref{tab:bootstrapped_bounds}
(page~\pageref{tab:bootstrapped_bounds}) lists the same quantities after
$N$ steps of the iteration \eqref{eq:main_bootstrap} of
Theorem~\ref{thm:main_iterative}, run to a tolerance of $10^{-6}$. In both,
$\Delta\lambda_{\max}$ is positive throughout: unlike the hydrodynamical
wavespeeds of Table~\ref{tab:elementary_wavespeeds}
(page~\pageref{tab:elementary_wavespeeds}), the estimates never fall below
the exact maximal wavespeed.

The two estimates differ markedly in how much room they leave. Without the
iteration, the pressure bounds \eqref{eq:main_pressure_bound} approximate
the star pressure by the larger of the two initial pressures, which for the
rarefaction-shock configurations E1--E4 is off by up to two orders of
magnitude, and the resulting overshoot of the
wavespeed bound ranges from $16.9\%$ to $909\%$ for E1--E4 and E6--E8,
almost independently of the bump strength. The iteration removes this
deficiency almost entirely. At $c_0=0$
the bracket collapses onto $\tilde p_h\star$ within two steps and the
bounds coincide with the exact wavespeeds in all ten configurations. At a
tenth of the admissible bump strength the overshoot stays below $0.2\%$
throughout, and in eight of the ten configurations
below $0.01\%$. Even at $c_0=0.99\,c_0^{\max}$, where the barotropic
contribution is as strong as the theory permits, the
overshoot remains below $0.1\%$ for eight of the ten configurations and
reaches $12.1\%$ only for the symmetric shock-shock configuration E5, which
is also one of the two cases that require the most iterations. For the
rarefaction-rarefaction configurations E9 and E10 the iterated pressure bracket
falls below $\ph$ on both sides, so that Corollary~\ref{cor:main_closing}
returns the exact head speed of the fan and the overshoot vanishes.

\begin{remark}[Comparison with $\tilde p_{\mathrm h}\starRR$]
  The two-rarefaction approximation
  \eqref{eq:nasg_two_rarefaction} is the customary estimate of the star
  pressure in practical wavespeed estimators
  \cite{GuermondPopov2016WaveSpeed}. It is instructive to measure the
  uniterated initial bounds of Table~\ref{tab:initial_bounds} against it.
  For the rarefaction-shock configurations E1--E4 of
  Table~\ref{tab:elementary_configurations} at $c_0=0$, the two yield the
  following star pressures and maximal wavespeeds
  $\lambda_{\max}:=\max\{|\lambda^-_1|,|\lambda^+_3|\}$:
  \\[-0.5em]
  \begin{center}
    \small
    \begin{tabular}{lrrrrrr}
      \toprule
      & exact & \multicolumn{3}{c}{two-rarefaction \eqref{eq:nasg_two_rarefaction}}
      & \multicolumn{2}{c}{bounds, Table~\ref{tab:initial_bounds}} \\
      \cmidrule(lr){2-2} \cmidrule(lr){3-5} \cmidrule(lr){6-7}
      & $\lambda_{\max}$ & $\tilde p_{\mathrm h}\starRR$ & $\lambda_{\max}$
      & $\Delta\lambda_{\max}$
      & $\lambda_{\max}$ & $\Delta\lambda_{\max}$ \\[0.25em]
      E1 & 0.886139 & 0.058724 & 8.39454 & $+847\%$   & 8.94427 & $+909\%$ \\
      E2 & 1.75216  & 0.306767 & 1.76209 & $+0.567\%$ & 3.1241  & $+78.3\%$ \\
      E3 & 2.63357  & 2.50966  & 2.63357 & $+0\%$     & 3.07858 & $+16.9\%$ \\
      E4 & 33.5175  & 912.449  & 43.0899 & $+28.6\%$  & 44.641  & $+33.2\%$ \\
      \bottomrule
    \end{tabular}
  \end{center}
  \vspace{0.75em}
  Where the two-rarefaction approximation overshoots,
  as for E1 and E4, the initial bounds overshoot by a comparable margin;
  where it is nearly exact, as for E2 and E3, they are less tight. The
  initial bounds thus perform on the level of an estimate that is used
  routinely in practice, at the cost of two root solves of the continued
  shock branch, and with the guarantee that they never fall short, also
  for the composite equation of state. They would serve well in real
  computations on their own, and the iteration of
  Theorem~\ref{thm:main_iterative} is available wherever tighter bounds pay
  off.
\end{remark}
\subsection{Composite and exotic wave test cases}
\label{sec:composite_tests}
For the composite tests we fix the bump strength at $c_0=0.1$, keeping $\rho_0=1$ and
$\varepsilon=10^{-2}$, so that the total pressure $p=\ph+\pb$ is nonconvex
over the whole band $\rho\approx1.008\ldots1.32$ at $\ph\approx1$. This
value exceeds the cut-off $c_0^{\max}$ of every configuration considered
below, so the iteration of Theorem~\ref{thm:main_iterative} does not start
and only the bounds of Theorem~\ref{thm:main}, which make no use of the
wave structure, remain available.

Riemann data whose outer waves are composite are hard to guess, so we
construct the waves backwards: both outer waves are glued together from
elementary pieces, and only afterwards do we read off the initial data
they connect. Which pieces are available at a given state is decided by
the fundamental derivative \eqref{A3}. Where $\mathcal G>0$ a compression
is a shock and an expansion a fan; where $\mathcal G<0$ the speed of sound
drops as the density grows and the two roles are exchanged, see
Remark~\ref{rem:composite_wave}. We mark the latter, exotic waves with a
dagger, $\mathrm S^\dagger$ and $\mathrm R^\dagger$. The wave curves
depend on $\rho$ and $\ph$ alone, so a branch can be shifted in velocity
after the fact; we use this to place the contact at rest, $u\star=0$, on
both sides, which leaves the star pressure as the only remaining
condition. It is met by bisecting a free parameter of one branch, or,
where that branch is a single elementary wave, by solving for its star
volume directly.

Table~\ref{tab:composite_configurations}
(page~\pageref{tab:composite_configurations}) lists the twelve
configurations C1--C12 obtained in this way, wave by wave. In C1--C6 both
outer waves are single elementary waves, but at least one of them runs
inside the nonconvex band and is exotic; the four types are paired in the
six combinations $\mathrm S\,\mathrm S^\dagger$, $\mathrm S\,\mathrm
R^\dagger$, $\mathrm R\,\mathrm S^\dagger$, $\mathrm R\,\mathrm R^\dagger$,
$\mathrm S^\dagger\mathrm S^\dagger$ and $\mathrm R^\dagger\mathrm
R^\dagger$. In C7--C12 at least one of the two outer waves is composite, of
two or of three pieces. Tables~\ref{tab:composite_wavespeeds}
(page~\pageref{tab:composite_wavespeeds}) and \ref{tab:composite_bounds}
(page~\pageref{tab:composite_bounds}) report the exact solution, the
hydrodynamical approximation and the guaranteed bounds for these
configurations.

The hydrodynamical wavespeeds now fall short by $25\%$ to $68\%$, several
times the deficit of Table~\ref{tab:elementary_wavespeeds}
(page~\pageref{tab:elementary_wavespeeds}), where the smallness condition
kept the bump weak. The bounds of Theorem~\ref{thm:main} remain upper
bounds on all twelve configurations and overestimate by $44\%$ to $94\%$,
apart from C1 and C11, where the overshoot reaches $197\%$ and $178\%$.


\subsection{1D numerical tests for composite wave configurations}
\label{sec:composite_numerics}
The configurations of Section~\ref{sec:composite_tests} were constructed
backwards from their wave structure satisfying \eqref{A3} to \eqref{A5}. It
is thus worth checking that they are indeed the solutions of the Riemann
problems whose data they define. We have implemented the composite equation
of state $p=\ph+\pb$ of Sections~\ref{sec:nasg} and \ref{sec:barotropic}
and our wavespeed estimator (Theorem~\ref{thm:main}) in the finite element
code \texttt{ryujin} \cite{ryujin-2021-1} and solved three of the
configurations as one-dimensional Riemann problems on $[-1,1]$ discretized
with $2^{15}$ cells. The final time was chosen such that the outer waves
did not reach the boundaries. We select C10, whose left wave is the
three-piece composite $\mathrm S{+}\mathrm R^\dagger{+}\mathrm S$; C11,
whose right wave is led by an expansive shock; and C12, a configuration in
which both outer waves are composite and of different type.

Figures~\ref{fig:composite_C10}--\ref{fig:composite_C12} show the density
at the final time against the exact solution, together with the split
$a^2=\ah^2+\ab^2$ of the speed of sound along the isentropes that carry the
waves. The computed profiles reproduce the constructed wave structures:
each fan appears between the two wave speeds quoted in
Table~\ref{tab:composite_configurations}, and the shock that bounds it sits
at the speed of the adjacent fan edge, which is the defining feature of a
composite wave and is visible as the coincidence of a fan endpoint with a
shock in every one of the three figures.

The lower panels show why the hydrodynamical wavespeeds of
Table~\ref{tab:composite_wavespeeds} fall so far short. The barotropic
contribution $\ab$ dominates $\ah$ in a narrow band around $\rho_0$ and is
negligible outside it, so a wave is distorted precisely when the states it
connects lie on the bump: the left waves of C10 and C12 cross it and run
several times faster than the hydrodynamical estimate, whereas the right wave
of C12 sits far out on the tail and is matched to within a percent. The bound
of Theorem~\ref{thm:main} is by contrast an upper bound in all three cases;
for C11 it falls outside the computational domain at the final time and is
only indicated in Figure~\ref{fig:composite_C11}.

Maintaining a strict upper bound $\lambda_{\max}$ on the maximal wavespeed
matters: Figure~\ref{fig:composite_C11_bounds} recomputes C11 three times,
changing only the estimate that enters the graph viscosity. Taking the
bound of Theorem~\ref{thm:main} for the hydrodynamical constituent alone,
which is what a closure that ignores the barotropic contribution delivers,
underestimates $\lambda^+_3$ by $37\%$ and fails outright: the tail of the
fan collapses onto a plateau and the shock closing the composite wave lags
visibly behind. The surrogate construction of
\cite{ClaytonGuermondPopov2022} comes much closer, and it does keep the
update in a (surrogate) admissible set of the original equation of state;
but it is also not free of defects as it underestimates the true wavespeed:
Its fan flattens over the last stretch before $\mathrm S^\dagger$ and
carries oscillations an order of magnitude larger than those of the
guaranteed bound.


\section{Conclusion}
\label{sec:conclusion}
We have derived guaranteed upper bounds on the extremal wave speeds of the
Riemann problem for the compressible Euler equations with a composite
equation of state $p=\ph+\pb$, in which a barotropic contribution
$\pb(\rho)$ is added to a hydrodynamical closure $\ph$ with the classical
convexity properties. The approach rests on parametrizing the wave curves
of the complete system by the hydrodynamical star pressure and comparing
them with the wave curves of the auxiliary problem for $\ph$ alone. The
comparison yields efficiently computable wave-speed bounds that require
little more than the solution of the hydrodynamical Riemann problem
(Theorems~\ref{thm:main} and~\ref{thm:main_iterative}).

A number of questions remain open. The standing assumption \eqref{A2} of a
real barotropic speed of sound excludes closures with a spinodal region,
such as the van der Waals equation of state~\cite{vanderWaals1873}, for
which the Riemann problem loses hyperbolicity and its resolution requires a
kinetic relation or a regularization. Thus, the natural question arises how
the comparison with the hydrodynamical problem could possibly extend to
such a configuration. A second question concerns
composite waves. The bounds of Theorem~\ref{thm:main} hold irrespective of
the wave structure, but the iteration of Theorem~\ref{thm:main_iterative},
which tightens them to within a fraction of a percent for elementary
waves, requires elementary outer waves: the volume bounds of
Lemma~\ref{lem:matching_volume_bound} and the deviation bounds of
Lemma~\ref{lem:deviation_bounds} compare a single shock or fan with its
hydrodynamical counterpart, and the star pressure bracket of
Corollary~\ref{cor:deviation_star_pressures} rests on this comparison.
Once an outer wave is composite, as in the configurations of
Section~\ref{sec:composite_tests}, we have no analogue of these
ingredients, and the initial bounds are all that Theorem~\ref{thm:main}
offers. How to construct an iteration for composite waves is an open
problem.


\section*{Acknowledgments}
NF gratefully acknowledges support by the Mobil Chair in Computational
Science held by Jean-Luc Guermond, which funded a six-month research stay
at Texas A\&M University in 2025. MM acknowledges partial support by the
National Science Foundation under grant DMS-2045636 and by the Air Force
Office of Scientific Research, USAF, under grant/contract number
FA9550-23-1-0007. We thank our colleagues Jean-Luc Guermond, Bojan Popov
and Ignacio Tomas for fruitful discussions. We would also like to thank
Sergey Gavrilyuk for discussions and for making us aware of the work of
Smith~\cite{Smith1979}.

\clearpage
\newpage
\begin{algorithm2e}[p]
  \Fn{\Hdr{jump\_pressure}$(W,\Delta)$}{
    \tcp{root of the continued shock branch \eqref{eq:f_shock_h_continued}; $\Delta=0$ returns $\ph$}
    \Return $q$ with $f\S_{\mathrm h}(q;W)=\Delta$\;
  }
  \vspace{1em}
  \Fn{\Hdr{wavespeed\_bound}$(W,\underbar P,\overbar P,s)$}{
    \tcp{volume bracket from \eqref{eq:rankine_hugoniot_h}, $\xi_h\S$ continued past $\ph$}
    $\bar v \;\leftarrow\; \xi_h\S(\overbar P;W)$;\quad
    $\underbar v \;\leftarrow\; \xi_h\S(\underbar P;W)$\;
    \tcp{hydrodynamical wavespeed \eqref{eq:main_lambda} and barotropic correction \eqref{eq:main_wavespeed}}
    $\bar\lambda \;\leftarrow\; u + s\,v\,Q_{\mathrm h}(\overbar P;W)$\;
    $M \;\leftarrow\; \max_{\bar v\le\tau\le\underbar v}\hatab^2(\tau)/\tau^2$\;
    \Return $u + s\sqrt{(\bar\lambda-u)^2+2\,v^2 M}$\;
  }
  \vspace{1em}
  \Fn{\Hdr{general\_bounds}$(W\L,W\R)$}{
    \tcp{Step 1: one-sided jump estimates \eqref{eq:main_onesided}}
    $\Delta u \;\leftarrow\; u\L-u\R$\;
    $\overbar q\L \;\leftarrow\;
      \FuncSty{jump\_pressure}\big(W\L,\max(\Delta u,0)\big)$;\quad
    $\overbar q\R \;\leftarrow\;
      \FuncSty{jump\_pressure}\big(W\R,\max(\Delta u,0)\big)$\;
    $\underbar q\L \;\leftarrow\;
      \FuncSty{jump\_pressure}\big(W\L,\min(\Delta u,0)\big)$;\quad
    $\underbar q\R \;\leftarrow\;
      \FuncSty{jump\_pressure}\big(W\R,\min(\Delta u,0)\big)$\;
    \tcp{Step 2: pressure bounds \eqref{eq:main_pressure_bound}}
    $\barph\starL \;\leftarrow\; \max\big(\overbar q\L,\,p\R-\hatpb(v\L)\big)$;\quad
    $\barph\starR \;\leftarrow\; \max\big(\overbar q\R,\,p\L-\hatpb(v\R)\big)$\;
    $\underph\starL \;\leftarrow\; \min\big(\underbar q\L,\,p\R-\hatpb(v\L)\big)$;\quad
    $\underph\starR \;\leftarrow\; \min\big(\underbar q\R,\,p\L-\hatpb(v\R)\big)$\;
    \tcp{Step 3: wavespeed bounds \eqref{eq:main_wavespeed}}
    $\lambda^-_1 \;\leftarrow\;
      \FuncSty{wavespeed\_bound}\big(W\L,\underph\starL,\barph\starL,-1\big)$;\quad
    $\lambda^+_3 \;\leftarrow\;
      \FuncSty{wavespeed\_bound}\big(W\R,\underph\starR,\barph\starR,+1\big)$\;
    \Return $\big(\underph\starL,\barph\starL,\underph\starR,\barph\starR\big)$
      and $\lambda_{\max}=\max\big(|\lambda^-_1|,|\lambda^+_3|\big)$\;
  }\vspace{1.75em}
  \caption{Wavespeed bounds of Theorem~\ref{thm:main}.
  $W=[\rho,u,\ph]$ denotes a state with specific volume $v=1/\rho$ and full
  pressure $p=\ph+\hatpb(v)$.}
  \label{alg:general}
\end{algorithm2e}
\begin{algorithm2e}[p]
  \Fn{\Hdr{admissible}$(W,\underbar P,\overbar P)$}{
    \tcp{smallness condition of Lemma~\ref{lem:matching_volume_bound}}
    $d \;\leftarrow\; 2\,(p_\infty+\ph)$;\quad
    $m \;\leftarrow\; \max_{\xi_h\S(\overbar P;W)\le\tau\le v}\hatab^2(\tau)/\tau^2$\;
    \If{$4\,v\,m>d$}{\Return false\;}
    $(\underbar V,\overbar V) \;\leftarrow\; \FuncSty{volume\_bounds}(W,\underbar P,\overbar P)$\;
    \tcp{nontrivial lower volume bound, Lemma~\ref{lem:deviation_bounds}}
    \If{$\underbar V\le0$}{\Return false\;}
    \tcp{smallness condition of Lemma~\ref{lem:deviation_bounds}}
    $\Delta_{\mathrm b}^+ \;\leftarrow\; \max\big(\hatpb(\underbar V)-\hatpb(v),\,0\big)$\;
    \If{$\Delta_{\mathrm b}^+\ge d$}{\Return false\;}
    \tcp{convexity criterion \eqref{eq:kappa_criterion} of Lemma~\ref{lem:elementary_criterion}}
    $I \;\leftarrow\; \big[\min(v,\underbar V),\,\max(v,\overbar V)\big]$\;
    $\kappa \;\leftarrow\; \min_{\tau\in I}\partial_v\partial_v\hatph(\tau,\sh)\big|_{\sh}
      +\min_{\tau\in I}\partial_v\partial_v\hatpb(\tau)$\;
    \Return $\kappa>0$\;
  }
  \vspace{1em}
  \Fn{\Hdr{volume\_bounds}$(W,\underbar P,\overbar P)$}{
    \tcp{Lemma~\ref{lem:matching_volume_bound}: shock branch above $\ph$, isentrope below}
    $d \;\leftarrow\; 2\,(p_\infty+\ph)$;\quad
    $m \;\leftarrow\; \max_{\xi_h\S(\overbar P;W)\le\tau\le v}\hatab^2(\tau)/\tau^2$\;
    \eIf{$\overbar P\ge\ph$}{
      $\overbar c \;\leftarrow\; v-\xi_h\S(\overbar P;W)$;\quad
      $\underbar V \;\leftarrow\; v-2\,\overbar c\big/\big(1+\sqrt{1-4\,m\,\overbar c/d}\,\big)$
      \tcp*{\eqref{eq:matching_volume_bound}}
    }{
      $\underbar V \;\leftarrow\; \xi_h\R(\overbar P;W)$
      \tcp*{\eqref{eq:matching_volume_bound_rarefaction}}
    }
    \eIf{$\underbar P\ge\ph$}{
      $\underbar c \;\leftarrow\; v-\xi_h\S(\underbar P;W)$;\quad
      $\overbar V \;\leftarrow\; v-2\,\underbar c\big/\big(1+\sqrt{1+4\,m\,\underbar c/d}\,\big)$
      \tcp*{\eqref{eq:matching_volume_bound_upper}}
    }{
      $\overbar V \;\leftarrow\; \xi_h\R(\underbar P;W)$
      \tcp*{\eqref{eq:matching_volume_bound_rarefaction}}
    }
    \Return $(\underbar V,\overbar V)$\;
  }
  \vspace{1em}
  \Fn{\Hdr{deviation\_bounds}$(W,\underbar P,\overbar P,\underbar V,\overbar V)$}{
    \tcp{Lemma~\ref{lem:deviation_bounds}; $\eta\S$ vanishes for $\overbar P\le\ph$ and $\eta\R$ for $\underbar P\ge\ph$}
    $d \;\leftarrow\; 2\,(p_\infty+\ph)$\;
    $\Delta_{\mathrm b}^+ \;\leftarrow\; \max\big(\hatpb(\underbar V)-\hatpb(v),\,0\big)$;\quad
    $\Delta_{\mathrm b}^- \;\leftarrow\; \max\big(\hatpb(v)-\hatpb(\overbar V),\,0\big)$
    \tcp*{\eqref{eq:deviation_delta_b}}
    $\eta\S \;\leftarrow\; \dfrac{v}{2\,\ah(W)}\,
      \dfrac{d+\overbar P-\ph}{d-\Delta_{\mathrm b}^+}\,\Delta_{\mathrm b}^+$;\quad
    $\eta\R \;\leftarrow\; \dfrac{\overbar V}{2\,\ah(\overbar V,\underbar P)}\,\Delta_{\mathrm b}^-$
    \tcp*{\eqref{eq:deviation_eta}}
    \Return $(\eta\S,\eta\R)$\;
  }
  \vspace{1em}
  \Fn{\Hdr{shifted\_pressure}$(W\L,W\R,\eta)$}{
    \tcp{root of the shifted pressure equation \eqref{eq:pressure_equation_shifted}, see Remark~\ref{rem:shifted_no_solution}}
    \Return $\Phi$ with
      $f_{\mathrm h}(\Phi;W\L)+f_{\mathrm h}(\Phi;W\R)+u\R-u\L-\eta=0$,
      or $\emptyset$ if no root exists\;
  }
  \vspace{1.75em}
  \caption{Auxiliary routines for
  Algorithm~\ref{alg:elementary}. $\sh$ is the specific entropy of $W$.}
  \label{alg:elementary_aux}
\end{algorithm2e}
\begin{algorithm2e}[p]
  \Fn{\Hdr{improved\_wavespeed\_bound}$(W,\underbar P,\overbar P,s)$}{
    \tcp{Corollary~\ref{cor:main_closing}: shock if the bracket lies above $\ph$}
    \If{$\ph<\underbar P$}{
      $\bar v \;\leftarrow\; \xi_h\S(\overbar P;W)$;\quad
      $\bar\lambda \;\leftarrow\; u + s\,v\,Q_{\mathrm h}(\overbar P;W)$\;
      $\theta \;\leftarrow\; \big(\hatpb(\bar v)-\hatpb(v)\big)\big/\big(\underbar P-\ph\big)$;\quad
      $M \;\leftarrow\; \max_{\bar v\le\tau\le v}\hatab^2(\tau)/\tau^2$\;
      \tcp{tighter of Corollary~\ref{cor:main_closing} and the clipped bound of Theorem~\ref{thm:main_iterative}}
      \Return $u + s\sqrt{(\bar\lambda-u)^2+v^2 M+\min\big(\theta\,(\bar\lambda-u)^2,\,v^2 M\big)}$\;
    }
    \tcp{rarefaction if the bracket lies below $\ph$, head speed \eqref{eq:lambda_rarefaction}}
    \If{$\overbar P<\ph$}{
      \Return $u + s\sqrt{\ah^2(W)+\hatab^2(v)}$\;
    }
    \tcp{otherwise the bracket contains $\ph$, return bounds of Theorem~\ref{thm:main_iterative}:}
    \Return \FuncSty{wavespeed\_bound}$\big(W,\min(\underbar P,\ph),\max(\overbar P,\ph),s\big)$\;
  }
  \vspace{1em}
  \Fn{\Hdr{regular\_bounds}$(W\L,W\R,\,N_{\max},\,\mathrm{tol})$}{
    \tcp{Step 0: initial bracket, Theorem~\ref{thm:main}}
    $\big(\underbar P_{\mathrm L},\overbar P_{\mathrm L},
      \underbar P_{\mathrm R},\overbar P_{\mathrm R}\big)
      \;\leftarrow\;\FuncSty{general\_bounds}(W\L,W\R)$\;
    \tcp{Step 1: admissibility; on failure fall back to Theorem~\ref{thm:main}}
    \If{\Not \FuncSty{admissible}$(W\L,\underbar P_{\mathrm L},\overbar P_{\mathrm L})$
        \Or \Not \FuncSty{admissible}$(W\R,\underbar P_{\mathrm R},\overbar P_{\mathrm R})$}{
      \Return the bracket and $\lambda_{\max}$ of \FuncSty{general\_bounds}\;
    }
    \tcp{Step 2: hydrodynamical star pressure \eqref{eq:pressure_equation_h} and bracket \eqref{eq:main_pistar}}
    $\tilde p_h\star \;\leftarrow\; p$ with
      $f_{\mathrm h}(p;W\L)+f_{\mathrm h}(p;W\R)+u\R-u\L=0$\;
    $\underbar\pi\star \;\leftarrow\; \min\big(\tilde p_h\star,\ph\L,\ph\R\big)$;\quad
    $\bar\pi\star \;\leftarrow\; \max\big(\tilde p_h\star,\ph\L,\ph\R\big)$\;
    \tcp{Step 3: iteration \eqref{eq:main_bootstrap}}
    \For{$n=1,\ldots,N_{\max}$}{
      $\big(\underbar V_{\mathrm L},\overbar V_{\mathrm L}\big) \;\leftarrow\;
        \FuncSty{volume\_bounds}(W\L,\underbar P_{\mathrm L},\overbar P_{\mathrm L})$\;
      $\big(\underbar V_{\mathrm R},\overbar V_{\mathrm R}\big) \;\leftarrow\;
        \FuncSty{volume\_bounds}(W\R,\underbar P_{\mathrm R},\overbar P_{\mathrm R})$\;
      $\underbar\delta \;\leftarrow\; \hatpb(\underbar V_{\mathrm L})-\hatpb(\overbar V_{\mathrm R})$;\quad
      $\bar\delta \;\leftarrow\; \hatpb(\underbar V_{\mathrm R})-\hatpb(\overbar V_{\mathrm L})$\;
      \tcp{deviation bounds \eqref{eq:deviation_eta_sum} and star pressure bracket \eqref{eq:deviation_star_pressures}}
      $\big(\eta\S_{\mathrm L},\eta\R_{\mathrm L}\big) \;\leftarrow\;
        \FuncSty{deviation\_bounds}(W\L,\underbar P_{\mathrm L},\overbar P_{\mathrm L},
        \underbar V_{\mathrm L},\overbar V_{\mathrm L})$\;
      $\big(\eta\S_{\mathrm R},\eta\R_{\mathrm R}\big) \;\leftarrow\;
        \FuncSty{deviation\_bounds}(W\R,\underbar P_{\mathrm R},\overbar P_{\mathrm R},
        \underbar V_{\mathrm R},\overbar V_{\mathrm R})$\;
      $\underbar\eta \;\leftarrow\; -\eta\S_{\mathrm L}-\eta\S_{\mathrm R}$;\quad
      $\bar\eta \;\leftarrow\; \eta\R_{\mathrm L}+\eta\R_{\mathrm R}$\;
      $\underbar\pi \;\leftarrow\; \max\big(\underbar\pi\star,\,
        \FuncSty{shifted\_pressure}(W\L,W\R,\underbar\eta)\big)$\;
      $\bar\pi \;\leftarrow\; \min\big(\bar\pi\star,\,
        \FuncSty{shifted\_pressure}(W\L,W\R,\bar\eta)\big)$\;
      $\overbar P_{\mathrm L}' \;\leftarrow\; \min\big(\overbar P_{\mathrm L},\,\bar\pi+\max(\bar\delta,0)\big)$;\quad
      $\overbar P_{\mathrm R}' \;\leftarrow\; \min\big(\overbar P_{\mathrm R},\,\bar\pi+\max(\underbar\delta,0)\big)$\;
      $\underbar P_{\mathrm L}' \;\leftarrow\; \max\big(\underbar P_{\mathrm L},\,\underbar\pi-\max(\underbar\delta,0)\big)$;\quad
      $\underbar P_{\mathrm R}' \;\leftarrow\; \max\big(\underbar P_{\mathrm R},\,\underbar\pi-\max(\bar\delta,0)\big)$\;
      $\big(\underbar P_{\mathrm L},\overbar P_{\mathrm L},
        \underbar P_{\mathrm R},\overbar P_{\mathrm R}\big)
        \;\leftarrow\;\big(\underbar P_{\mathrm L}',\overbar P_{\mathrm L}',
        \underbar P_{\mathrm R}',\overbar P_{\mathrm R}'\big)$\;
      \If{relative change of all four bounds $<\mathrm{tol}$}{\Break\;}
    }
    \tcp{Step 4: wavespeed bounds, Corollary~\ref{cor:main_closing}}
    $\lambda^-_1 \;\leftarrow\;
      \FuncSty{improved\_wavespeed\_bound}\big(W\L,\underbar P_{\mathrm L},\overbar P_{\mathrm L},-1\big)$\;
    $\lambda^+_3 \;\leftarrow\;
      \FuncSty{improved\_wavespeed\_bound}\big(W\R,\underbar P_{\mathrm R},\overbar P_{\mathrm R},+1\big)$\;
    \Return $\big(\underbar P_{\mathrm L},\overbar P_{\mathrm L},
      \underbar P_{\mathrm R},\overbar P_{\mathrm R}\big)$
      and $\lambda_{\max}=\max\big(|\lambda^-_1|,|\lambda^+_3|\big)$\;
  }\vspace{1.75em}
  \caption{Iterated wavespeed bounds of Theorem~\ref{thm:main_iterative}
  and Corollary~\ref{cor:main_closing}.}
  \label{alg:elementary}
\end{algorithm2e}


\begin{table}[p]
  \centering
  \setlength{\tabcolsep}{5.39pt}
  \captionsetup[subfloat]{captionskip=3pt,farskip=7.5pt}
  \subfloat[$c_0=0$]{%
    \scriptsize
    \begin{tabular}{l*{8}{r}}
      \toprule
      & \multicolumn{4}{c}{exact \eqref{eq:riemann_problem}} & \multicolumn{4}{c}{hydrodynamical \eqref{eq:riemann_problem_h}} \\
      \cmidrule(lr){2-5} \cmidrule(lr){6-9}
      & $\ph\starL$ & $\ph\starR$ & $\lambda^-_1$ & $\lambda^+_3$ & $\tilde p_h\star$ & $\lambda^-_{1,\mathrm h}$ & $\lambda^+_{3,\mathrm h}$ & $\Delta\lambda_{\max}$ \\[0.25em]
      E1   & 6.544e-4 & 6.544e-4 & -0.305505 & \textbf{0.886139} & 6.544e-4 & -0.305505 & \textbf{0.886139} & +0\% \\
      E2   & 0.30313 & 0.30313 & -1.18322 & \textbf{1.75216} & 0.30313 & -1.18322 & \textbf{1.75216} & +0\% \\
      E3   & 2.4661 & 2.4661 & \textbf{-2.63357} & 2.47932 & 2.4661 & \textbf{-2.63357} & 2.47932 & +0\% \\
      E4   & 460.894 & 460.894 & -27.4166 & \textbf{33.5175} & 460.894 & -27.4166 & \textbf{33.5175} & +0\% \\
      E5   & 1.02391 & 1.02391 & \textbf{-1.17528} & 1.17528 & 1.02391 & \textbf{-1.17528} & 1.17528 & +0\% \\
      E6   & 10.4769 & 10.4769 & -1.52708 & \textbf{2.81737} & 10.4769 & -1.52708 & \textbf{2.81737} & +0\% \\
      E7   & 10.4194 & 10.4194 & -0.619073 & \textbf{2.19288} & 10.4194 & -0.619073 & \textbf{2.19288} & +0\% \\
      E8   & 1.17562 & 1.17562 & -0.284928 & \textbf{0.647425} & 1.17562 & -0.284928 & \textbf{0.647425} & +0\% \\
      E9   & 9.86846 & 9.86846 & \textbf{-26.7075} & 26.7075 & 9.86846 & \textbf{-26.7075} & 26.7075 & +0\% \\
      E10  & 8.66235 & 8.66235 & \textbf{-3.99166} & 1.43322 & 8.66235 & \textbf{-3.99166} & 1.43322 & +0\% \\
      \bottomrule
    \end{tabular}}
  \\[0.75em]
  \subfloat[$c_0=0.1\,c_0^{\max}$]{%
    \scriptsize
    \begin{tabular}{l*{8}{r}}
      \toprule
      & \multicolumn{4}{c}{exact \eqref{eq:riemann_problem}} & \multicolumn{4}{c}{hydrodynamical \eqref{eq:riemann_problem_h}} \\
      \cmidrule(lr){2-5} \cmidrule(lr){6-9}
      & $\ph\starL$ & $\ph\starR$ & $\lambda^-_1$ & $\lambda^+_3$ & $\tilde p_h\star$ & $\lambda^-_{1,\mathrm h}$ & $\lambda^+_{3,\mathrm h}$ & $\Delta\lambda_{\max}$ \\[0.25em]
      E1   & 6.544e-4 & 6.544e-4 & -0.305505 & \textbf{0.886139} & \textcolor{red}{\textbf{6.544e-4}} & -0.305505 & \textbf{0.886139} & +0\% \\
      E2   & 0.303171 & 0.303174 & -1.18533 & \textbf{1.75228} & \textcolor{red}{\textbf{0.30313}} & -1.18322 & \textbf{1.75216} & \textcolor{red}{\textbf{\boldmath$<0.01\%$}} \\
      E3   & 2.46611 & 2.466 & \textbf{-2.63357} & 2.47935 & \textcolor{red}{\textbf{2.4661}} & \textbf{-2.63357} & 2.47932 & \textcolor{red}{\textbf{\boldmath$<0.01\%$}} \\
      E4   & 460.894 & 460.894 & -27.4166 & \textbf{33.5175} & \textcolor{red}{\textbf{460.894}} & -27.4166 & \textbf{33.5175} & $<0.01\%$ \\
      E5   & 1.02387 & 1.02387 & \textbf{-1.17724} & 1.17724 & 1.02391 & \textbf{-1.17528} & 1.17528 & \textcolor{red}{\textbf{-0.17\%}} \\
      E6   & 10.4769 & 10.4769 & -1.52708 & \textbf{2.81738} & \textcolor{red}{\textbf{10.4769}} & -1.52708 & \textbf{2.81737} & \textcolor{red}{\textbf{\boldmath$<0.01\%$}} \\
      E7   & 10.4194 & 10.4194 & -0.619073 & \textbf{2.19288} & \textcolor{red}{\textbf{10.4194}} & -0.619073 & \textbf{2.19288} & \textcolor{red}{\textbf{\boldmath$<0.01\%$}} \\
      E8   & 1.17562 & 1.17562 & -0.284928 & \textbf{0.647425} & \textcolor{red}{\textbf{1.17562}} & -0.284928 & \textbf{0.647425} & \textcolor{red}{\textbf{\boldmath$<0.01\%$}} \\
      E9   & 9.86869 & 9.86869 & \textbf{-26.7547} & 26.7547 & \textcolor{red}{\textbf{9.86846}} & \textbf{-26.7075} & 26.7075 & \textcolor{red}{\textbf{-0.18\%}} \\
      E10  & 8.66483 & 8.65841 & \textbf{-3.99671} & 1.43322 & \textcolor{red}{\textbf{8.66235}} & \textbf{-3.99166} & 1.43322 & \textcolor{red}{\textbf{-0.13\%}} \\
      \bottomrule
    \end{tabular}}
  \\[0.75em]
  \subfloat[$c_0=0.5\,c_0^{\max}$]{%
    \scriptsize
    \begin{tabular}{l*{8}{r}}
      \toprule
      & \multicolumn{4}{c}{exact \eqref{eq:riemann_problem}} & \multicolumn{4}{c}{hydrodynamical \eqref{eq:riemann_problem_h}} \\
      \cmidrule(lr){2-5} \cmidrule(lr){6-9}
      & $\ph\starL$ & $\ph\starR$ & $\lambda^-_1$ & $\lambda^+_3$ & $\tilde p_h\star$ & $\lambda^-_{1,\mathrm h}$ & $\lambda^+_{3,\mathrm h}$ & $\Delta\lambda_{\max}$ \\[0.25em]
      E1   & 6.544e-4 & 6.544e-4 & -0.305505 & \textbf{0.886139} & \textcolor{red}{\textbf{6.544e-4}} & -0.305505 & \textbf{0.886139} & +0\% \\
      E2   & 0.304139 & 0.304211 & -1.23491 & \textbf{1.75526} & \textcolor{red}{\textbf{0.30313}} & -1.18322 & \textbf{1.75216} & \textcolor{red}{\textbf{-0.18\%}} \\
      E3   & 2.46634 & 2.46377 & \textbf{-2.63358} & 2.47999 & \textcolor{red}{\textbf{2.4661}} & \textbf{-2.63357} & 2.47932 & \textcolor{red}{\textbf{\boldmath$<0.01\%$}} \\
      E4   & 460.894 & 460.894 & -27.4166 & \textbf{33.5175} & \textcolor{red}{\textbf{460.894}} & -27.4166 & \textbf{33.5175} & $<0.01\%$ \\
      E5   & 1.02296 & 1.02296 & \textbf{-1.22356} & 1.22356 & 1.02391 & \textbf{-1.17528} & 1.17528 & \textcolor{red}{\textbf{-3.95\%}} \\
      E6   & 10.4767 & 10.4769 & -1.52706 & \textbf{2.81782} & \textcolor{red}{\textbf{10.4769}} & -1.52708 & \textbf{2.81737} & \textcolor{red}{\textbf{-0.02\%}} \\
      E7   & 10.4194 & 10.4194 & -0.619073 & \textbf{2.19288} & \textcolor{red}{\textbf{10.4194}} & -0.619073 & \textbf{2.19288} & \textcolor{red}{\textbf{\boldmath$<0.01\%$}} \\
      E8   & 1.17562 & 1.17562 & -0.284928 & \textbf{0.647425} & \textcolor{red}{\textbf{1.17562}} & -0.284928 & \textbf{0.647425} & \textcolor{red}{\textbf{\boldmath$<0.01\%$}} \\
      E9   & 9.87394 & 9.87394 & \textbf{-27.8634} & 27.8634 & \textcolor{red}{\textbf{9.86846}} & \textbf{-26.7075} & 26.7075 & \textcolor{red}{\textbf{-4.15\%}} \\
      E10  & 8.72319 & 8.56431 & \textbf{-4.11591} & 1.43322 & \textcolor{red}{\textbf{8.66235}} & \textbf{-3.99166} & 1.43322 & \textcolor{red}{\textbf{-3.02\%}} \\
      \bottomrule
    \end{tabular}}
  \\[0.75em]
  \subfloat[$c_0=0.99\,c_0^{\max}$]{%
    \scriptsize
    \begin{tabular}{l*{8}{r}}
      \toprule
      & \multicolumn{4}{c}{exact \eqref{eq:riemann_problem}} & \multicolumn{4}{c}{hydrodynamical \eqref{eq:riemann_problem_h}} \\
      \cmidrule(lr){2-5} \cmidrule(lr){6-9}
      & $\ph\starL$ & $\ph\starR$ & $\lambda^-_1$ & $\lambda^+_3$ & $\tilde p_h\star$ & $\lambda^-_{1,\mathrm h}$ & $\lambda^+_{3,\mathrm h}$ & $\Delta\lambda_{\max}$ \\[0.25em]
      E1   & 6.544e-4 & 6.544e-4 & -0.305505 & \textbf{0.886139} & \textcolor{red}{\textbf{6.544e-4}} & -0.305505 & \textbf{0.886139} & \textcolor{red}{\textbf{\boldmath$<0.01\%$}} \\
      E2   & 0.306962 & 0.307249 & -1.37479 & \textbf{1.76399} & \textcolor{red}{\textbf{0.30313}} & -1.18322 & \textbf{1.75216} & \textcolor{red}{\textbf{-0.67\%}} \\
      E3   & 2.46707 & 2.45699 & \textbf{-2.63361} & 2.48197 & \textcolor{red}{\textbf{2.4661}} & \textbf{-2.63357} & 2.47932 & \textcolor{red}{\textbf{\boldmath$<0.01\%$}} \\
      E4   & 460.894 & 460.893 & -27.4166 & \textbf{33.5175} & \textcolor{red}{\textbf{460.894}} & -27.4166 & \textbf{33.5175} & $<0.01\%$ \\
      E5   & 1.02072 & 1.02072 & \textbf{-1.35528} & 1.35528 & 1.02391 & \textbf{-1.17528} & 1.17528 & \textcolor{red}{\textbf{-13.3\%}} \\
      E6   & 10.4761 & 10.477 & -1.52699 & \textbf{2.81913} & \textcolor{red}{\textbf{10.4769}} & -1.52708 & \textbf{2.81737} & \textcolor{red}{\textbf{-0.06\%}} \\
      E7   & 10.4194 & 10.4194 & -0.619073 & \textbf{2.19288} & \textcolor{red}{\textbf{10.4194}} & -0.619073 & \textbf{2.19288} & \textcolor{red}{\textbf{\boldmath$<0.01\%$}} \\
      E8   & 1.17562 & 1.17562 & -0.284928 & \textbf{0.647427} & \textcolor{red}{\textbf{1.17562}} & -0.284928 & \textbf{0.647425} & \textcolor{red}{\textbf{\boldmath$<0.01\%$}} \\
      E9   & 9.88672 & 9.88672 & \textbf{-30.9913} & 30.9913 & \textcolor{red}{\textbf{9.86846}} & \textbf{-26.7075} & 26.7075 & \textcolor{red}{\textbf{-13.8\%}} \\
      E10  & 8.88861 & 8.28532 & \textbf{-4.45782} & 1.43322 & \textcolor{red}{\textbf{8.66235}} & \textbf{-3.99166} & 1.43322 & \textcolor{red}{\textbf{-10.5\%}} \\
      \bottomrule
    \end{tabular}}
  \caption{Star pressures and extremal wavespeeds of the exact Riemann
  problem and of the reduced hydrodynamical problem for the configurations
  E1--E10 of Table~\ref{tab:elementary_configurations}. In each pair of
  wavespeeds the entry of larger modulus, printed in bold, determines
  $\lambda_{\max}$; a hydrodynamical star pressure that falls below the
  larger of the two exact ones is flagged in red. $\Delta\lambda_{\max}$ is
  the relative deviation of the hydrodynamical from the exact maximal
  wavespeed.}
  \label{tab:elementary_wavespeeds}
\end{table}

\begin{table}[p]
  \centering
  \setlength{\tabcolsep}{5.5pt}
  \captionsetup[subfloat]{captionskip=3pt,farskip=7.5pt}
  \subfloat[$c_0=0$]{%
    \scriptsize
    \begin{tabular}{l*{7}{r}}
      \toprule
      & \multicolumn{4}{c}{pressure bounds \eqref{eq:main_pressure_bound}} & \multicolumn{3}{c}{wavespeed bounds \eqref{eq:main_wavespeed}} \\
      \cmidrule(lr){2-5} \cmidrule(lr){6-8}
      & $\underph\starL$ & $\barph\starL$ & $\underph\starR$ & $\barph\starR$ & $\lambda^-_1$ & $\lambda^+_3$ & $\Delta\lambda_{\max}$ \\[0.25em]
      E1   & 6.67e-11 & 0.066667 & 6.67e-11 & 0.066667 & -0.305505 & \textbf{8.94427} & +909\% \\
      E2   & 0.1 & 1 & 0.1 & 1 & -1.18322 & \textbf{3.1241} & +78.3\% \\
      E3   & 0.571 & 4.70105 & 0.571 & 3.528 & \textbf{-3.07858} & 2.94883 & +16.9\% \\
      E4   & 0.01 & 1e3 & 0.01 & 1e3 & -27.4166 & \textbf{44.641} & +33.2\% \\
      E5   & 1 & 1.0483 & 1 & 1.0483 & \textbf{-1.18746} & 1.18746 & +1.04\% \\
      E6   & 1 & 11.6515 & 1 & 20.2588 & -1.66287 & \textbf{4.12938} & +46.6\% \\
      E7   & 0.1 & 12.2156 & 0.1 & 16.6259 & -0.711555 & \textbf{2.81293} & +28.3\% \\
      E8   & 0.1 & 1.70878 & 0.1 & 2.18882 & -0.393511 & \textbf{0.938819} & +45\% \\
      E9   & 9.73841 & 10 & 9.73841 & 10 & \textbf{-26.7075} & 26.7075 & +0\% \\
      E10  & 8.27317 & 10 & 5.39672 & 10 & \textbf{-3.99166} & 1.43322 & +0\% \\
      \bottomrule
    \end{tabular}}
  \\[0.75em]
  \subfloat[$c_0=0.1\,c_0^{\max}$]{%
    \scriptsize
    \begin{tabular}{l*{7}{r}}
      \toprule
      & \multicolumn{4}{c}{pressure bounds \eqref{eq:main_pressure_bound}} & \multicolumn{3}{c}{wavespeed bounds \eqref{eq:main_wavespeed}} \\
      \cmidrule(lr){2-5} \cmidrule(lr){6-8}
      & $\underph\starL$ & $\barph\starL$ & $\underph\starR$ & $\barph\starR$ & $\lambda^-_1$ & $\lambda^+_3$ & $\Delta\lambda_{\max}$ \\[0.25em]
      E1   & 6.6e-11 & 0.066667 & 6.67e-11 & 0.066667 & -0.305505 & \textbf{8.94427} & +909\% \\
      E2   & 0.099503 & 1 & 0.1 & 1.0005 & -1.18743 & \textbf{3.12502} & +78.3\% \\
      E3   & 0.571 & 4.70105 & 0.571 & 3.528 & \textbf{-3.07859} & 2.94956 & +16.9\% \\
      E4   & 0.01 & 1e3 & 0.01 & 1e3 & -27.4166 & \textbf{44.641} & +33.2\% \\
      E5   & 1 & 1.0483 & 1 & 1.0483 & \textbf{-1.19141} & 1.19141 & +1.2\% \\
      E6   & 1 & 11.6515 & 1.00001 & 20.2588 & -1.66287 & \textbf{4.1294} & +46.6\% \\
      E7   & 0.1 & 12.2156 & 0.1 & 16.6259 & -0.711555 & \textbf{2.81293} & +28.3\% \\
      E8   & 0.1 & 1.70878 & 0.1 & 2.18882 & -0.393511 & \textbf{0.938831} & +45\% \\
      E9   & 9.73841 & 10 & 9.73841 & 10 & \textbf{-26.8018} & 26.8018 & +0.18\% \\
      E10  & 8.27317 & 10.0038 & 5.39672 & 10 & \textbf{-4.00236} & 1.43322 & +0.15\% \\
      \bottomrule
    \end{tabular}}
  \\[0.75em]
  \subfloat[$c_0=0.5\,c_0^{\max}$]{%
    \scriptsize
    \begin{tabular}{l*{7}{r}}
      \toprule
      & \multicolumn{4}{c}{pressure bounds \eqref{eq:main_pressure_bound}} & \multicolumn{3}{c}{wavespeed bounds \eqref{eq:main_wavespeed}} \\
      \cmidrule(lr){2-5} \cmidrule(lr){6-8}
      & $\underph\starL$ & $\barph\starL$ & $\underph\starR$ & $\barph\starR$ & $\lambda^-_1$ & $\lambda^+_3$ & $\Delta\lambda_{\max}$ \\[0.25em]
      E1   & 5e-11 & 0.066667 & 6.67e-11 & 0.066667 & -0.305505 & \textbf{8.94427} & +909\% \\
      E2   & 0.087581 & 1 & 0.1 & 1.01242 & -1.28452 & \textbf{3.14702} & +79.3\% \\
      E3   & 0.571005 & 4.70105 & 0.571 & 3.528 & \textbf{-3.07864} & 2.9669 & +16.9\% \\
      E4   & 0.01 & 1e3 & 0.01 & 1e3 & -27.4166 & \textbf{44.6411} & +33.2\% \\
      E5   & 1 & 1.0483 & 1 & 1.0483 & \textbf{-1.28254} & 1.28254 & +4.82\% \\
      E6   & 1 & 11.6515 & 1.00035 & 20.2588 & -1.66288 & \textbf{4.13006} & +46.6\% \\
      E7   & 0.1 & 12.2156 & 0.1 & 16.6259 & -0.711555 & \textbf{2.81294} & +28.3\% \\
      E8   & 0.1 & 1.70878 & 0.100001 & 2.18882 & -0.393511 & \textbf{0.939113} & +45.1\% \\
      E9   & 9.73841 & 10 & 9.73841 & 10 & \textbf{-28.9728} & 28.9728 & +3.98\% \\
      E10  & 8.27317 & 10.0945 & 5.39672 & 10 & \textbf{-4.25207} & 1.43322 & +3.31\% \\
      \bottomrule
    \end{tabular}}
  \\[0.75em]
  \subfloat[$c_0=0.99\,c_0^{\max}$]{%
    \scriptsize
    \begin{tabular}{l*{7}{r}}
      \toprule
      & \multicolumn{4}{c}{pressure bounds \eqref{eq:main_pressure_bound}} & \multicolumn{3}{c}{wavespeed bounds \eqref{eq:main_wavespeed}} \\
      \cmidrule(lr){2-5} \cmidrule(lr){6-8}
      & $\underph\starL$ & $\barph\starL$ & $\underph\starR$ & $\barph\starR$ & $\lambda^-_1$ & $\lambda^+_3$ & $\Delta\lambda_{\max}$ \\[0.25em]
      E1   & 1.33e-12 & 0.066667 & 6.67e-11 & 0.066667 & -0.305505 & \textbf{8.94427} & +909\% \\
      E2   & 0.051312 & 1 & 0.1 & 1.04869 & -1.54276 & \textbf{3.21391} & +82.2\% \\
      E3   & 0.571019 & 4.70105 & 0.571 & 3.52798 & \textbf{-3.07879} & 3.01905 & +16.9\% \\
      E4   & 0.01 & 1e3 & 0.01 & 1e3 & -27.4167 & \textbf{44.6411} & +33.2\% \\
      E5   & 1 & 1.0483 & 1 & 1.0483 & \textbf{-1.52713} & 1.52713 & +12.7\% \\
      E6   & 1 & 11.6515 & 1.00136 & 20.2588 & -1.66288 & \textbf{4.13205} & +46.6\% \\
      E7   & 0.1 & 12.2156 & 0.1 & 16.6259 & -0.711555 & \textbf{2.81295} & +28.3\% \\
      E8   & 0.1 & 1.70878 & 0.100003 & 2.18882 & -0.393511 & \textbf{0.93997} & +45.2\% \\
      E9   & 9.73841 & 10 & 9.73841 & 10 & \textbf{-34.7471} & 34.7471 & +12.1\% \\
      E10  & 8.27317 & 10.3706 & 5.39672 & 10 & \textbf{-4.93034} & 1.43322 & +10.6\% \\
      \bottomrule
    \end{tabular}}
  \caption{Initial one-sided bounds \eqref{eq:main_pressure_bound} on the star
  pressures and the guaranteed wavespeed estimates \eqref{eq:main_wavespeed}
  they imply, for the configurations E1--E10 of
  Table~\ref{tab:elementary_configurations}. In each pair of wavespeeds the
  entry of larger modulus, printed in bold, determines $\lambda_{\max}$;
  $\Delta\lambda_{\max}$ is its relative deviation from the exact maximal
  wavespeed of Table~\ref{tab:elementary_wavespeeds}.}
  \label{tab:initial_bounds}
\end{table}

\begin{table}[p]
  \centering
  \setlength{\tabcolsep}{5.5pt}
  \captionsetup[subfloat]{captionskip=3pt,farskip=7.5pt}
  \subfloat[$c_0=0$]{%
    \scriptsize
    \begin{tabular}{l*{8}{r}}
      \toprule
      & \multicolumn{4}{c}{pressure bounds \eqref{eq:main_bootstrap}} & \multicolumn{4}{c}{wavespeed bounds \eqref{eq:main_wavespeed}} \\
      \cmidrule(lr){2-5} \cmidrule(lr){6-9}
      & $\underbar P^{(N)}_{\mathrm L}$ & $\bar P^{(N)}_{\mathrm L}$ & $\underbar P^{(N)}_{\mathrm R}$ & $\bar P^{(N)}_{\mathrm R}$ & $\lambda^-_1$ & $\lambda^+_3$ & $\Delta\lambda_{\max}$ & $N$ \\[0.25em]
      E1   & 6.544e-4 & 6.544e-4 & 6.544e-4 & 6.544e-4 & -0.305505 & \textbf{0.886139} & +0\% & 2 \\
      E2   & 0.30313 & 0.30313 & 0.30313 & 0.30313 & -1.18322 & \textbf{1.75216} & +0\% & 2 \\
      E3   & 2.4661 & 2.4661 & 2.4661 & 2.4661 & \textbf{-2.63357} & 2.47932 & +0\% & 2 \\
      E4   & 460.894 & 460.894 & 460.894 & 460.894 & -27.4166 & \textbf{33.5175} & +0\% & 2 \\
      E5   & 1.02391 & 1.02391 & 1.02391 & 1.02391 & \textbf{-1.17528} & 1.17528 & +0\% & 2 \\
      E6   & 10.4769 & 10.4769 & 10.4769 & 10.4769 & -1.52708 & \textbf{2.81737} & +0\% & 2 \\
      E7   & 10.4194 & 10.4194 & 10.4194 & 10.4194 & -0.619073 & \textbf{2.19288} & +0\% & 2 \\
      E8   & 1.17562 & 1.17562 & 1.17562 & 1.17562 & -0.284928 & \textbf{0.647425} & +0\% & 2 \\
      E9   & 9.86846 & 9.86846 & 9.86846 & 9.86846 & \textbf{-26.7075} & 26.7075 & +0\% & 2 \\
      E10  & 8.66235 & 8.66235 & 8.66235 & 8.66235 & \textbf{-3.99166} & 1.43322 & +0\% & 2 \\
      \bottomrule
    \end{tabular}}
  \\[0.75em]
  \subfloat[$c_0=0.1\,c_0^{\max}$]{%
    \scriptsize
    \begin{tabular}{l*{8}{r}}
      \toprule
      & \multicolumn{4}{c}{pressure bounds \eqref{eq:main_bootstrap}} & \multicolumn{4}{c}{wavespeed bounds \eqref{eq:main_wavespeed}} \\
      \cmidrule(lr){2-5} \cmidrule(lr){6-9}
      & $\underbar P^{(N)}_{\mathrm L}$ & $\bar P^{(N)}_{\mathrm L}$ & $\underbar P^{(N)}_{\mathrm R}$ & $\bar P^{(N)}_{\mathrm R}$ & $\lambda^-_1$ & $\lambda^+_3$ & $\Delta\lambda_{\max}$ & $N$ \\[0.25em]
      E1   & 6.54e-4 & 6.544e-4 & 6.54e-4 & 6.544e-4 & -0.305505 & \textbf{0.886139} & +0\% & 3 \\
      E2   & 0.303125 & 0.303236 & 0.303128 & 0.303239 & -1.18533 & \textbf{1.75248} & +0.02\% & 3 \\
      E3   & 2.46595 & 2.4662 & 2.46584 & 2.4661 & \textbf{-2.63357} & 2.47982 & +0\% & 3 \\
      E4   & 456.503 & 460.894 & 456.503 & 460.894 & -27.4166 & \textbf{33.5175} & $<0.01\%$ & 3 \\
      E5   & 1.02386 & 1.02391 & 1.02386 & 1.02391 & \textbf{-1.17926} & 1.17926 & +0.18\% & 4 \\
      E6   & 10.4769 & 10.4769 & 10.4769 & 10.4769 & -1.52708 & \textbf{2.8174} & $<0.01\%$ & 3 \\
      E7   & 10.4194 & 10.4194 & 10.4194 & 10.4194 & -0.619073 & \textbf{2.19288} & $<0.01\%$ & 2 \\
      E8   & 1.17562 & 1.17562 & 1.17562 & 1.17562 & -0.284928 & \textbf{0.647425} & $<0.01\%$ & 3 \\
      E9   & 9.86846 & 9.8687 & 9.86846 & 9.8687 & \textbf{-26.7547} & 26.7547 & +0\% & 3 \\
      E10  & 8.66235 & 8.66978 & 8.65593 & 8.66336 & \textbf{-3.99671} & 1.43322 & +0\% & 3 \\
      \bottomrule
    \end{tabular}}
  \\[0.75em]
  \subfloat[$c_0=0.5\,c_0^{\max}$]{%
    \scriptsize
    \begin{tabular}{l*{8}{r}}
      \toprule
      & \multicolumn{4}{c}{pressure bounds \eqref{eq:main_bootstrap}} & \multicolumn{4}{c}{wavespeed bounds \eqref{eq:main_wavespeed}} \\
      \cmidrule(lr){2-5} \cmidrule(lr){6-9}
      & $\underbar P^{(N)}_{\mathrm L}$ & $\bar P^{(N)}_{\mathrm L}$ & $\underbar P^{(N)}_{\mathrm R}$ & $\bar P^{(N)}_{\mathrm R}$ & $\lambda^-_1$ & $\lambda^+_3$ & $\Delta\lambda_{\max}$ & $N$ \\[0.25em]
      E1   & 6.453e-4 & 6.544e-4 & 6.453e-4 & 6.544e-4 & -0.305505 & \textbf{0.886139} & $<0.01\%$ & 4 \\
      E2   & 0.303009 & 0.305788 & 0.303083 & 0.305861 & -1.23491 & \textbf{1.76017} & +0.28\% & 4 \\
      E3   & 2.46229 & 2.46868 & 2.45971 & 2.4661 & \textbf{-2.63358} & 2.49175 & +0\% & 3 \\
      E4   & 359.899 & 460.894 & 359.899 & 460.894 & -27.4166 & \textbf{33.5176} & $<0.01\%$ & 4 \\
      E5   & 1.02277 & 1.02401 & 1.02277 & 1.02401 & \textbf{-1.27131} & 1.27131 & +3.9\% & 7 \\
      E6   & 10.4766 & 10.4769 & 10.4769 & 10.4771 & -1.52708 & \textbf{2.81831} & +0.02\% & 3 \\
      E7   & 10.4194 & 10.4194 & 10.4194 & 10.4194 & -0.619073 & \textbf{2.19288} & $<0.01\%$ & 2 \\
      E8   & 1.17562 & 1.17562 & 1.17562 & 1.17562 & -0.284928 & \textbf{0.647426} & $<0.01\%$ & 3 \\
      E9   & 9.86781 & 9.87502 & 9.86781 & 9.87502 & \textbf{-27.8634} & 27.8634 & +0\% & 7 \\
      E10  & 8.66235 & 8.84795 & 8.50185 & 8.68745 & \textbf{-4.11591} & 1.43322 & +0\% & 3 \\
      \bottomrule
    \end{tabular}}
  \\[0.75em]
  \subfloat[$c_0=0.99\,c_0^{\max}$]{%
    \scriptsize
    \begin{tabular}{l*{8}{r}}
      \toprule
      & \multicolumn{4}{c}{pressure bounds \eqref{eq:main_bootstrap}} & \multicolumn{4}{c}{wavespeed bounds \eqref{eq:main_wavespeed}} \\
      \cmidrule(lr){2-5} \cmidrule(lr){6-9}
      & $\underbar P^{(N)}_{\mathrm L}$ & $\bar P^{(N)}_{\mathrm L}$ & $\underbar P^{(N)}_{\mathrm R}$ & $\bar P^{(N)}_{\mathrm R}$ & $\lambda^-_1$ & $\lambda^+_3$ & $\Delta\lambda_{\max}$ & $N$ \\[0.25em]
      E1   & 6.193e-4 & 6.544e-4 & 6.193e-4 & 6.544e-4 & -0.305505 & \textbf{0.886139} & $<0.01\%$ & 5 \\
      E2   & 0.302624 & 0.313667 & 0.302931 & 0.313974 & -1.37479 & \textbf{1.78385} & +1.13\% & 5 \\
      E3   & 2.45097 & 2.47633 & 2.44074 & 2.46611 & \textbf{-2.63361} & 2.52777 & +0\% & 3 \\
      E4   & 156.49 & 460.895 & 156.49 & 460.894 & -27.4166 & \textbf{33.5176} & $<0.01\%$ & 5 \\
      E5   & 1.01481 & 1.02823 & 1.01481 & 1.02823 & \textbf{-1.51933} & 1.51933 & +12.1\% & 27 \\
      E6   & 10.4758 & 10.4769 & 10.4767 & 10.4778 & -1.52709 & \textbf{2.82106} & +0.07\% & 3 \\
      E7   & 10.4194 & 10.4194 & 10.4194 & 10.4194 & -0.619073 & \textbf{2.19288} & $<0.01\%$ & 3 \\
      E8   & 1.17562 & 1.17562 & 1.17562 & 1.17562 & -0.284929 & \textbf{0.647431} & $<0.01\%$ & 3 \\
      E9   & 9.83399 & 9.93217 & 9.83399 & 9.93217 & \textbf{-30.9913} & 30.9913 & +0\% & 32 \\
      E10  & 8.66235 & 9.39033 & 8.03312 & 8.7611 & \textbf{-4.45782} & 1.43322 & +0\% & 3 \\
      \bottomrule
    \end{tabular}}
  \caption{Bounds on the star pressures after $N$ steps of the iteration
  \eqref{eq:main_bootstrap}, run to a tolerance of $10^{-6}$, and the
  wavespeed estimates \eqref{eq:main_wavespeed} they imply, for the
  configurations E1--E10 of Table~\ref{tab:elementary_configurations}.
  Notation as in Table~\ref{tab:initial_bounds}.}
  \label{tab:bootstrapped_bounds}
\end{table}

\begin{table}[p]
  \centering
  \setlength{\tabcolsep}{5.5pt}
  \captionsetup[subfloat]{captionskip=3pt,farskip=7.5pt}
  \subfloat[both outer waves elementary, at least one of them exotic]{%
    \scriptsize
    \begin{tabular}{lrrrrrcl}
      \toprule
      state & $\rho$ & $u$ & $p$ & $\ph$ & $\mathcal G$ & wave & $\lambda$ \\
      \midrule
      \multicolumn{8}{l}{\textbf{C1}\quad compressive--expansive;\quad left: elementary $\mathrm S$;\quad right: elementary $\mathrm S^\dagger$} \\[0.2em]
      $W^{(0)}{=}W\L$ & 0.7 & 0.759793 & 0.846354 & 0.8 & 1.687 & $\mathrm S$ & -1.84285 \\
      $W^{(1)}{=}W\starL$ & 0.988605 & 0 & 2.23058 & 1.34876 & 2.429 &  &  \\
      \cmidrule(lr){1-1}
      $W^{(1)}{=}W\starR$ & 1.05 & 0 & 2.23058 & 0.78833 & -4.467 &  &  \\
      $W^{(0)}{=}W\R$ & 1.25 & 0.324946 & 2.92351 & 1 & -0.8919 & $\mathrm S^\dagger$ & 2.03091 \\
      \midrule
      \multicolumn{8}{l}{\textbf{C2}\quad compressive--compressive;\quad left: elementary $\mathrm S$;\quad right: elementary $\mathrm R^\dagger$} \\[0.2em]
      $W^{(0)}{=}W\L$ & 1.5 & 0.139109 & 2.97562 & 1 & 0.8615 & $\mathrm S$ & -0.93612 \\
      $W^{(1)}{=}W\starL$ & 1.7229 & 0 & 3.19998 & 1.21437 & 1.106 &  &  \\
      \cmidrule(lr){1-1}
      $W^{(1)}{=}W\starR$ & 1.25 & 0 & 3.19998 & 1.27647 & -0.5583 &  &  \\
      $W^{(0)}{=}W\R$ & 1.05 & -0.334366 & 2.44225 & 1 & -4.279 & $\mathrm R^\dagger$ & 2.57918$\,\to\,$1.39352 \\
      \midrule
      \multicolumn{8}{l}{\textbf{C3}\quad expansive--expansive;\quad left: elementary $\mathrm R$;\quad right: elementary $\mathrm S^\dagger$} \\[0.2em]
      $W^{(0)}{=}W\L$ & 0.9 & -0.27341 & 2.78793 & 2.5 & 4.319 & $\mathrm R$ & -2.99819$\,\to\,$-2.15222 \\
      $W^{(1)}{=}W\starL$ & 0.802169 & 0 & 2.23058 & 2.12801 & 2.121 &  &  \\
      \cmidrule(lr){1-1}
      $W^{(1)}{=}W\starR$ & 1.05 & 0 & 2.23058 & 0.78833 & -4.467 &  &  \\
      $W^{(0)}{=}W\R$ & 1.25 & 0.324946 & 2.92351 & 1 & -0.8919 & $\mathrm S^\dagger$ & 2.03091 \\
      \midrule
      \multicolumn{8}{l}{\textbf{C4}\quad expansive--compressive;\quad left: elementary $\mathrm R$;\quad right: elementary $\mathrm R^\dagger$} \\[0.2em]
      $W^{(0)}{=}W\L$ & 1.8 & -0.289923 & 3.73732 & 1.75 & 1.151 & $\mathrm R$ & -1.46474$\,\to\,$-1.17124 \\
      $W^{(1)}{=}W\starL$ & 1.40281 & 0 & 3.19998 & 1.2344 & 0.6779 &  &  \\
      \cmidrule(lr){1-1}
      $W^{(1)}{=}W\starR$ & 1.25 & 0 & 3.19998 & 1.27647 & -0.5583 &  &  \\
      $W^{(0)}{=}W\R$ & 1.05 & -0.334366 & 2.44225 & 1 & -4.279 & $\mathrm R^\dagger$ & 2.57918$\,\to\,$1.39352 \\
      \midrule
      \multicolumn{8}{l}{\textbf{C5}\quad expansive--expansive;\quad left: elementary $\mathrm S^\dagger$;\quad right: elementary $\mathrm S^\dagger$} \\[0.2em]
      $W^{(0)}{=}W\L$ & 1.2 & -0.268028 & 2.88946 & 1 & -2.011 & $\mathrm S^\dagger$ & -2.15697 \\
      $W^{(1)}{=}W\starL$ & 1.05089 & 0 & 2.28192 & 0.833355 & -4.481 &  &  \\
      \cmidrule(lr){1-1}
      $W^{(1)}{=}W\starR$ & 1.05 & 0 & 2.28192 & 0.839667 & -4.421 &  &  \\
      $W^{(0)}{=}W\R$ & 1.28 & 0.359416 & 3.03678 & 1.1 & -0.3092 & $\mathrm S^\dagger$ & 2.00023 \\
      \midrule
      \multicolumn{8}{l}{\textbf{C6}\quad compressive--compressive;\quad left: elementary $\mathrm R^\dagger$;\quad right: elementary $\mathrm R^\dagger$} \\[0.2em]
      $W^{(0)}{=}W\L$ & 1.03 & 0.399106 & 2.18239 & 0.9 & -2.708 & $\mathrm R^\dagger$ & -2.76486$\,\to\,$-1.33639 \\
      $W^{(1)}{=}W\starL$ & 1.26063 & 0 & 3.12292 & 1.19425 & -0.4792 &  &  \\
      \cmidrule(lr){1-1}
      $W^{(1)}{=}W\starR$ & 1.22 & 0 & 3.12292 & 1.21752 & -1.187 &  &  \\
      $W^{(0)}{=}W\R$ & 1.06 & -0.272185 & 2.50953 & 1 & -4.765 & $\mathrm R^\dagger$ & 2.48931$\,\to\,$1.45111 \\
      \bottomrule
    \end{tabular}}
  \caption{The twelve configurations C1--C12 at the fixed bump strength
  $c_0=0.1$, wave by wave: for each of them the left wave
  $W\L=W^{(0)},\ldots,W^{(n)}=W\starL$ is listed first, the right wave
  $W\starR=W^{(n)},\ldots,W^{(0)}=W\R$ below it. Here $p$ is the total and
  $\ph$ the hydrodynamical pressure, $\mathcal G$ the fundamental
  derivative \eqref{A3}, and $\lambda$ the speed of the elementary wave
  leaving the state, quoted from head to tail for a fan. A dagger marks an
  exotic wave, that is, an expansive shock or a compressive fan.}
  \label{tab:composite_configurations}
\end{table}

\begin{table}[p]
  \ContinuedFloat
  \centering
  \setlength{\tabcolsep}{5.5pt}
  \captionsetup[subfloat]{captionskip=3pt,farskip=7.5pt}
  \subfloat[at least one outer wave composite]{%
    \scriptsize
    \begin{tabular}{lrrrrrcl}
      \toprule
      state & $\rho$ & $u$ & $p$ & $\ph$ & $\mathcal G$ & wave & $\lambda$ \\
      \midrule
      \multicolumn{8}{l}{\textbf{C7}\quad compressive--compressive;\quad left: composite $\mathrm S{+}\mathrm R^\dagger$;\quad right: elementary $\mathrm S$} \\[0.2em]
      $W^{(0)}{=}W\L$ & 0.95 & 0.60589 & 1.54782 & 1 & 5.761 & $\mathrm S$ & -2.80201 \\
      $W^{(1)}$ & 1.03749 & 0.318515 & 2.4782 & 1.13215 & -3.281 & $\mathrm R^\dagger$ & -2.80201$\,\to\,$-1.58547 \\
      $W^{(2)}{=}W\starL$ & 1.2 & 0 & 3.27743 & 1.38797 & -1.433 &  &  \\
      \cmidrule(lr){1-1}
      $W^{(1)}{=}W\starR$ & 2.46794 & 0 & 3.27743 & 1.28471 & 1.188 &  &  \\
      $W^{(0)}{=}W\R$ & 1.45 & -0.448205 & 2.57122 & 0.6 & 0.48 & $\mathrm S$ & 0.638444 \\
      \midrule
      \multicolumn{8}{l}{\textbf{C8}\quad compressive--compressive;\quad left: composite $\mathrm R^\dagger{+}\mathrm S$;\quad right: elementary $\mathrm S$} \\[0.2em]
      $W^{(0)}{=}W\L$ & 1.05 & 0.650388 & 2.44225 & 1 & -4.279 & $\mathrm R^\dagger$ & -2.26316$\,\to\,$-1.14171 \\
      $W^{(1)}$ & 1.2 & 0.375166 & 3.09502 & 1.20556 & -1.677 & $\mathrm S$ & -1.14171 \\
      $W^{(2)}{=}W\starL$ & 1.59432 & 0 & 3.77792 & 1.79674 & 1.084 &  &  \\
      \cmidrule(lr){1-1}
      $W^{(1)}{=}W\starR$ & 2.62274 & 0 & 3.77792 & 1.78478 & 1.194 &  &  \\
      $W^{(0)}{=}W\R$ & 1.5 & -0.534831 & 2.77562 & 0.8 & 0.7846 & $\mathrm S$ & 0.71454 \\
      \midrule
      \multicolumn{8}{l}{\textbf{C9}\quad compressive--compressive;\quad left: composite $\mathrm S{+}\mathrm R^\dagger$;\quad right: composite $\mathrm S{+}\mathrm R^\dagger$} \\[0.2em]
      $W^{(0)}{=}W\L$ & 0.95 & 0.534021 & 1.54782 & 1 & 5.761 & $\mathrm S$ & -2.87388 \\
      $W^{(1)}$ & 1.03749 & 0.246646 & 2.4782 & 1.13215 & -3.281 & $\mathrm R^\dagger$ & -2.87388$\,\to\,$-1.81624 \\
      $W^{(2)}{=}W\starL$ & 1.15 & 0 & 3.13477 & 1.30769 & -3.023 &  &  \\
      \cmidrule(lr){1-1}
      $W^{(2)}{=}W\starR$ & 1.15 & 0 & 3.13477 & 1.30769 & -3.023 &  &  \\
      $W^{(1)}$ & 1.03749 & -0.246646 & 2.4782 & 1.13215 & -3.281 & $\mathrm R^\dagger$ & 2.87388$\,\to\,$1.81624 \\
      $W^{(0)}{=}W\R$ & 0.95 & -0.534021 & 1.54782 & 1 & 5.761 & $\mathrm S$ & 2.87388 \\
      \midrule
      \multicolumn{8}{l}{\textbf{C10}\quad compressive--compressive;\quad left: composite $\mathrm S{+}\mathrm R^\dagger{+}\mathrm S$;\quad right: elementary $\mathrm S$} \\[0.2em]
      $W^{(0)}{=}W\L$ & 0.95 & 0.944406 & 1.54782 & 1 & 5.761 & $\mathrm S$ & -2.46349 \\
      $W^{(1)}$ & 1.03749 & 0.657031 & 2.4782 & 1.13215 & -3.281 & $\mathrm R^\dagger$ & -2.46349$\,\to\,$-1.24696 \\
      $W^{(2)}$ & 1.2 & 0.338516 & 3.27743 & 1.38797 & -1.433 & $\mathrm S$ & -1.24696 \\
      $W^{(3)}{=}W\starL$ & 1.52577 & 0 & 3.92148 & 1.94406 & 1.045 &  &  \\
      \cmidrule(lr){1-1}
      $W^{(1)}{=}W\starR$ & 2.6382 & 0 & 3.92148 & 1.9283 & 1.194 &  &  \\
      $W^{(0)}{=}W\R$ & 1.55 & -0.52673 & 2.87891 & 0.9 & 0.9176 & $\mathrm S$ & 0.750262 \\
      \midrule
      \multicolumn{8}{l}{\textbf{C11}\quad compressive--expansive;\quad left: elementary $\mathrm S$;\quad right: composite $\mathrm S^\dagger{+}\mathrm R$} \\[0.2em]
      $W^{(0)}{=}W\L$ & 0.5 & 0.517362 & 0.314457 & 0.3 & 1.302 & $\mathrm S$ & -0.885178 \\
      $W^{(1)}{=}W\starL$ & 0.792236 & 0 & 0.677267 & 0.58329 & 3.163 &  &  \\
      \cmidrule(lr){1-1}
      $W^{(2)}{=}W\starR$ & 0.8 & 0 & 0.677267 & 0.576657 & 3.361 &  &  \\
      $W^{(1)}$ & 0.948929 & 0.329757 & 1.27256 & 0.73235 & 6.027 & $\mathrm R$ & 3.18352$\,\to\,$1.3797 \\
      $W^{(0)}{=}W\R$ & 1.2 & 0.926837 & 2.88946 & 1 & -2.011 & $\mathrm S^\dagger$ & 3.18352 \\
      \midrule
      \multicolumn{8}{l}{\textbf{C12}\quad compressive--expansive;\quad left: composite $\mathrm S{+}\mathrm R^\dagger$;\quad right: composite $\mathrm R{+}\mathrm S^\dagger$} \\[0.2em]
      $W^{(0)}{=}W\L$ & 0.95 & 0.522434 & 1.24782 & 0.7 & 6.007 & $\mathrm S$ & -2.81255 \\
      $W^{(1)}$ & 1.03675 & 0.243373 & 2.13195 & 0.791955 & -3.436 & $\mathrm R^\dagger$ & -2.81255$\,\to\,$-1.65563 \\
      $W^{(2)}{=}W\starL$ & 1.15356 & 0 & 2.75265 & 0.919628 & -3.609 &  &  \\
      \cmidrule(lr){1-1}
      $W^{(2)}{=}W\starR$ & 1.20175 & 0 & 2.75265 & 0.861629 & -2.235 &  &  \\
      $W^{(1)}$ & 1.4 & 0.181419 & 3.03197 & 1.06679 & 0.5935 & $\mathrm S^\dagger$ & 1.28117 \\
      $W^{(0)}{=}W\R$ & 1.7 & 0.39206 & 3.38499 & 1.4 & 1.11 & $\mathrm R$ & 1.4789$\,\to\,$1.28117 \\
      \bottomrule
    \end{tabular}}
  \caption{\emph{cont'd.}}
\end{table}

\begin{table}[p]
  \centering
  \setlength{\tabcolsep}{5.5pt}
  \scriptsize
  \begin{tabular}{l*{8}{r}}
    \toprule
    & \multicolumn{4}{c}{exact \eqref{eq:riemann_problem}} & \multicolumn{4}{c}{hydrodynamical \eqref{eq:riemann_problem_h}} \\
    \cmidrule(lr){2-5} \cmidrule(lr){6-9}
    & $\ph\starL$ & $\ph\starR$ & $\lambda^-_1$ & $\lambda^+_3$ & $\tilde p_h\star$ & $\lambda^-_{1,\mathrm h}$ & $\lambda^+_{3,\mathrm h}$ & $\Delta\lambda_{\max}$ \\[0.25em]
      C1   & 1.34876 & 0.78833 & -1.84285 & \textbf{2.03091} & 1.14403 & -0.719994 & \textbf{1.44667} & \textcolor{red}{\textbf{-28.8\%}} \\
      C2   & 1.21437 & 1.27647 & -0.93612 & \textbf{2.57918} & 1.35721 & -0.965022 & \textbf{0.985323} & \textcolor{red}{\textbf{-61.8\%}} \\
      C3   & 2.12801 & 0.78833 & \textbf{-2.99819} & 2.03091 & 1.35029 & \textbf{-2.24544} & 1.53171 & \textcolor{red}{\textbf{-25.1\%}} \\
      C4   & 1.2344 & 1.27647 & -1.46474 & \textbf{2.57918} & 1.35288 & \textbf{-1.45659} & 0.983446 & \textcolor{red}{\textbf{-43.5\%}} \\
      C5   & 0.833355 & 0.839667 & \textbf{-2.15697} & 2.00023 & 0.692298 & -1.34815 & \textbf{1.45629} & \textcolor{red}{\textbf{-32.5\%}} \\
      C6   & 1.19425 & 1.21752 & \textbf{-2.76486} & 2.48931 & 1.42195 & -0.954187 & \textbf{1.06887} & \textcolor{red}{\textbf{-61.3\%}} \\
      C7   & 1.38797 & 1.28471 & \textbf{-2.80201} & 0.638444 & 1.65072 & \textbf{-0.90925} & 0.755487 & \textcolor{red}{\textbf{-67.6\%}} \\
      C8   & 1.79674 & 1.78478 & \textbf{-2.26316} & 0.71454 & 1.96139 & \textbf{-0.909119} & 0.759687 & \textcolor{red}{\textbf{-59.8\%}} \\
      C9   & 1.30769 & 1.30769 & \textbf{-2.87388} & 2.87388 & 1.79951 & \textbf{-1.04192} & 1.04192 & \textcolor{red}{\textbf{-63.7\%}} \\
      C10  & 1.94406 & 1.9283 & \textbf{-2.46349} & 0.750262 & 2.35223 & -0.839343 & \textbf{0.865107} & \textcolor{red}{\textbf{-64.9\%}} \\
      C11  & 0.58329 & 0.576657 & -0.885178 & \textbf{3.18352} & 0.418176 & -0.542648 & \textbf{2.00696} & \textcolor{red}{\textbf{-37.0\%}} \\
      C12  & 0.919628 & 0.861629 & \textbf{-2.81255} & 1.4789 & 1.0803 & -0.707183 & \textbf{1.46581} & \textcolor{red}{\textbf{-47.9\%}} \\
    \bottomrule
  \end{tabular}
  \caption{Star pressures and extremal wavespeeds of the exact solution of
  \eqref{eq:riemann_problem} and of the reduced hydrodynamical problem
  \eqref{eq:riemann_problem_h} for the configurations C1--C12 of
    Table~\ref{tab:composite_configurations}. Notation as in
    Table~\ref{tab:elementary_wavespeeds}.}
  \label{tab:composite_wavespeeds}
\end{table}

\begin{table}[p]
  \centering
  \setlength{\tabcolsep}{5.5pt}
  \scriptsize
  \begin{tabular}{l*{9}{r}}
    \toprule
    & \multicolumn{4}{c}{pressure bounds \eqref{eq:main_pressure_bound}} & \multicolumn{3}{c}{wavespeed bounds \eqref{eq:main_wavespeed}} & & \\
    \cmidrule(lr){2-5} \cmidrule(lr){6-8}
    & $\underbar P^{(N)}_{\mathrm L}$ & $\bar P^{(N)}_{\mathrm L}$ & $\underbar P^{(N)}_{\mathrm R}$ & $\bar P^{(N)}_{\mathrm R}$ & $\lambda^-_1$ & $\lambda^+_3$ & $\Delta\lambda_{\max}$ & $c_0^{\max}$ & $N$ \\[0.25em]
      C1   & 0.8 & 2.87716 & -1.07716 & 1.73429 & \textbf{-6.04087} & 4.16033 & +197\% & 0.0169 & 0 \\
      C2   & 0.466633 & 1.91694 & 1 & 1.73241 & -3.11895 & \textbf{3.72534} & +44.4\% & 0.0203 & 0 \\
      C3   & 1.6139 & 2.63558 & 0.432666 & 1 & \textbf{-4.36137} & 4.0673 & +45.5\% & 0.00872 & 0 \\
      C4   & 0.454936 & 1.84549 & 1 & 2.29506 & -3.05375 & \textbf{3.80378} & +47.5\% & 0.0203 & 0 \\
      C5   & 0.422211 & 1.14731 & 0.470981 & 1.1 & \textbf{-4.17896} & 4.03249 & +93.7\% & 0.00812 & 0 \\
      C6   & 0.9 & 1.99236 & 0.672852 & 2.15313 & -4.08093 & \textbf{4.26044} & +54.1\% & 0.0178 & 0 \\
      C7   & 1 & 3.00407 & -0.4234 & 3.07922 & \textbf{-4.52377} & 3.04563 & +61.4\% & 0.0101 & 0 \\
      C8   & 1 & 3.57264 & 0.466633 & 4.05383 & \textbf{-3.66057} & 2.90976 & +61.7\% & 0.0158 & 0 \\
      C9   & 1 & 3.04302 & 1 & 3.04302 & \textbf{-4.60043} & 4.60043 & +60.1\% & 0.0183 & 0 \\
      C10  & 1 & 4.3313 & -0.43108 & 5.78988 & \textbf{-4.34614} & 3.07587 & +76.4\% & 0.017 & 0 \\
      C11  & 0.156031 & 2.87501 & -1.57501 & 1 & \textbf{-8.8396} & 4.81889 & +178\% & 0.00239 & 0 \\
      C12  & 0.7 & 2.83716 & -0.73716 & 1.65595 & \textbf{-4.58045} & 3.27305 & +62.9\% & 0.0153 & 0 \\
    \bottomrule
  \end{tabular}
  \caption{Bounds \eqref{eq:main_pressure_bound} on the star pressures and the
    wavespeed estimates \eqref{eq:main_wavespeed} they imply, for the
    configurations C1--C12 of Table~\ref{tab:composite_configurations}. The
    bump strength $c_0=0.1$ exceeds the cut-off $c_0^{\max}$ of every
    configuration, so the iteration \eqref{eq:main_bootstrap} never starts,
    $N=0$, and the bounds are those of Theorem~\ref{thm:main}. Notation as
    in Table~\ref{tab:bootstrapped_bounds}.}
    \label{tab:composite_bounds}
\end{table}

\begin{figure}[p]
  \centering
  \resizebox{\textwidth}{!}{\input{C10.tex}}
  \caption{Configuration C10, left wave composite $\mathrm S{+}\mathrm
    R^\dagger{+}\mathrm S$, right wave an elementary shock. Top: density at
    $t=0.24$. Bottom: the split $a^2=\ah^2+\ab^2$ over density; the states
    of the composite fans (Table~\ref{tab:composite_configurations}) are
    marked in the diagram.}
  \label{fig:composite_C10}
\end{figure}
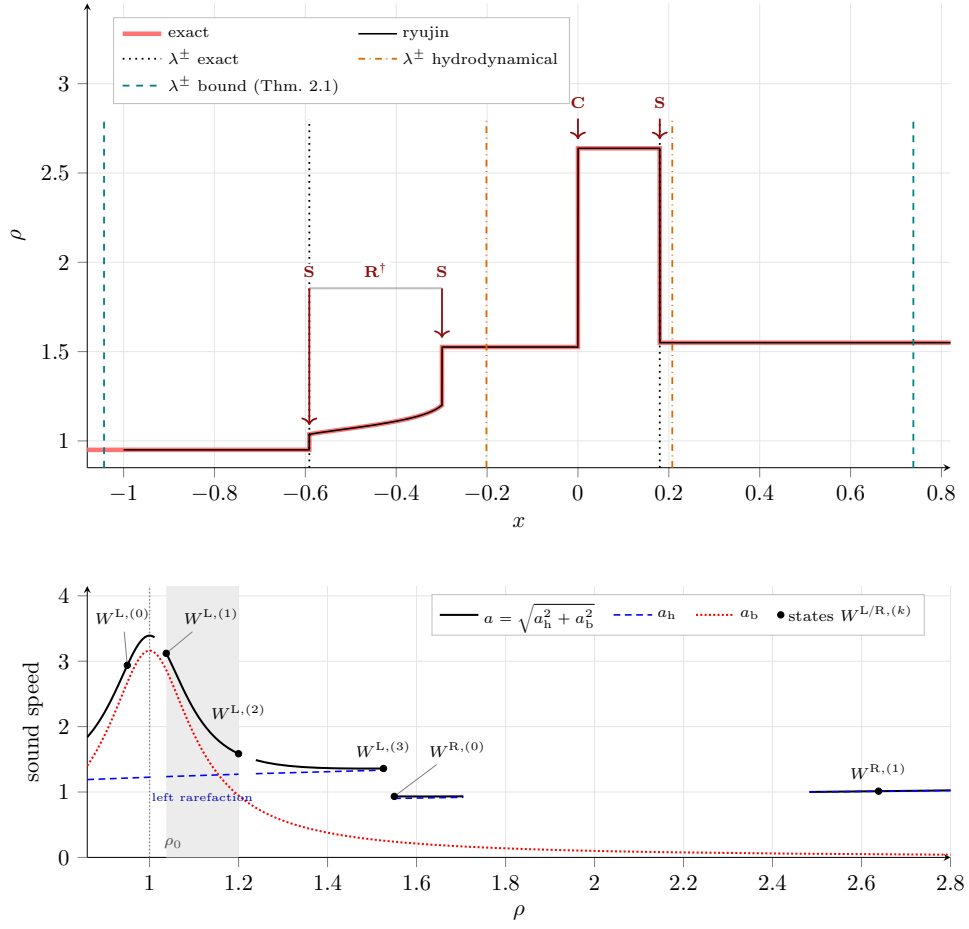

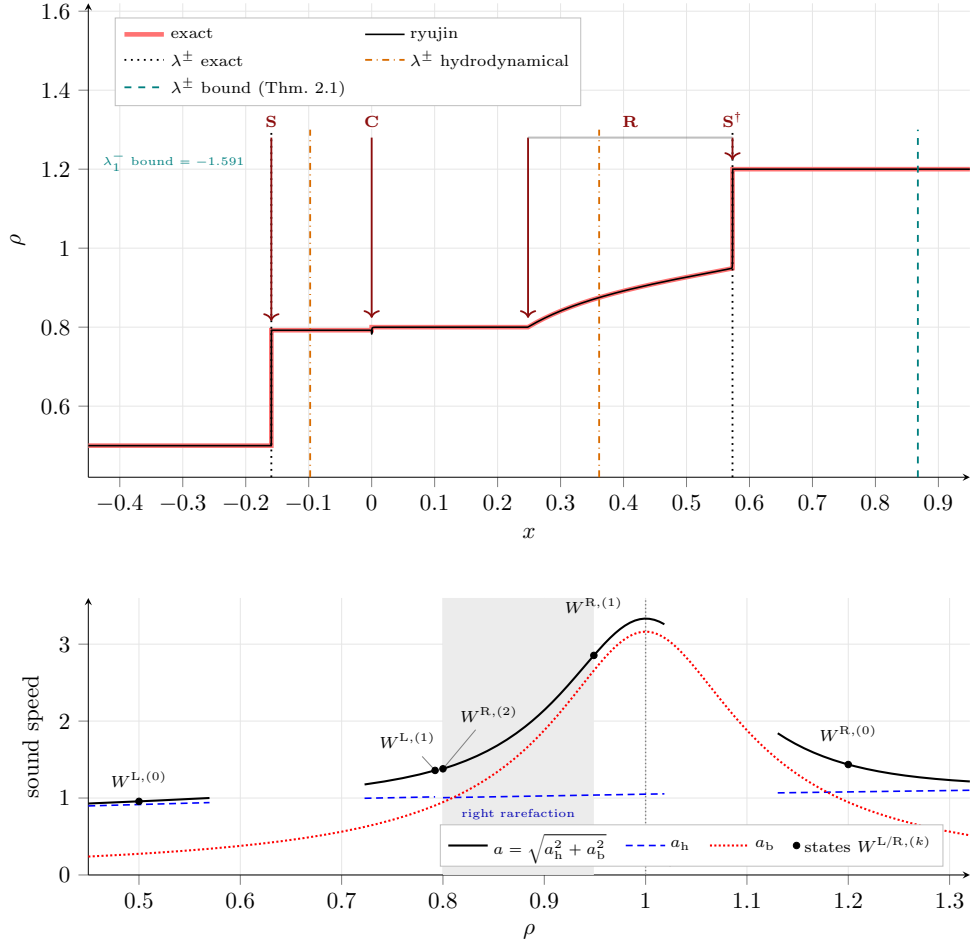
\begin{figure}[p]
  \centering
  \resizebox{\textwidth}{!}{\input{C11.tex}}
  \caption{Configuration C11, left wave an elementary shock, right wave the
    composite $\mathrm S^\dagger{+}\mathrm R$ led by an expansive shock.
    Top: density at $t=0.18$; the guaranteed lower bound
    $\lambda^-_1=-8.8396$ falls outside the plotted range. Bottom: as in
    Figure~\ref{fig:composite_C10}.}
  \label{fig:composite_C11}
\end{figure}

\begin{figure}[p]
  \centering
  \resizebox{\textwidth}{!}{\input{C12.tex}}
  \caption{Configuration C12, both outer waves composite and of different
    type: $\mathrm S{+}\mathrm R^\dagger$ on the left and $\mathrm
    R{+}\mathrm S^\dagger$ on the right. Top: density at $t=0.20$. Bottom:
    as in Figure~\ref{fig:composite_C10}; the left wave sits on the
    barotropic bump, the right one far out on its tail.}
  \label{fig:composite_C12}
\end{figure}
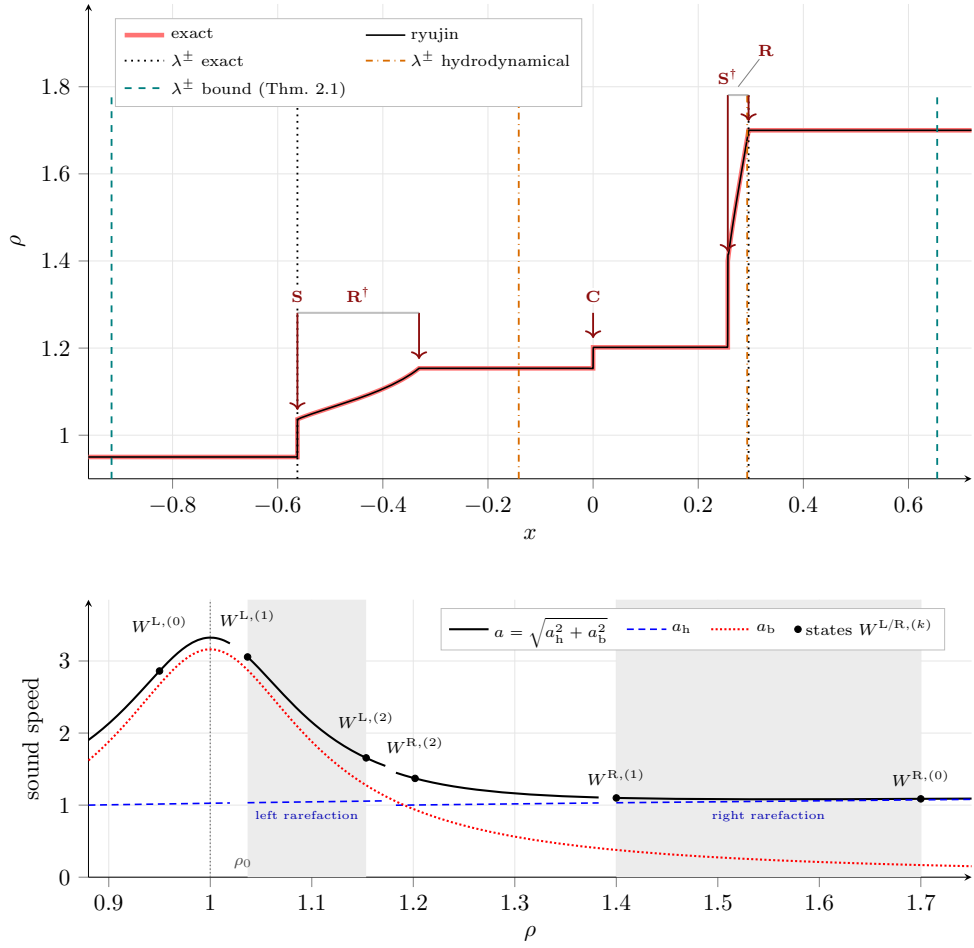

\begin{figure}[p]
  \centering
  \resizebox{\textwidth}{!}{\input{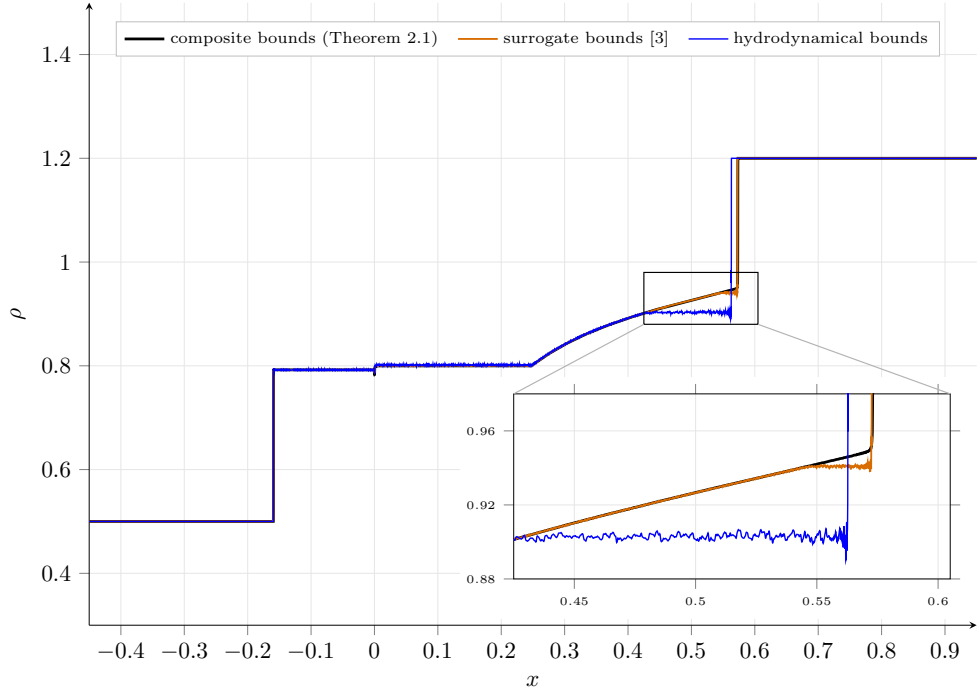}}
  \caption{Configuration C11 computed three times with the same
    discretization and limiter, changing only the estimate of
    $\lambda_{\max}$ that enters the graph viscosity: the bounds of
    Theorem~\ref{thm:main}; the surrogate estimate of
    \cite{ClaytonGuermondPopov2022}; and the bounds of
    Theorem~\ref{thm:main} applied to the hydrodynamical constituent alone,
    that is, with $c_0=0$. Density at $t=0.18$; the inset blows up the tail
    of the fan running into the right shock.}
  \label{fig:composite_C11_bounds}
\end{figure}


\clearpage
\newpage
\bibliographystyle{abbrvnat}

\end{document}

%% file: C10.tex
\definecolor{wavered}{rgb}{0.55,0.05,0.05}
\begin{tikzpicture}

  \begin{axis}[
      name=density,
      width=150mm, height=88mm,
      xmin=-1.08, xmax=0.82, ymin=0.85, ymax=3.45,
      xlabel={$x$}, ylabel={$\rho$},
      axis lines=left, tick align=outside,
      grid=major, grid style={gray!20},
      legend cell align=left, legend columns=2, legend pos=north west,
      legend style={draw=gray!50, fill=white, font=\scriptsize,
                    /tikz/every even column/.append style={column sep=6pt}},
    ]

    \addplot[red!55, line width=2.0pt, no marks] coordinates {
        (-1.08000, 0.950000)
        (-0.59124, 0.950000)
        (-0.59124, 1.037490)
        (-0.57808, 1.042785)
        (-0.56277, 1.048487)
        (-0.48544, 1.075368)
        (-0.46633, 1.082292)
        (-0.44917, 1.088809)
        (-0.43101, 1.096140)
        (-0.41414, 1.103472)
        (-0.39863, 1.110803)
        (-0.38447, 1.118134)
        (-0.36962, 1.126687)
        (-0.35646, 1.135240)
        (-0.35021, 1.139721)
        (-0.34437, 1.144201)
        (-0.33893, 1.148681)
        (-0.33344, 1.153569)
        (-0.32836, 1.158456)
        (-0.32369, 1.163344)
        (-0.31940, 1.168231)
        (-0.31547, 1.173119)
        (-0.31158, 1.178413)
        (-0.30806, 1.183708)
        (-0.30489, 1.189003)
        (-0.30182, 1.194705)
        (-0.29927, 1.200000)
        (-0.29927, 1.525770)
        (0.00000, 1.525770)
        (0.00000, 2.638200)
        (0.18006, 2.638200)
        (0.18006, 1.550000)
        (0.82000, 1.550000)
    };
    \addlegendentry{exact}

    \addplot[black, line width=0.7pt, no marks] coordinates {
        (-1.00000, 0.950000)
        (-0.59186, 0.950036)
        (-0.59167, 0.950751)
        (-0.59161, 0.952037)
        (-0.59155, 0.955441)
        (-0.59149, 0.964123)
        (-0.59143, 0.984424)
        (-0.59137, 1.015672)
        (-0.59131, 1.026658)
        (-0.59125, 1.027577)
        (-0.59119, 1.031195)
        (-0.59113, 1.032455)
        (-0.59106, 1.032546)
        (-0.59100, 1.033327)
        (-0.59094, 1.034699)
        (-0.59082, 1.034484)
        (-0.59076, 1.035568)
        (-0.59064, 1.035538)
        (-0.59058, 1.036082)
        (-0.59045, 1.036020)
        (-0.59033, 1.036559)
        (-0.58948, 1.037459)
        (-0.58234, 1.040959)
        (-0.56500, 1.047647)
        (-0.48248, 1.076368)
        (-0.44690, 1.089667)
        (-0.41034, 1.105214)
        (-0.39429, 1.112979)
        (-0.38049, 1.120325)
        (-0.36603, 1.128920)
        (-0.35284, 1.137806)
        (-0.34167, 1.146394)
        (-0.33173, 1.155206)
        (-0.32269, 1.164500)
        (-0.31427, 1.174775)
        (-0.31049, 1.180119)
        (-0.30682, 1.185887)
        (-0.30341, 1.192016)
        (-0.30090, 1.197511)
        (-0.29950, 1.198227)
        (-0.29944, 1.198851)
        (-0.29938, 1.203760)
        (-0.29932, 1.221190)
        (-0.29926, 1.271205)
        (-0.29919, 1.364292)
        (-0.29913, 1.473056)
        (-0.29907, 1.500769)
        (-0.29901, 1.503007)
        (-0.29895, 1.516758)
        (-0.29889, 1.521940)
        (-0.29883, 1.522220)
        (-0.29877, 1.524176)
        (-0.29871, 1.525196)
        (-0.29852, 1.525679)
        (-0.01947, 1.525757)
        (-0.00500, 1.525687)
        (-0.00037, 1.524893)
        (-0.00031, 1.525401)
        (-0.00024, 1.529144)
        (-0.00018, 1.556746)
        (-0.00012, 1.634611)
        (-0.00006, 1.757514)
        (0.00000, 1.918867)
        (0.00006, 2.103289)
        (0.00012, 2.292189)
        (0.00018, 2.457056)
        (0.00024, 2.569728)
        (0.00031, 2.623463)
        (0.00037, 2.635815)
        (0.00043, 2.637530)
        (0.00134, 2.638026)
        (0.17645, 2.638047)
        (0.17737, 2.638368)
        (0.17896, 2.638321)
        (0.17969, 2.637762)
        (0.17975, 2.637189)
        (0.17981, 2.634621)
        (0.17987, 2.634482)
        (0.17993, 2.616724)
        (0.17999, 2.602262)
        (0.18005, 2.585532)
        (0.18011, 2.112066)
        (0.18018, 1.686436)
        (0.18024, 1.570212)
        (0.18030, 1.551762)
        (0.18036, 1.550158)
        (0.18042, 1.550017)
        (0.81995, 1.550000)
    };
    \addlegendentry{ryujin}

    \addplot[black, dotted, line width=0.8pt, no marks, forget plot]
      coordinates {(-0.59124,0.85) (-0.59124,2.8)};
    \addplot[black, dotted, line width=0.8pt, no marks]
      coordinates {(0.18006,0.85) (0.18006,2.8)};
    \addlegendentry{$\lambda^\pm$ exact}
    \addplot[orange!85!black, dashdotted, line width=0.8pt, no marks, forget plot]
      coordinates {(-0.20144,0.85) (-0.20144,2.8)};
    \addplot[orange!85!black, dashdotted, line width=0.8pt, no marks]
      coordinates {(0.20763,0.85) (0.20763,2.8)};
    \addlegendentry{$\lambda^\pm$ hydrodynamical}
    \addplot[teal, dashed, line width=0.8pt, no marks, forget plot]
      coordinates {(-1.04307,0.85) (-1.04307,2.8)};
    \addplot[teal, dashed, line width=0.8pt, no marks]
      coordinates {(0.73821,0.85) (0.73821,2.8)};
    \addlegendentry{$\lambda^\pm$ bound (Thm.~2.1)}

    \node[font=\scriptsize\boldmath, wavered, anchor=south,
          fill=white, inner sep=1pt] at (axis cs:-0.59124, 1.9) {$\mathrm S$};
    \node[font=\scriptsize\boldmath, wavered, anchor=south,
          fill=white, inner sep=1pt] at (axis cs:-0.44525, 1.9) {$\mathrm R^\dagger$};
    \node[font=\scriptsize\boldmath, wavered, anchor=south,
          fill=white, inner sep=1pt] at (axis cs:-0.29927, 1.9) {$\mathrm S$};
    \node[font=\scriptsize\boldmath, wavered, anchor=south,
          fill=white, inner sep=1pt] at (axis cs:0.00000, 2.85) {$\mathrm C$};
    \node[font=\scriptsize\boldmath, wavered, anchor=south,
          fill=white, inner sep=1pt] at (axis cs:0.18006, 2.85) {$\mathrm S$};

    \draw[->, wavered, line width=0.8pt, shorten <=3.58751pt, shorten >=3.98612pt]
      (axis cs:-0.59124, 1.9) -- (axis cs:-0.59124, 1.03749);
    \draw[->, wavered, line width=0.8pt, shorten <=3.58751pt, shorten >=3.98612pt]
      (axis cs:-0.29927, 1.9) -- (axis cs:-0.29927, 1.52577);
    \draw[->, wavered, line width=0.8pt, shorten <=3.58751pt, shorten >=3.98612pt]
      (axis cs:0.00000, 2.85) -- (axis cs:0.00000, 2.63820);
    \draw[->, wavered, line width=0.8pt, shorten <=3.58751pt, shorten >=3.98612pt]
      (axis cs:0.18006, 2.85) -- (axis cs:0.18006, 2.63820);

    \draw[gray, opacity=0.5, line width=0.9pt, yshift=-3.58751pt]
      (axis cs:-0.59124, 1.9) -- (axis cs:-0.29927, 1.9);

  \end{axis}

  \begin{axis}[
      name=sound, at={(density.below south west)}, anchor=north west,
      yshift=-8mm,
      width=150mm, height=58mm,
      xmin=0.86, xmax=2.8, ymin=0, ymax=4.15,
      xlabel={$\rho$}, ylabel={sound speed},
      axis lines=left, tick align=outside,
      grid=major, grid style={gray!20},
      legend cell align=left, legend columns=4, legend pos=north east,
      legend style={draw=gray!50, fill=white, font=\scriptsize,
                    /tikz/every even column/.append style={column sep=6pt}},
    ]

    \addplot[draw=none, fill=gray!15, forget plot] coordinates {
        (1.03749,0) (1.20000,0) (1.20000,4.15) (1.03749,4.15)
    } \closedcycle;

    \addplot[gray, densely dotted, line width=0.6pt, no marks, forget plot]
      coordinates {(1.0, 0) (1.0, 4.15)};
    \node[gray!60!black, font=\scriptsize, anchor=south west]
      at (axis cs:1.015, 0.013) {$\rho_0$};

    \addplot[black, line width=0.9pt, no marks] coordinates {
        (0.86000, 1.838426)
        (0.87110, 1.930420)
        (0.87652, 1.979872)
        (0.88182, 2.031191)
        (0.88712, 2.085592)
        (0.89242, 2.143147)
        (0.89759, 2.202414)
        (0.90289, 2.266283)
        (0.91210, 2.384689)
        (0.92181, 2.518895)
        (0.93089, 2.651282)
        (0.94994, 2.936628)
        (0.95852, 3.059334)
        (0.96293, 3.118559)
        (0.96709, 3.170976)
        (0.97100, 3.216543)
        (0.97466, 3.255417)
        (0.97958, 3.301149)
        (0.98425, 3.336732)
        (0.98879, 3.363314)
        (0.99333, 3.381400)
        (0.99787, 3.390599)
        (1.00229, 3.390836)
        (1.00670, 3.382493)
        (1.01124, 3.365163)
    };
    \addplot[black, line width=0.9pt, no marks, forget plot] coordinates {
        (1.03749, 3.120299)
        (1.04264, 3.051725)
        (1.04847, 2.969825)
        (1.07490, 2.583760)
        (1.08127, 2.495125)
        (1.08723, 2.415528)
        (1.09523, 2.314667)
        (1.10309, 2.222577)
        (1.11095, 2.137694)
        (1.11881, 2.059960)
        (1.12776, 1.979891)
        (1.13697, 1.906190)
        (1.14646, 1.838915)
        (1.15622, 1.777962)
        (1.16639, 1.722414)
        (1.17709, 1.671635)
        (1.18821, 1.626236)
        (1.20000, 1.585074)
    };
    \addplot[black, line width=0.9pt, no marks, forget plot] coordinates {
        (1.23880, 1.488304)
        (1.26086, 1.452244)
        (1.28531, 1.423020)
        (1.31288, 1.399802)
        (1.34404, 1.382217)
        (1.37953, 1.369758)
        (1.42052, 1.361915)
        (1.46871, 1.358339)
        (1.52625, 1.358851)
    };
    \addplot[black, line width=0.9pt, no marks, forget plot] coordinates {
        (1.54953, 0.932840)
        (1.61912, 0.931534)
        (1.70520, 0.933921)
    };
    \addplot[black, line width=0.9pt, no marks, forget plot] coordinates {
        (2.48300, 1.000904)
        (2.80000, 1.024521)
    };
    \addlegendentry{$a=\sqrt{a_{\mathrm h}^2+a_{\mathrm b}^2}$}

    \addplot[blue, densely dashed, line width=0.7pt, no marks] coordinates {
        (0.86000, 1.190028)
        (1.01124, 1.229217)
    };
    \addplot[blue, densely dashed, line width=0.7pt, no marks, forget plot] coordinates {
        (1.03749, 1.235533)
        (1.20000, 1.272019)
    };
    \addplot[blue, densely dashed, line width=0.7pt, no marks, forget plot] coordinates {
        (1.23880, 1.280141)
        (1.52625, 1.334696)
    };
    \addplot[blue, densely dashed, line width=0.7pt, no marks, forget plot] coordinates {
        (1.54953, 0.901556)
        (1.70520, 0.918984)
    };
    \addplot[blue, densely dashed, line width=0.7pt, no marks, forget plot] coordinates {
        (2.48300, 0.999381)
        (2.80000, 1.023688)
    };
    \addlegendentry{$a_{\mathrm h}$}

    \addplot[red, densely dotted, line width=0.9pt, no marks] coordinates {
        (0.86000, 1.401300)
        (0.87035, 1.509307)
        (0.88069, 1.628096)
        (0.89104, 1.758245)
        (0.90139, 1.899958)
        (0.91044, 2.033134)
        (0.91949, 2.173982)
        (0.92790, 2.310210)
        (0.94665, 2.620761)
        (0.95441, 2.744731)
        (0.96217, 2.860493)
        (0.96605, 2.913899)
        (0.96929, 2.955551)
        (0.97381, 3.008822)
        (0.97769, 3.049179)
        (0.98222, 3.089317)
        (0.98610, 3.117215)
        (0.98998, 3.138673)
        (0.99386, 3.153366)
        (0.99774, 3.161067)
        (1.00162, 3.161655)
        (1.00485, 3.156703)
        (1.00809, 3.146856)
        (1.01132, 3.132223)
        (1.01520, 3.108567)
        (1.01843, 3.084014)
        (1.02231, 3.049114)
        (1.02555, 3.015833)
        (1.02943, 2.971323)
        (1.03331, 2.922362)
        (1.03719, 2.869537)
        (1.04171, 2.803812)
        (1.04624, 2.734551)
        (1.05465, 2.599483)
        (1.07599, 2.246695)
        (1.08827, 2.052712)
        (1.09862, 1.899863)
        (1.10897, 1.758158)
        (1.11543, 1.675467)
        (1.12255, 1.589698)
        (1.12966, 1.509234)
        (1.13677, 1.433870)
        (1.14389, 1.363360)
        (1.15165, 1.291656)
        (1.15941, 1.225042)
        (1.16717, 1.163153)
        (1.17622, 1.096447)
        (1.18527, 1.035153)
        (1.19433, 0.978774)
        (1.20403, 0.923301)
        (1.21437, 0.869191)
        (1.22472, 0.819763)
        (1.23571, 0.771822)
        (1.24671, 0.728081)
        (1.25835, 0.685837)
        (1.27063, 0.645270)
        (1.28357, 0.606509)
        (1.29650, 0.571311)
        (1.31073, 0.536206)
        (1.32495, 0.504425)
        (1.33983, 0.474311)
        (1.35599, 0.444733)
        (1.39091, 0.390151)
        (1.42971, 0.341235)
        (1.47239, 0.298034)
        (1.52025, 0.259339)
        (1.57263, 0.225636)
        (1.63147, 0.195614)
        (1.69679, 0.169321)
        (1.77051, 0.146014)
        (1.85199, 0.125862)
        (1.94381, 0.108152)
        (2.04663, 0.092757)
        (2.16239, 0.079355)
        (2.44110, 0.057596)
        (2.80000, 0.041313)
    };
    \addlegendentry{$a_{\mathrm b}$}

    \addplot[only marks, mark=*, mark size=1.4pt, black] coordinates {
        (0.95000, 2.937533)
        (1.03749, 3.120490)
        (1.20000, 1.585473)
        (1.52577, 1.359795)
        (2.63820, 1.012691)
        (1.55000, 0.932816)
    };
    \addlegendentry{states $W^{\mathrm{L/R},(k)}$}

    \node[font=\scriptsize, anchor=south, xshift=-1.99306pt]
      (state0) at (axis cs:0.95000, 3.45180) {$W^{\mathrm L,(0)}$};
    \draw[gray, line width=0.4pt] (state0) -- (axis cs:0.95000, 2.93753);
    \node[font=\scriptsize, anchor=south, xshift=19.93060pt]
      (state1) at (axis cs:1.03749, 3.45180) {$W^{\mathrm L,(1)}$};
    \draw[gray, line width=0.4pt] (state1) -- (axis cs:1.03749, 3.12049);
    \node[font=\scriptsize, anchor=south] at (axis cs:1.20000, 1.98872) {$W^{\mathrm L,(2)}$};
    \node[font=\scriptsize, anchor=south] at (axis cs:1.52577, 1.41980) {$W^{\mathrm L,(3)}$};
    \node[font=\scriptsize, anchor=south] at (axis cs:2.63820, 1.07769) {$W^{\mathrm R,(1)}$};
    \node[font=\scriptsize, anchor=south, xshift=27.90284pt]
      (state5) at (axis cs:1.55000, 1.41885) {$W^{\mathrm R,(0)}$};
    \draw[gray, line width=0.4pt] (state5) -- (axis cs:1.55000, 0.93282);

    \node[font=\tiny, blue!70!black, anchor=north, yshift=3.40282pt]
      at (axis cs:1.11875, 0.99169) {left rarefaction};
  \end{axis}

\end{tikzpicture}

%% file: C11.tex
\definecolor{wavered}{rgb}{0.55,0.05,0.05}
\begin{tikzpicture}

  \begin{axis}[
      name=density,
      width=150mm, height=88mm,
      xmin=-0.45, xmax=0.95, ymin=0.42, ymax=1.62,
      xlabel={$x$}, ylabel={$\rho$},
      axis lines=left, tick align=outside,
      grid=major, grid style={gray!20},
      legend cell align=left, legend columns=2, legend pos=north west,
      legend style={draw=gray!50, fill=white, font=\scriptsize,
                    /tikz/every even column/.append style={column sep=6pt}},
    ]

    \addplot[red!55, line width=2.0pt, no marks] coordinates {
        (-0.45000, 0.500000)
        (-0.15933, 0.500000)
        (-0.15933, 0.792236)
        (0.00000, 0.792236)
        (0.00000, 0.800000)
        (0.24835, 0.800000)
        (0.26061, 0.811198)
        (0.27322, 0.821649)
        (0.28711, 0.832100)
        (0.30131, 0.841805)
        (0.31629, 0.851136)
        (0.33203, 0.860094)
        (0.34924, 0.869052)
        (0.36721, 0.877637)
        (0.38585, 0.885849)
        (0.40503, 0.893687)
        (0.42558, 0.901526)
        (0.44857, 0.909737)
        (0.47069, 0.917202)
        (0.49621, 0.925414)
        (0.57303, 0.948929)
        (0.57303, 1.200000)
        (0.95000, 1.200000)
    };
    \addlegendentry{exact}

    \addplot[black, line width=0.7pt, no marks] coordinates {
        (-0.44995, 0.500000)
        (-0.15955, 0.500004)
        (-0.15948, 0.500467)
        (-0.15942, 0.515827)
        (-0.15936, 0.578699)
        (-0.15930, 0.722554)
        (-0.15924, 0.781584)
        (-0.15918, 0.790304)
        (-0.15912, 0.790317)
        (-0.15906, 0.792062)
        (-0.15887, 0.792197)
        (-0.00220, 0.792184)
        (-0.00037, 0.791660)
        (-0.00031, 0.791016)
        (-0.00024, 0.789439)
        (-0.00012, 0.783421)
        (0.00006, 0.782974)
        (0.00018, 0.783076)
        (0.00024, 0.784786)
        (0.00037, 0.791641)
        (0.00043, 0.793763)
        (0.00049, 0.794989)
        (0.00073, 0.797073)
        (0.00092, 0.797913)
        (0.00238, 0.799337)
        (0.00525, 0.799819)
        (0.24725, 0.799934)
        (0.24841, 0.800424)
        (0.26257, 0.813061)
        (0.27557, 0.823597)
        (0.28778, 0.832677)
        (0.30286, 0.842910)
        (0.31885, 0.852730)
        (0.33344, 0.860921)
        (0.35022, 0.869613)
        (0.36774, 0.877952)
        (0.38794, 0.886808)
        (0.40796, 0.894917)
        (0.44586, 0.908852)
        (0.48425, 0.921691)
        (0.56445, 0.946406)
        (0.57117, 0.948862)
        (0.57178, 0.949382)
        (0.57239, 0.950990)
        (0.57251, 0.951740)
        (0.57269, 0.954076)
        (0.57275, 0.955722)
        (0.57281, 0.956755)
        (0.57288, 0.959301)
        (0.57294, 0.968082)
        (0.57300, 0.976243)
        (0.57306, 0.991341)
        (0.57312, 1.026888)
        (0.57318, 1.087115)
        (0.57324, 1.156252)
        (0.57330, 1.190066)
        (0.57336, 1.199389)
        (0.57343, 1.199974)
        (0.94995, 1.200000)
    };
    \addlegendentry{ryujin}

    \addplot[black, dotted, line width=0.8pt, no marks, forget plot]
      coordinates {(-0.15933,0.42) (-0.15933,1.3)};
    \addplot[black, dotted, line width=0.8pt, no marks]
      coordinates {(0.57303,0.42) (0.57303,1.3)};
    \addlegendentry{$\lambda^\pm$ exact}
    \addplot[orange!85!black, dashdotted, line width=0.8pt, no marks, forget plot]
      coordinates {(-0.09768,0.42) (-0.09768,1.3)};
    \addplot[orange!85!black, dashdotted, line width=0.8pt, no marks]
      coordinates {(0.36125,0.42) (0.36125,1.3)};
    \addlegendentry{$\lambda^\pm$ hydrodynamical}
    \addplot[teal, dashed, line width=0.8pt, no marks]
      coordinates {(0.86740,0.42) (0.86740,1.3)};
    \addlegendentry{$\lambda^\pm$ bound (Thm.~2.1)}

    \node[font=\scriptsize\boldmath, wavered, anchor=south,
          fill=white, inner sep=1pt] at (axis cs:-0.15933, 1.3009) {$\mathrm S$};
    \node[font=\scriptsize\boldmath, wavered, anchor=south,
          fill=white, inner sep=1pt] at (axis cs:0.00000, 1.3009) {$\mathrm C$};
    \node[font=\scriptsize\boldmath, wavered, anchor=south,
          fill=white, inner sep=1pt] at (axis cs:0.41069, 1.3009) {$\mathrm R$};
    \node[font=\scriptsize\boldmath, wavered, anchor=south,
          fill=white, inner sep=1pt] at (axis cs:0.57303, 1.3009) {$\mathrm S^\dagger$};

    \draw[->, wavered, line width=0.8pt, shorten <=3.58751pt, shorten >=3.98612pt]
      (axis cs:-0.15933, 1.3009) -- (axis cs:-0.15933, 0.79224);
    \draw[->, wavered, line width=0.8pt, shorten <=3.58751pt, shorten >=3.98612pt]
      (axis cs:0.00000, 1.3009) -- (axis cs:0.00000, 0.80000);
    \draw[->, wavered, line width=0.8pt, shorten <=3.58751pt, shorten >=3.98612pt]
      (axis cs:0.24835, 1.3009) -- (axis cs:0.24835, 0.80000);
    \draw[->, wavered, line width=0.8pt, shorten <=3.58751pt, shorten >=3.98612pt]
      (axis cs:0.57303, 1.3009) -- (axis cs:0.57303, 1.20000);

    \draw[gray, opacity=0.5, line width=0.9pt, yshift=-3.58751pt]
      (axis cs:0.24835, 1.3009) -- (axis cs:0.57303, 1.3009);
    \node[font=\tiny, teal, anchor=west]
      at (axis cs:-0.44, 1.22) {$\lambda_1^-$ bound $=-1.591$};
  \end{axis}

  \begin{axis}[
      name=sound, at={(density.below south west)}, anchor=north west,
      yshift=-8mm,
      width=150mm, height=58mm,
      xmin=0.45, xmax=1.32, ymin=0, ymax=3.6,
      xlabel={$\rho$}, ylabel={sound speed},
      axis lines=left, tick align=outside,
      grid=major, grid style={gray!20},
      legend cell align=left, legend columns=4, legend pos=south east,
      legend style={draw=gray!50, fill=white, font=\scriptsize,
                    /tikz/every even column/.append style={column sep=6pt}},
    ]

    \addplot[draw=none, fill=gray!15, forget plot] coordinates {
        (0.80000,0) (0.94893,0) (0.94893,3.6) (0.80000,3.6)
    } \closedcycle;

    \addplot[gray, densely dotted, line width=0.6pt, no marks, forget plot]
      coordinates {(1.0, 0) (1.0, 3.6)};
    \node[gray!60!black, font=\scriptsize, anchor=south west]
      at (axis cs:1.015, 0.011) {$\rho_0$};

    \addplot[black, line width=0.9pt, no marks] coordinates {
        (0.45000, 0.928751)
        (0.51863, 0.967743)
        (0.56960, 1.000432)
    };
    \addplot[black, line width=0.9pt, no marks, forget plot] coordinates {
        (0.72264, 1.176305)
        (0.74214, 1.215266)
        (0.76008, 1.258269)
        (0.77674, 1.306018)
        (0.79224, 1.358859)
    };
    \addplot[black, line width=0.9pt, no marks, forget plot] coordinates {
        (0.80000, 1.379702)
        (0.81166, 1.430513)
        (0.82278, 1.485739)
        (0.83354, 1.546421)
        (0.84374, 1.611546)
        (0.85358, 1.682165)
        (0.86306, 1.758235)
        (0.87236, 1.841269)
        (0.88129, 1.929523)
        (0.88785, 1.999937)
        (0.89441, 2.075216)
        (0.90097, 2.155396)
        (0.90772, 2.242829)
        (0.91519, 2.345359)
        (0.92321, 2.461253)
        (0.93104, 2.579109)
        (0.94963, 2.864414)
        (0.95948, 3.007653)
        (0.96494, 3.080896)
        (0.97005, 3.143391)
        (0.97497, 3.197028)
        (0.97953, 3.239888)
        (0.98463, 3.279074)
        (0.98955, 3.307134)
        (0.99429, 3.324471)
        (0.99903, 3.331894)
        (1.00377, 3.329233)
        (1.00869, 3.315879)
        (1.01361, 3.292093)
        (1.01853, 3.258491)
    };
    \addplot[black, line width=0.9pt, no marks, forget plot] coordinates {
        (1.13040, 1.841900)
        (1.13878, 1.772128)
        (1.14748, 1.706922)
        (1.15649, 1.646447)
        (1.16566, 1.591643)
        (1.17531, 1.540584)
        (1.18527, 1.494149)
        (1.19571, 1.451553)
        (1.20662, 1.412803)
        (1.21816, 1.377378)
        (1.23018, 1.345720)
        (1.24299, 1.316970)
        (1.25643, 1.291471)
        (1.27066, 1.268830)
        (1.28600, 1.248586)
        (1.30245, 1.230809)
        (1.32000, 1.215472)
    };
    \addlegendentry{$a=\sqrt{a_{\mathrm h}^2+a_{\mathrm b}^2}$}

    \addplot[blue, densely dashed, line width=0.7pt, no marks] coordinates {
        (0.45000, 0.897404)
        (0.56960, 0.940718)
    };
    \addplot[blue, densely dashed, line width=0.7pt, no marks, forget plot] coordinates {
        (0.72264, 0.996763)
        (0.79224, 1.015264)
    };
    \addplot[blue, densely dashed, line width=0.7pt, no marks, forget plot] coordinates {
        (0.80000, 1.004564)
        (1.01853, 1.054276)
    };
    \addplot[blue, densely dashed, line width=0.7pt, no marks, forget plot] coordinates {
        (1.13040, 1.067293)
        (1.32000, 1.100910)
    };
    \addlegendentry{$a_{\mathrm h}$}

    \addplot[red, densely dotted, line width=0.9pt, no marks] coordinates {
        (0.45000, 0.239256)
        (0.48306, 0.261742)
        (0.51380, 0.285945)
        (0.54251, 0.312051)
        (0.56948, 0.340325)
        (0.59471, 0.370768)
        (0.61849, 0.403736)
        (0.64082, 0.439253)
        (0.66170, 0.477268)
        (0.68142, 0.518279)
        (0.70027, 0.563025)
        (0.71796, 0.610890)
        (0.73449, 0.661715)
        (0.75015, 0.716281)
        (0.76523, 0.775764)
        (0.77944, 0.839107)
        (0.79278, 0.906069)
        (0.80380, 0.967678)
        (0.81424, 1.032002)
        (0.82439, 1.100765)
        (0.83396, 1.171859)
        (0.84324, 1.247213)
        (0.85252, 1.329501)
        (0.86151, 1.416419)
        (0.87021, 1.507812)
        (0.87920, 1.610260)
        (0.88790, 1.717525)
        (0.89660, 1.832969)
        (0.90530, 1.956511)
        (0.91313, 2.074240)
        (0.92125, 2.202076)
        (0.92908, 2.329638)
        (0.94561, 2.603692)
        (0.95344, 2.729544)
        (0.96127, 2.847582)
        (0.96794, 2.938543)
        (0.97229, 2.991599)
        (0.97635, 3.035799)
        (0.98041, 3.074212)
        (0.98418, 3.104191)
        (0.98795, 3.128272)
        (0.99172, 3.146115)
        (0.99549, 3.157462)
        (0.99926, 3.162148)
        (1.00274, 3.160498)
        (1.00651, 3.152263)
        (1.00999, 3.138813)
        (1.01376, 3.118104)
        (1.01753, 3.091301)
        (1.02130, 3.058777)
        (1.02536, 3.017852)
        (1.02942, 2.971403)
        (1.03290, 2.927685)
        (1.03638, 2.880808)
        (1.04392, 2.770435)
        (1.05233, 2.637322)
        (1.07350, 2.287241)
        (1.08394, 2.119750)
        (1.08974, 2.030407)
        (1.09525, 1.948468)
        (1.10076, 1.869614)
        (1.10627, 1.793980)
        (1.11265, 1.710503)
        (1.11903, 1.631427)
        (1.12541, 1.556681)
        (1.13208, 1.483034)
        (1.13875, 1.413801)
        (1.14571, 1.346035)
        (1.15267, 1.282588)
        (1.15963, 1.223197)
        (1.16746, 1.160901)
        (1.17529, 1.103041)
        (1.18370, 1.045439)
        (1.19211, 0.992149)
        (1.20081, 0.941162)
        (1.20980, 0.892499)
        (1.21908, 0.846158)
        (1.22865, 0.802111)
        (1.23851, 0.760315)
        (1.24895, 0.719629)
        (1.25968, 0.681244)
        (1.27099, 0.644149)
        (1.29448, 0.576591)
        (1.32000, 0.515140)
    };
    \addlegendentry{$a_{\mathrm b}$}

    \addplot[only marks, mark=*, mark size=1.4pt, black] coordinates {
        (0.50000, 0.956781)
        (0.79224, 1.358859)
        (0.80000, 1.379702)
        (0.94893, 2.853759)
        (1.20000, 1.435651)
    };
    \addlegendentry{states $W^{\mathrm{L/R},(k)}$}

    \node[font=\scriptsize, anchor=south] at (axis cs:0.50000, 1.03439) {$W^{\mathrm L,(0)}$};
    \node[font=\scriptsize, anchor=south, xshift=-11.95836pt]
      (state1) at (axis cs:0.79224, 1.54065) {$W^{\mathrm L,(1)}$};
    \draw[gray, line width=0.4pt] (state1) -- (axis cs:0.79224, 1.35886);
    \node[font=\scriptsize, anchor=south, xshift=19.93060pt]
      (state2) at (axis cs:0.80000, 1.88781) {$W^{\mathrm R,(2)}$};
    \draw[gray, line width=0.4pt] (state2) -- (axis cs:0.80000, 1.37970);
    \node[font=\scriptsize, anchor=south] at (axis cs:0.94893, 3.29721) {$W^{\mathrm R,(1)}$};
    \node[font=\scriptsize, anchor=south] at (axis cs:1.20000, 1.62943) {$W^{\mathrm R,(0)}$};

    \node[font=\tiny, blue!70!black, anchor=north, yshift=3.40282pt]
      at (axis cs:0.87446, 0.86026) {right rarefaction};
  \end{axis}

\end{tikzpicture}

%% file: C12.tex
\definecolor{wavered}{rgb}{0.55,0.05,0.05}
\begin{tikzpicture}

  \begin{axis}[
      name=density,
      width=150mm, height=88mm,
      xmin=-0.96, xmax=0.72, ymin=0.9, ymax=1.99,
      xlabel={$x$}, ylabel={$\rho$},
      axis lines=left, tick align=outside,
      grid=major, grid style={gray!20},
      legend cell align=left, legend columns=2, legend pos=north west,
      legend style={draw=gray!50, fill=white, font=\scriptsize,
                    /tikz/every even column/.append style={column sep=6pt}},
    ]

    \addplot[red!55, line width=2.0pt, no marks] coordinates {
        (-0.96000, 0.950000)
        (-0.56251, 0.950000)
        (-0.56251, 1.036750)
        (-0.54932, 1.042898)
        (-0.53352, 1.049631)
        (-0.47232, 1.073930)
        (-0.44745, 1.084177)
        (-0.42490, 1.094130)
        (-0.40543, 1.103499)
        (-0.39418, 1.109354)
        (-0.38314, 1.115502)
        (-0.37288, 1.121649)
        (-0.36337, 1.127797)
        (-0.35461, 1.133945)
        (-0.34617, 1.140386)
        (-0.33846, 1.146827)
        (-0.33113, 1.153560)
        (0.00000, 1.153560)
        (0.00000, 1.201750)
        (0.25623, 1.201750)
        (0.25623, 1.400000)
        (0.25733, 1.411278)
        (0.25849, 1.422556)
        (0.25980, 1.434586)
        (0.26118, 1.446617)
        (0.26270, 1.459398)
        (0.26446, 1.473684)
        (0.26627, 1.487970)
        (0.26833, 1.503759)
        (0.27268, 1.536090)
        (0.27788, 1.573684)
        (0.28371, 1.615038)
        (0.29578, 1.700000)
        (0.72000, 1.700000)
    };
    \addlegendentry{exact}

    \addplot[black, line width=0.7pt, no marks] coordinates {
        (-0.95996, 0.950000)
        (-0.56305, 0.950014)
        (-0.56287, 0.950403)
        (-0.56281, 0.951231)
        (-0.56274, 0.953679)
        (-0.56268, 0.960669)
        (-0.56262, 0.978598)
        (-0.56256, 1.012388)
        (-0.56250, 1.030309)
        (-0.56244, 1.032520)
        (-0.56219, 1.033539)
        (-0.56213, 1.034277)
        (-0.56189, 1.035695)
        (-0.56085, 1.037033)
        (-0.55524, 1.040175)
        (-0.53516, 1.048941)
        (-0.44757, 1.084086)
        (-0.42603, 1.093562)
        (-0.40552, 1.103401)
        (-0.38330, 1.115351)
        (-0.36182, 1.128794)
        (-0.34540, 1.140961)
        (-0.33789, 1.147287)
        (-0.33112, 1.153589)
        (-0.00885, 1.153556)
        (-0.00031, 1.153156)
        (-0.00012, 1.153706)
        (-0.00006, 1.156062)
        (0.00000, 1.163566)
        (0.00012, 1.185756)
        (0.00018, 1.193321)
        (0.00024, 1.197936)
        (0.00037, 1.200359)
        (0.00043, 1.200789)
        (0.00055, 1.201344)
        (0.00092, 1.201624)
        (0.25555, 1.201856)
        (0.25574, 1.202461)
        (0.25592, 1.204749)
        (0.25604, 1.216434)
        (0.25610, 1.219697)
        (0.25616, 1.252956)
        (0.25623, 1.304463)
        (0.25629, 1.347357)
        (0.25635, 1.375647)
        (0.25641, 1.392758)
        (0.25647, 1.402719)
        (0.25653, 1.408390)
        (0.25659, 1.411596)
        (0.25665, 1.413386)
        (0.25671, 1.414383)
        (0.25690, 1.415365)
        (0.25757, 1.415926)
        (0.25806, 1.419045)
        (0.26196, 1.452909)
        (0.26526, 1.479685)
        (0.26996, 1.515685)
        (0.27527, 1.554461)
        (0.28296, 1.609148)
        (0.29333, 1.681672)
        (0.29559, 1.696882)
        (0.29602, 1.699130)
        (0.29633, 1.699924)
        (0.71997, 1.700000)
    };
    \addlegendentry{ryujin}

    \addplot[black, dotted, line width=0.8pt, no marks, forget plot]
      coordinates {(-0.56251,0.9) (-0.56251,1.78)};
    \addplot[black, dotted, line width=0.8pt, no marks]
      coordinates {(0.29578,0.9) (0.29578,1.78)};
    \addlegendentry{$\lambda^\pm$ exact}
    \addplot[orange!85!black, dashdotted, line width=0.8pt, no marks, forget plot]
      coordinates {(-0.14144,0.9) (-0.14144,1.78)};
    \addplot[orange!85!black, dashdotted, line width=0.8pt, no marks]
      coordinates {(0.29316,0.9) (0.29316,1.78)};
    \addlegendentry{$\lambda^\pm$ hydrodynamical}
    \addplot[teal, dashed, line width=0.8pt, no marks, forget plot]
      coordinates {(-0.91609,0.9) (-0.91609,1.78)};
    \addplot[teal, dashed, line width=0.8pt, no marks]
      coordinates {(0.65461,0.9) (0.65461,1.78)};
    \addlegendentry{$\lambda^\pm$ bound (Thm.~2.1)}

    \node[font=\scriptsize\boldmath, wavered, anchor=south,
          fill=white, inner sep=1pt] at (axis cs:-0.56251, 1.3) {$\mathrm S$};
    \node[font=\scriptsize\boldmath, wavered, anchor=south,
          fill=white, inner sep=1pt] at (axis cs:-0.44682, 1.3) {$\mathrm R^\dagger$};
    \node[font=\scriptsize\boldmath, wavered, anchor=south,
          fill=white, inner sep=1pt] at (axis cs:0.00000, 1.3) {$\mathrm C$};
    \node[font=\scriptsize\boldmath, wavered, anchor=south,
          fill=white, inner sep=1pt] at (axis cs:0.25623, 1.8) {$\mathrm S^\dagger$};
    \node[font=\scriptsize\boldmath, wavered, anchor=south,
          fill=white, inner sep=1pt, xshift=11.95836pt, yshift=11.95836pt]
      (offlabel) at (axis cs:0.27601, 1.8) {$\mathrm R$};
    \draw[gray, line width=0.4pt] (offlabel) -- (axis cs:0.27601, 1.8);

    \draw[->, wavered, line width=0.8pt, shorten <=3.58751pt, shorten >=3.98612pt]
      (axis cs:-0.56251, 1.3) -- (axis cs:-0.56251, 1.03675);
    \draw[->, wavered, line width=0.8pt, shorten <=3.58751pt, shorten >=3.98612pt]
      (axis cs:-0.33113, 1.3) -- (axis cs:-0.33113, 1.15356);
    \draw[->, wavered, line width=0.8pt, shorten <=3.58751pt, shorten >=3.98612pt]
      (axis cs:0.00000, 1.3) -- (axis cs:0.00000, 1.20175);
    \draw[->, wavered, line width=0.8pt, shorten <=3.58751pt, shorten >=3.98612pt]
      (axis cs:0.25623, 1.8) -- (axis cs:0.25623, 1.40000);
    \draw[->, wavered, line width=0.8pt, shorten <=3.58751pt, shorten >=3.98612pt]
      (axis cs:0.29578, 1.8) -- (axis cs:0.29578, 1.70000);

    \draw[gray, opacity=0.5, line width=0.9pt, yshift=-3.58751pt]
      (axis cs:-0.56251, 1.3) -- (axis cs:-0.33113, 1.3);
    \draw[gray, opacity=0.5, line width=0.9pt, yshift=-3.58751pt]
      (axis cs:0.25623, 1.8) -- (axis cs:0.29578, 1.8);

  \end{axis}

  \begin{axis}[
      name=sound, at={(density.below south west)}, anchor=north west,
      yshift=-8mm,
      width=150mm, height=58mm,
      xmin=0.88, xmax=1.75, ymin=0, ymax=3.85,
      xlabel={$\rho$}, ylabel={sound speed},
      axis lines=left, tick align=outside,
      grid=major, grid style={gray!20},
      legend cell align=left, legend columns=4, legend pos=north east,
      legend style={draw=gray!50, fill=white, font=\scriptsize,
                    /tikz/every even column/.append style={column sep=6pt}},
    ]

    \addplot[draw=none, fill=gray!15, forget plot] coordinates {
        (1.03675,0) (1.15356,0) (1.15356,3.85) (1.03675,3.85)
    } \closedcycle;
    \addplot[draw=none, fill=gray!15, forget plot] coordinates {
        (1.40000,0) (1.70000,0) (1.70000,3.85) (1.40000,3.85)
    } \closedcycle;

    \addplot[gray, densely dotted, line width=0.6pt, no marks, forget plot]
      coordinates {(1.0, 0) (1.0, 3.85)};
    \node[gray!60!black, font=\scriptsize, anchor=south west]
      at (axis cs:1.015, 0.012) {$\rho_0$};

    \addplot[black, line width=0.9pt, no marks] coordinates {
        (0.88000, 1.903728)
        (0.88709, 1.979440)
        (0.89418, 2.060818)
        (0.90127, 2.147913)
        (0.90847, 2.242190)
        (0.91580, 2.343592)
        (0.92358, 2.456897)
        (0.93125, 2.572801)
        (0.94904, 2.846716)
        (0.95822, 2.981725)
        (0.96333, 3.051762)
        (0.96798, 3.110850)
        (0.97240, 3.162000)
        (0.97670, 3.206321)
        (0.98228, 3.254544)
        (0.98774, 3.290349)
        (0.99297, 3.313045)
        (0.99564, 3.320043)
        (0.99820, 3.323753)
        (1.00331, 3.322338)
        (1.00587, 3.317227)
        (1.00854, 3.308793)
        (1.01121, 3.297263)
        (1.01389, 3.282729)
        (1.01935, 3.244167)
    };
    \addplot[black, line width=0.9pt, no marks, forget plot] coordinates {
        (1.03675, 3.055761)
        (1.04196, 2.985207)
        (1.04773, 2.902286)
        (1.07217, 2.532672)
        (1.08405, 2.360795)
        (1.09062, 2.271371)
        (1.09695, 2.189399)
        (1.10318, 2.113320)
        (1.10940, 2.041736)
        (1.11676, 1.962929)
        (1.12411, 1.890265)
        (1.13170, 1.821543)
        (1.13939, 1.757864)
        (1.14731, 1.698285)
        (1.15546, 1.642849)
        (1.16383, 1.591533)
        (1.17244, 1.544256)
    };
    \addplot[black, line width=0.9pt, no marks, forget plot] coordinates {
        (1.18288, 1.449761)
        (1.19154, 1.411265)
        (1.20037, 1.376162)
        (1.20969, 1.343116)
        (1.21936, 1.312776)
        (1.22935, 1.285058)
        (1.23984, 1.259477)
        (1.25084, 1.236050)
        (1.26250, 1.214469)
        (1.28748, 1.177468)
        (1.31530, 1.147599)
        (1.34679, 1.123959)
        (1.38260, 1.106059)
    };
    \addplot[black, line width=0.9pt, no marks, forget plot] coordinates {
        (1.40000, 1.099844)
        (1.45867, 1.086882)
        (1.53136, 1.081307)
        (1.62419, 1.082539)
        (1.75000, 1.090729)
    };
    \addlegendentry{$a=\sqrt{a_{\mathrm h}^2+a_{\mathrm b}^2}$}

    \addplot[blue, densely dashed, line width=0.7pt, no marks] coordinates {
        (0.88000, 1.000237)
        (1.01935, 1.030081)
    };
    \addplot[blue, densely dashed, line width=0.7pt, no marks, forget plot] coordinates {
        (1.03675, 1.033573)
        (1.17244, 1.059313)
    };
    \addplot[blue, densely dashed, line width=0.7pt, no marks, forget plot] coordinates {
        (1.18288, 0.998716)
        (1.38260, 1.030371)
    };
    \addplot[blue, densely dashed, line width=0.7pt, no marks, forget plot] coordinates {
        (1.40000, 1.032951)
        (1.75000, 1.080095)
    };
    \addlegendentry{$a_{\mathrm h}$}

    \addplot[red, densely dotted, line width=0.9pt, no marks] coordinates {
        (0.88000, 1.619786)
        (0.88725, 1.709230)
        (0.89421, 1.800439)
        (0.90117, 1.896873)
        (0.90842, 2.002713)
        (0.91538, 2.109122)
        (0.92263, 2.224297)
        (0.94554, 2.602545)
        (0.95366, 2.732987)
        (0.96207, 2.859025)
        (0.96584, 2.911055)
        (0.96932, 2.955966)
        (0.97396, 3.010441)
        (0.97860, 3.057840)
        (0.98295, 3.095040)
        (0.98701, 3.122839)
        (0.99136, 3.144688)
        (0.99542, 3.157312)
        (0.99948, 3.162214)
        (1.00354, 3.159309)
        (1.00702, 3.150640)
        (1.01050, 3.136379)
        (1.01398, 3.116704)
        (1.01746, 3.091852)
        (1.02094, 3.062118)
        (1.02471, 3.024792)
        (1.02848, 2.982616)
        (1.03254, 2.932361)
        (1.03921, 2.840654)
        (1.04646, 2.731110)
        (1.05400, 2.610081)
        (1.07169, 2.316953)
        (1.08010, 2.180466)
        (1.08909, 2.040268)
        (1.09779, 1.911725)
        (1.10765, 1.775550)
        (1.11751, 1.649870)
        (1.12737, 1.534570)
        (1.13752, 1.426245)
        (1.14854, 1.319731)
        (1.15985, 1.221383)
        (1.17174, 1.128747)
        (1.18392, 1.043991)
        (1.19668, 0.964866)
        (1.21031, 0.889854)
        (1.22452, 0.820677)
        (1.23960, 0.755902)
        (1.25729, 0.689512)
        (1.27643, 0.627428)
        (1.29673, 0.570714)
        (1.31819, 0.519147)
        (1.34139, 0.471314)
        (1.36604, 0.427820)
        (1.39272, 0.387626)
        (1.42114, 0.351151)
        (1.45188, 0.317609)
        (1.48523, 0.286768)
        (1.52090, 0.258869)
        (1.55976, 0.233219)
        (1.60152, 0.210012)
        (1.64705, 0.188758)
        (1.69635, 0.169477)
        (1.75000, 0.151939)
    };
    \addlegendentry{$a_{\mathrm b}$}

    \addplot[only marks, mark=*, mark size=1.4pt, black] coordinates {
        (0.95000, 2.861293)
        (1.03675, 3.055951)
        (1.15356, 1.655639)
        (1.20175, 1.371006)
        (1.40000, 1.099754)
        (1.70000, 1.086842)
    };
    \addlegendentry{states $W^{\mathrm{L/R},(k)}$}

    \node[font=\scriptsize, anchor=south] at (axis cs:0.95000, 3.29263) {$W^{\mathrm L,(0)}$};
    \node[font=\scriptsize, anchor=south] at (axis cs:1.03675, 3.37351) {$W^{\mathrm L,(1)}$};
    \node[font=\scriptsize, anchor=south] at (axis cs:1.15356, 1.95153) {$W^{\mathrm L,(2)}$};
    \node[font=\scriptsize, anchor=south] at (axis cs:1.20175, 1.60934) {$W^{\mathrm R,(2)}$};
    \node[font=\scriptsize, anchor=south] at (axis cs:1.40000, 1.17162) {$W^{\mathrm R,(1)}$};
    \node[font=\scriptsize, anchor=south] at (axis cs:1.70000, 1.14917) {$W^{\mathrm R,(0)}$};

    \node[font=\tiny, blue!70!black, anchor=north, yshift=3.40282pt]
      at (axis cs:1.09516, 0.92000) {left rarefaction};
    \node[font=\tiny, blue!70!black, anchor=north, yshift=3.40282pt]
      at (axis cs:1.55000, 0.92000) {right rarefaction};
  \end{axis}

\end{tikzpicture}